\documentclass[10pt,reqno,oneside]{amsproc} \usepackage{amsfonts} \usepackage{dsfont} \usepackage{amsmath, amsthm, amssymb, enumerate} \textwidth 16truecm \textheight 8in\oddsidemargin0.2truecm\evensidemargin0.7truecm\voffset-.1truecm \usepackage{color}   \usepackage{marginnote} \usepackage[colorlinks=true, pdfstartview=FitV, linkcolor=black, citecolor=black, urlcolor=black]{hyperref}   \chardef\forshowkeys=0   \chardef\refcheck=0   \chardef\showllabel=0   \chardef\sketches=0   \chardef\showcolors=0 \ifnum\forshowkeys=1      \usepackage[notref,notcite,color]{showkeys} \fi \ifnum\showllabel=1   \def\llabel#1{\marginnote{\color{colorcccc}\rm\small(#1)}[-0.0cm]\notag}   \def\llabel{\label} \else  \def\llabel#1{\notag} \fi  \ifnum\refcheck=1   \usepackage{refcheck} \fi   \newtheorem{theorem}{Theorem}[section] \newtheorem{Theorem}{Theorem}[section]     \newtheorem{Lemma}[theorem]{Lemma}  \theoremstyle{definition}  \newtheorem{Remark}[theorem]{Remark}  \def\DD{\mathcal{D}} \def\MM{\tilde M}        \def\uk{u^{(k)}} \def\ukm{u^{(k-1)}}   \def\umo{u^{(-1)}} \def\vk{v^{(k)}} \def\vj{v^{(j)}}    \def\ujm{u^{(j-1)}}  \def\vj{v^{(j)}}   \def\Vj{V^{(j)}} \def\Vjm{V^{(j-1)}} \def\Vjp{V^{(j+1)}} \def\vz{v^{(0)}} \def\zj{z^{(j)}} \def\zjm{z^{(j-1)}}  \def\tae{\indeq\text{a.e.}~}                       \def\LLL#1#2{L_{\omega}^{#1}L_{x}^{#2}}    \def\LLLL#1#2#3{L_{\omega}^{#1}L_{t}^{#2}L_{x}^{#3}}    \def\LLLot{\LLL{1}{3}}    \def\LLLit{\LLL{\infty}{3}}    \def\LLLis{\LLL{\infty}{6}}        \def\LLLii{\LLL{\infty}{\infty}}    \def\Pas{\indeq\mathbb{P}\text{-a.s.}}         
       \def\PP{{\mathbb P}}        \def\dd{d}    \def\ueps{u^{\epsilon}}                                                                                                                               \def\uu{{{u}}}            \def\startnewsection#1#2{\section{#1}\label{#2}\setcounter{equation}{0}}       \def\NNp{{\mathbb N}}    \def\NNz{{\mathbb N}_0}         \def\TT{{\mathbb T}}    \def\WW{{\mathbb W}}    \def\EE{{\mathbb E}}    \def\comma{ {\rm ,\qquad{}} }                \def\commaone{ {\rm ,\quad{}} }             \def\fractext#1#2{{#1}/{#2}}                                                   \def\indeq{\qquad{}}                                        \def\TT{\mathbb T}                 \def\tilde{\widetilde}     \def\PP{\mathbb{P}}   
                 \def\div{\mathop{\rm div}\nolimits}      \def\indeq{\quad{}}                                     \ifnum\showcolors=1   \def\cole{\color{coloroftheorems}}   \definecolor{colorcccc}{rgb}{0.7,0.7,0.7}                  \def\colr{\color{red}}      \def\colb{\color{black}}   \definecolor{colorpppp}{rgb}{0.6,0.0,0.1}   \definecolor{colorgggg}{rgb}{.0,0.4,0.0}   \definecolor{colorhhhh}{rgb}{0,0.6,0.2}   \definecolor{colorgray}{rgb}{0.8,0.8,0.8}         \definecolor{coloroftheorems}{rgb}{0.6,0.0,0.6}   \definecolor{colorigor}{rgb}{1, 0.2, 0.8}   \definecolor{amethyst}{rgb}{0.6, 0.4, 0.8}               \definecolor{colororange}{rgb}{0.8,0.2,0}   \definecolor{colorpurple}{rgb}{0.6,0.0,0.6}    \else      \def\cole{}   \definecolor{colorcccc}{rgb}{0,0,0}                  \def\colr{\color{black}}      \def\colb{\color{black}}   \definecolor{colorpppp}{rgb}{0,0,0}   \definecolor{colorgggg}{rgb}{0,0,0}   \definecolor{colorhhhh}{rgb}{0,0,0}   \definecolor{colorgray}{rgb}{0,0,0}         \definecolor{coloroftheorems}{rgb}{0,0,0}   \definecolor{colorigor}{rgb}{0,0,0}   \definecolor{amethyst}{rgb}{0,0,0}      \def\cole{\color{coloroftheorems}}             \definecolor{colororange}{rgb}{0.8,0.2,0}   \definecolor{colorpurple}{rgb}{0.6,0.0,0.6}    \fi   \def\bea{\begin{align}}   \def\ena{\end{align}}          
 \def\bega{\begin{aligned}}   \def\enda{\end{aligned}}               \def\bcase{\begin{cases}}   \def\ecase{\end{cases}}  \def\bmx{\begin{bmatrix}}   \def\emx{\end{bmatrix}}  \def\cf{\mathcal{F}}       \def\dd{d} \def\uu{{{u}}} \def\WW{ W} \def\sadklfjsdfgsdfgsdfgsdfgdsfgsdfgadfasdf{\EE}\def\sifpoierjsodfgupoefasdfgjsdfgjsdfgjsdjflxncvzxnvasdjfaopsruosihjsdfghajsdflahgfsif{\int} \def\sifpoierjsodfgupoefasdfgjsdfgjsdfgjsdjflxncvzxnvasdjfaopsruosihjsdgghajsdflahgfsif{\Vert}  \def\sifpoierjsodfgupoefasdfgjsdfgjsdfgjsdjflxncvzxnvasdjfaopsruosihjsdjghajsdflahgfsif{\rho} \def\sifpoierjsodfgupoefasdfgjsdfgjsdfgjsdjflxncvzxnvasdjfaopsruosihjsdkghajsdflahgfsif{\phi} \def\sifpoierjsodfgupoefasdfgjsdfgjsdfgjsdjflxncvzxnvasdjfaopsruosihjsdlghajsdflahgfsif{\psi}   \def\NNp{{\mathbb N}}                      
\begin{document} \baselineskip=12.6pt $\,$ \vskip1.2truecm \title[Global existence of the stochastic Navier-Stokes equations with small $L^3$ data]{Global existence of the stochastic Navier-Stokes equations in $L^3$ with small data} \author[I.~Kukavica]{Igor Kukavica} \address{Department of Mathematics, University of Southern California, Los Angeles, CA 90089} \email{kukavica@usc.edu} \author[F.H.~Xu]{Fanhui Xu} \address{Department of Mathematics, Harvard University, Cambridge, MA 02138} \email{fanhui.xu.math@gmail.com} \begin{abstract} We address the global-in-time existence and pathwise uniqueness of solutions for the stochastic incompressible Navier-Stokes equations with a multiplicative noise on the three-dimensional torus. Under natural smallness conditions on the noise, we prove the almost global existence result for small $L^{3}$ data. Namely, we show that for data sufficiently small, there exists a global-in-time strong $L^{3}$ solution in a space of probability arbitrarily close to~$1$.  \hfill  \today \end{abstract} \maketitle \date{} \startnewsection{Introduction}{sec01} In this paper, we consider the stochastic incompressible Navier-Stokes equations (SNSE)    \begin{align}    \begin{split}     &\partial_t u - \Delta u  + \mathcal{P}((u\cdot\nabla) u) = \sigma(t,u) \dot{W}(t),     \\     &\nabla\cdot u = 0   \end{split}   \label{EQ01}   \end{align} on the three-dimensional torus~$\mathbb{T}^3$, where $\mathcal{P}$ is the Leray projector.  Here, $u$ denotes the velocity field of the stochastic flow, and the stochastic term, $\sigma(t, u) \dot{W}(t)$, denotes a time-dependent infinite-dimensional multiplicative noise. We assume that the initial datum  \begin{equation}    \llabel{8ThswELzXU3X7Ebd1KdZ7v1rN3GiirRXGKWK099ovBM0FDJCvkopYNQ2aN94Z7k0UnUKamE3OjU8DFYFFokbSI2J9V9gVlM8ALWThDPnPu3EL7HPD2VDaZTggzcCCmbvc70qqPcC9mt60ogcrTiA3HEjwTK8ymKeuJMc4q6dVz200XnYUtLR9GYjPXvFOVr6W1zUK1WbPToaWJJuKnxBLnd0ftDEbMmj4loHYyhZyMjM91zQS4p7z8eKa9h0JrbacekcirexG0z4n3xz0QOWSvFj3jLhWXUIU21iIAwJtI3RbWa90I7rzAIqI3UElUJG7tLtUXzw4KQNETvXzqWaujEMenYlNIzLGxgB3AuJ86VS6RcPJ8OXWw8imtcKZEzHop84G1gSAs0PCowMI2fLKTdD60ynHg7lkNFjJLqOoQvfkfZBNG3o1DgCn9hyUh5VSP5z61qvQwceUdVJJsBvXDG4ELHQHIaPTbMTrsLsmtXGyOB7p2Os43USbq5ik4Lin769OTkUxmpI8uGYnfBKbYI9AQzCFw3h0geJftZZKU74rYleajmkmZJdiTGHOOaSt1NnlB7Y7h0yoWJryrVrTzHO82S7oubQAWx9dz2XYWBe5Kf3ALsUFvqgtM2O2IdimrjZ7RN284KGYtrVaWW4nTZXVbRVoQ77hVLX6K2kqFWFmaZnsF9Chp8KxrscSGPiStVXBJ3xZcD5IP4Fu9LcdTR2VwbcLDlGK1ro3EEyqEAzw6sKeEg2sFfjzMtrZ9kbdxEQ02}    u|_{t=0} = u_0  \end{equation} is a random function in $L^1(\Omega, L^{3}(\TT^3))$ such that $\nabla\cdot u_0=0$ and $\sifpoierjsodfgupoefasdfgjsdfgjsdfgjsdjflxncvzxnvasdjfaopsruosihjsdfghajsdflahgfsif_{\TT^3} u_0=0$. \par The well-posedness theory for the SNSE was initiated by Bensoussan and Temam in \cite{BeT}, who obtained the local existence for the SNSE with an additive noise in the Hilbert space $V$ (see \cite{CF,GZ,T} for the definition of the space).  The main device in establishing the existence theory was the It\^o isometry, which is well-adapted to the Hilbert space setting.  For the introduction to the SNSE with multiplicative noise and the associated difficulties, see~\cite{F}. Also, various notions of solutions emerged from these works and other
studies of stochastic evolution equations; see~\cite{BCF,BF,BT,BR,FRS,GV,MeS,MoS,MR2,ZBL}. Showing the existence of a probabilistically strong solution is challenging.  Many papers have focused on this issue in the context of the SNSE with multiplicative noise, starting with the works of Krylov~\cite{Kr} and Mikulevicius and Rozovski~\cite{MR1}, who established the local existence in $W^{1,p}$ spaces for the SNSE. The local and global existence of solutions was addressed by Kim, who showed in \cite{Ki} the local and almost global existence of solutions with the initial data in $H^{1/2+}$ and the almost global existence under a smallness assumption. In the 2D case, Glatt-Holtz and Ziane proved the global existence of the NSE in a smooth bounded domain with the initial data in~$V$ (see \cite{CF,GZ} for definition of the space).  For other aspects of the well-posedness theory for the stochastic fluid equations, see~\cite{CC,DZ,CFH,FS,KV}.  Using maximal regularity arguments, Agresti and Veraar established in \cite{AV} the local existence result in a range of Besov spaces; in particular, they considered $B^{-1+3/q}_{3,3}$ with an interval of $q$ which includes $q=3$, using maximal regularity arguments.  For the existence theory in the deterministic case, see~\cite{FJR,K}.  Note that for the additive noise, one may separately consider the linear equation for the linear equation with this noise. The existence for the difference of the sought-after solution and the solution of the linear equation can then be studied using deterministic methods. \par In a sequence of works, the authors of the present paper have addressed the case of $L^{p}$ data with the goal of understanding the evolution of $L^{p}$ norms. The main difficulty is the use the It\^o isometry, which is normally unavailable in the non-Hilbert setting. In the initial work~\cite{KXZ}, the authors obtained the local pathwise solutions and almost global solutions for the initial data in $L^{p}$, where $p\geq5$, with the smallness assumptions in the case of the global existence, for the torus setting. The result was improved to $p>3$ in \cite{KX,KWX}, thus almost recovering the deterministic counterpart. The main ideas of the approach were the $L^p$-type energy inequalities, inspired by~\cite{R}, and the direct convergence of the Galerkin approximates to the solution, thus also providing explicitly computable convergence rates of Galerkin approximates (with significant modifications necessary when the spatial domain is~$\mathbb{R}^{3}$~\cite{KWX}). \par However, the methods in \cite{KX} and \cite{KXZ} do not apply to the borderline case $p=3$ for the following reasons. When establishing the convergence of the approximates, the most challenging terms are the ones in which we need to show the convergence of the difference of Galerkin projectors, which are, however, only known to converge weakly.  For this purpose, we interpolate the nonlinear part between the $L^{p}$ and $L^{2}$ norms , which achieves the goal by the use of the $L^{2}$ Poincar\'e type inequality.  While this accomplishes the task of showing the convergence, a part requiring Poincar\'e inequalities does not allow us to reach the critical exponent~$p=3$.  We believe that this is a substantial difficulty, and a construction by approximating the equation and passing to the limit needs to be significantly modified. \par In this paper, we obtain the almost global existence (i.e., the global existence with a large probability) for small data in $L^{3}$ in three space dimensions. As the approximation procedure is not suitable for the reasons outlined in the previous paragraph, we introduce a new approach by splitting the data and the equation. In the first step, we separate the small initial data into an infinite sequence of parts that rapidly converge to the original in~$\mathcal{F}_0$. In the next step, we split the equation into a sequence of Navier-Stokes-like equations. These are arranged so that the (infinite) sum of their solutions converges to the solution of the original equations, as shown in the final step of the proof. We then solve these equations inductively. Note that solving each equation in the critical space presents the same difficulty as when not splitting the data. We overcome this by introducing a sequence of cutoffs and working in a higher regularity space. Specifically, we choose~$L^{6}$ to apply the results from \cite{KXZ}. The cutoffs allow the solution to be defined for all time and be globally bounded in $L^{6}$, as shown in Lemma~\ref{L07} below. However, we need to prove that the solution obeys the Navier-Stokes-like system, i.e., that the $L^{3}$ and $L^{6}$ thresholds are reached with a low probability. Note the $L^{6}$ norm of the data and the solution can actually be very large. But we show that the $L^{3}$ control allows us to decouple the nonlinearity, ensuring the high $L^{6}$ threshold is only reached on an insignificant portion of the probability space. In the last step, we prove that the sum of the sequence of solutions converges to the solution of the original SNSE. Finally, we point out that our results also apply in to other space dimensions $d\geq4$ with the Lebesgue exponent~$d$; for clarity, we provide the details only for the case~$d=3$. \par The paper is organized as follows. In Section~\ref{sec02}, we introduce the notation and state the main result on the global existence with large probability for small initial data.  Section~\ref{sec03} explains the decomposition of initial data (Lemma~\ref{L05}) and introduces the sequence of Navier-Stokes-like equations (see \eqref{EQ32} and~\eqref{EQ34}). Lemma~\ref{L06} contains the $L^{6}$ global-in-time existence result for the $k$-th iterate. The lemma which follows provides the $L^{3}$ and $L^{6}$ controls of solutions to the approximate equations.  The next lemma shows the boundedness in $L^{3}$ provided by the cutoffs.  The rest of the section establishes the finiteness and positivity of the associated stopping times. \par \startnewsection{Preliminaries and auxiliary results}{sec02} \subsection{Notation and preliminaries}\colb \par For a function $u$ on $\mathcal{D}=\mathbb{T}^{3}$ satisfying suitable integrability assumptions, we denote by $\mathcal{P}$ the Leray projector \begin{equation} ( \mathcal{P} \uu)_j( x)=\sum_{k=1}^{3}( \delta_{jk}+R_j R_k) \uu_k( x)\comma j=1,2,3, \llabel{Nw66cxftlzDGZhxQAWQKkSXjqmmrEpNuG6Pyloq8hHlSfMaLXm5RzEXW4Y1Bqib3UOhYw95h6f6o8kw6frZwg6fIyXPnae1TQJMt2TTfWWfjJrXilpYGrUlQ4uM7Dsp0rVg3gIEmQOzTFh9LAKO8csQu6mh25r8WqRIDZWgSYkWDulL8GptZW10GdSYFUXLzyQZhVZMn9amP9aEWzkau06dZghMym3RjfdePGln8s7xHYCIV9HwKa6vEjH5J8Ipr7NkCxWR84TWnqs0fsiPqGgsId1fs53AT71qRIczPX77Si23GirL9MQZ4FpigdruNYth1K4MZilvrRk6B4W5B8Id3Xq9nhxEN4P6ipZla2UQQx8mdag7rVD3zdDrhBvkLDJotKyV5IrmyJR5etxS1cvEsYxGzj2TrfSRmyZo4Lm5DmqNiZdacgGQ0KRwQKGXg9o8v8wmBfUutCOcKczzkx4UfhuAa8pYzWVq9Sp6CmAcZLMxceBXDwugsjWuiiGlvJDb08hBOVC1pni64TTqOpzezqZBJy5oKS8BhHsdnKkHgnZlUCm7j0IvYjQE7JN9fdEDddys3y1x52pbiGLca71jG3euliCeuzv2R40Q50JZUBuKdU3mMay0uoS7ulWDh7qG2FKw2TJXzBES2JkQ4UDy4aJ2IXs4RNH41spyTGNhhk0w5ZC8B3nUBp9p8eLKh8UO4fMqY6wlcAGMxCHtvlOxMqAJoQQU1e8a2aX9Y62rlIS6dejKY3KCUm257oClVeEEQ07} \end{equation}  where $\delta_{jk}=1$ if $j=k$, or else $\delta_{jk}=0$, and $R_j$, for $j=1,2,3$, represent the Riesz transforms. The projector $\mathcal{P}$ was used in \eqref{EQ01}, to eliminate the pressure gradient. We impose that $\div (\sigma(t, u))=0$ for all $t\geq 0$ if $\div u=0$, as otherwise the non-zero orthogonal component to the divergence-free part can be grouped with the pressure gradient and eliminated by~$\mathcal{P}$. We also assume that $\sigma$ has zero mean over~$\mathbb{T}^{3}$. We shall restrict our attention to the SNSE driven by a cylindrical multiplicative noise satisfying the assumptions, \begin{align}   & \sifpoierjsodfgupoefasdfgjsdfgjsdfgjsdjflxncvzxnvasdjfaopsruosihjsdgghajsdflahgfsif\sigma(t, u)\sifpoierjsodfgupoefasdfgjsdfgjsdfgjsdjflxncvzxnvasdjfaopsruosihjsdgghajsdflahgfsif_{\mathbb{L}^p} \le  C(\sifpoierjsodfgupoefasdfgjsdfgjsdfgjsdjflxncvzxnvasdjfaopsruosihjsdgghajsdflahgfsif u\sifpoierjsodfgupoefasdfgjsdfgjsdfgjsdjflxncvzxnvasdjfaopsruosihjsdgghajsdflahgfsif_{p}+1) \quad \mbox{ for all } t\comma p=3,6, \label{EQ03}  \\   & \sifpoierjsodfgupoefasdfgjsdfgjsdfgjsdjflxncvzxnvasdjfaopsruosihjsdgghajsdflahgfsif\sigma(t, u_1)-\sigma(t, u_2)\sifpoierjsodfgupoefasdfgjsdfgjsdfgjsdjflxncvzxnvasdjfaopsruosihjsdgghajsdflahgfsif_{\mathbb{L}^p} \le C \epsilon_{\sigma}\sifpoierjsodfgupoefasdfgjsdfgjsdfgjsdjflxncvzxnvasdjfaopsruosihjsdgghajsdflahgfsif  u_1-u_2\sifpoierjsodfgupoefasdfgjsdfgjsdfgjsdjflxncvzxnvasdjfaopsruosihjsdgghajsdflahgfsif_{p}\quad \mbox{ for all } t\mbox{ and } u_1, u_2\in L^p(\TT^3) \comma p=3,6,\label{EQ30}   \\&   \sigma(t,0)=0\quad \mbox{ for all } t    \label{EQ04}   , \end{align} where $\epsilon_{\sigma}\in(0,1]$ is a small parameter and the $\mathbb{L}^p$-norm is defined in \eqref{EQ11} below by setting $s=0$. This setup includes the linear-structured noise. For instance,    \begin{equation}      \sigma(t, u) = \frac{1}{C}\epsilon_{\sigma}\mathcal{P}_0(\sifpoierjsodfgupoefasdfgjsdfgjsdfgjsdjflxncvzxnvasdjfaopsruosihjsdkghajsdflahgfsif*P(t,u))   \llabel{e8p1zUJSvbmLdFy7ObQFNlJ6FRdFkEmqMN0FdNZJ08DYuq2pLXJNz4rOZkZX2IjTD1fVtz4BmFIPi0GKDR2WPhOzHzTLPlbAEOT9XW0gbTLb3XRQqGG8o4TPE6WRcuMqMXhs6xOfv8stjDiu8rtJtTKSKjlGkGwt8nFDxjA9fCmiuFqMWjeox5Akw3wSd81vK8c4C0OdjCHIseHUOhyqGx3KwOlDql1Y4NY4IvI7XDE4cFeXdFVbCFHaJsb4OC0huMj65J4favgGo7qY5XtLyizYDvHTRzd9xSRVg0Pl6Z89XzfLhGlHIYBx9OELo5loZx4wag4cnFaCEKfA0uzfwHMUVM9QyeARFe3Py6kQGGFxrPf6TZBQRla1a6AekerXgkblznSmmhYjcz3ioWYjzh33sxRJMkDosEAAhUOOzaQfKZ0cn5kqYPnW71vCT69aEC9LDEQ5SBK4JfVFLAoQpNdzZHAlJaLMnvRqH7pBBqOr7fvoaeBSA8TEbtxy3jwK3v244dlfwRLDcgX14vTpWd8zyYWjweQmFyD5y5lDNlZbAJaccldkxYn3VQYIVv6fwmHz19w3yD4YezRM9BduEL7D92wTHHcDogZxZWRWJxipvfz48ZVB7FZtgK0Y1woCohLAi70NOTa06u2sYGlmspVl2xy0XB37x43k5kaoZdeyEsDglRFXi96b6w9BdIdKogSUMNLLbCRzeQLUZmi9O2qvVzDhzv1r6spSljwNhG6s6iSdXhobhbp2usEdl95LPAEQ06} \end{equation} fulfills all above assumptions if $\sifpoierjsodfgupoefasdfgjsdfgjsdfgjsdjflxncvzxnvasdjfaopsruosihjsdkghajsdflahgfsif$ is a function in $C^\infty(\TT^3)$, $P(t, \cdot)$ is a time-dependent operator satisfying \eqref{EQ03} and~\eqref{EQ04}, and $C$ is sufficiently large\colb; the symbol $\mathcal{P}_0$ denotes the Leray projection composed with the mean-zero projection. \colb \par \colb Let $(\Omega, \mathcal{F},(\mathcal{F}_t)_{t\geq 0},\mathbb{P})$ be a given stochastic basis satisfying the standard assumptions and $\mathcal{H}, \mathcal{Y}$ separable Hilbert spaces. Suppose that $\{W_k: k\in\NNp\}$ is a family of independent Brownian motions in $(\Omega, \mathcal{F},(\mathcal{F}_t)_{t\geq 0},\mathbb{P})$ and $\{\mathbf{e}_k\}_{k\geq 1}$ a complete orthonormal basis of~$\mathcal{H}$. Then, $\WW( t,\omega):=\sum_{k\geq 1} W_k( t,\omega) \mathbf{e}_k$ defines a cylindrical Wiener process over~$\mathcal{H}$. We use $l^2( \mathcal{H},\mathcal{Y})$ for the set of Hilbert-Schmidt operators, which is a subset of bounded linear operators from $\mathcal{H}$ to $\mathcal{Y}$ satisfying  \begin{equation} \sifpoierjsodfgupoefasdfgjsdfgjsdfgjsdjflxncvzxnvasdjfaopsruosihjsdgghajsdflahgfsif G\sifpoierjsodfgupoefasdfgjsdfgjsdfgjsdjflxncvzxnvasdjfaopsruosihjsdgghajsdflahgfsif_{l^2( \mathcal{H},\mathcal{Y})}^2:= \sum_{k=1}^{\dim \mathcal{H}} | G \mathbf{e}_k|_{\mathcal{Y}}^2<\infty. \llabel{trBBibPCwShpFCCUayzxYS578rof3UwDPsCIpESHB1qFPSW5tt0I7ozjXun6cz4cQLBJ4MNmI6F08S2Il8C0JQYiUlI1YkKoiubVtfGuOegSllvb4HGn3bSZLlXefaeN6v1B6m3Ek3JSXUIjX8PdNKIUFNJvPHaVr4TeARPdXEV7BxM0A7w7jep8M4QahOihEVoPxbi1VuGetOtHbPtsO5r363Rez9nA5EJ55pcLlQQHg6X1JEWK8Cf9kZm14A5lirN7kKZrY0K10IteJd3kMGwopVnfYEG2orGfj0TTAXtecJKeTM0x1N9f0lRpQkPM373r0iA6EFs1F6f4mjOB5zu5GGTNclBmkb5jOOK4ynyMy04oz6m6AkzNnPJXhBnPHRuN5LyqSguz5NnW2lUYx3fX4huLieHL30wg93Xwcgj1I9dO9bEPCR0vc6A005QVFy1lyK7oVRVpbJzZnxYdcldXgQaDXY3gzx368ORJFK9UhXTe3xYbVHGoYqdHgVyf5kKQzmmK49xxiApjVkwgzJOdE4vghAv9bVIHewcVqcbSUcF1pHzolNjTl1BurcSamIPzkUS8wwSa7wVWR4DLVGf1RFr599HtyGqhDT0TDlooamgj9ampngaWenGXU2TzXLhIYOW5v2dArCGsLks53pWAuAyDQlF6spKydHT9Z1Xn2sU1g0DLlaoYuLPPB6YKoD1M0fiqHUl4AIajoiVQ6afVT6wvYMd0pCYBZp7RXHdxTb0sjJ0Beqpkc8bNOgZ0TrEQ08} \end{equation} With this definition of $l^2( \mathcal{H},\mathcal{Y})$-norms, the Burkholder-Davis-Gundy (BDG) inequality holds for all $p\in [1,\infty)$ and reads  \begin{equation} \sadklfjsdfgsdfgsdfgsdfgdsfgsdfgadfasdf \biggl[ \sup_{s\in(0,t]}\biggl| \sifpoierjsodfgupoefasdfgjsdfgjsdfgjsdjflxncvzxnvasdjfaopsruosihjsdfghajsdflahgfsif_0^s G \,d\WW_r \biggr|_{\mathcal{Y}}^p\biggr] \leq  C_p \sadklfjsdfgsdfgsdfgsdfgdsfgsdfgadfasdf\biggl[ \left(\sifpoierjsodfgupoefasdfgjsdfgjsdfgjsdjflxncvzxnvasdjfaopsruosihjsdfghajsdflahgfsif_0^t \sifpoierjsodfgupoefasdfgjsdfgjsdfgjsdjflxncvzxnvasdjfaopsruosihjsdgghajsdflahgfsif G\sifpoierjsodfgupoefasdfgjsdfgjsdfgjsdjflxncvzxnvasdjfaopsruosihjsdgghajsdflahgfsif^2_{ l^2( \mathcal{H},\mathcal{Y})}\, dr \right)^{p/2}\biggr]. \llabel{0wqh1C2HnYQXM8nJ0PfuGJBe2vuqDukLVAJwv2tYcJOM1uKh7pcgoiiKt0b3eURecDVM7ivRMh1T6pAWlupjkEjULR3xNVAu5kEbnrVHE1OrJ2bxdUPyDvyVix6sCBpGDSxjBCn9PFiuxkFvw0QPofRjy2OFItVeDBtDzlc9xVyA0de9Y5h8c7dYCFkFlvWPDSuNVI6MZ72u9MBtK9BGLNsYplX2yb5UHgHADbW8XRzkvUJZShWQHGoKXyVArsHTQ1VbddK2MIxmTf6wET9cXFbuuVxCbSBBp0v2JMQ5Z8z3pMEGpTU6KCcYN2BlWdp2tmliPDHJQWjIRRgqi5lAPgiklc8ruHnvYFMAIrIh7Ths9tEhAAYgSswZZfws19P5weJvMimbsFHThCnSZHORmyt98w3U3zantzAyTwq0CjgDIEtkbh98V4uo52jjAZz1kLoC8oHGvZ5RuGwv3kK4WB50ToMtq7QWG9mtbSIlc87ruZfKwZPh31ZAOsq8ljVQJLTXCgyQn0vKESiSqBpawtHxcIJe4SiE1izzximkePY3s7SX5DASGXHqCr38VYP3HxvOIRZtMfqNoLFoU7vNdtxzwUkX32t94nFdqqTRQOvYqEbigjrSZkTN7XwtPFgNsO7M1mbDAbtVB3LGCpgE9hVFKYLcSGmF8637aZDiz4CuJbLnpE7yl85jg1MTPOLOGEPOeMru1v25XLJFzhwgElnuYmqrX1YKVKvgmMK7gI46h5kZBOoJtfC5gVvA1kNJr2EQ09} \end{equation}  \par For $s\geq0$ and $p\in[1,\infty)$, we write $W^{s,p}$ for the standard Sobolev spaces and use   \begin{equation}    \mathbb{W}^{s,p}=\left\{f\colon\TT^3\to l^2( \mathcal{H},\mathcal{Y})    :    f\mathbf{e}_k\in W^{s,p}(\TT^3) \mbox{ for each }k, \mbox{ and } \sifpoierjsodfgupoefasdfgjsdfgjsdfgjsdjflxncvzxnvasdjfaopsruosihjsdfghajsdflahgfsif_{\TT^3} \bigl\sifpoierjsodfgupoefasdfgjsdfgjsdfgjsdjflxncvzxnvasdjfaopsruosihjsdgghajsdflahgfsif \mathcal{F}^{-1}\bigl((1+|\xi|^2)^{s/2}\mathcal{F}f\bigr)\bigr\sifpoierjsodfgupoefasdfgjsdfgjsdfgjsdjflxncvzxnvasdjfaopsruosihjsdgghajsdflahgfsif_{l^2( \mathcal{H},\mathcal{Y})}^p \,dx<\infty   \right\}   \llabel{o7om1XNpUwtCWXfFTSWDjsIwuxOJxLU1SxA5ObG3IOUdLqJcCArgzKM08DvX2mui13Tt71IwqoFUI0EEf5SV2vxcySYIQGrqrBHIDTJv1OB1CzDIDdW4E4jJmv6KtxoBOs9ADWBq218BJJzRyUQi2GpweET8LaO4ho95g4vWQmoiqjSwMA9CvnGqxl1LrYuMjGboUpuvYQ2CdBlAB97ewjc5RJESFGsORedoM0bBk25VEKB8VA9ytAEOyofG8QIj27aI3jyRmzyETKxpgUq4BvbcD1b1gKByoE3azgelVNu8iZ1w1tqtwKx8CLN28ynjdojUWvNH9qyHaXZGhjUgmuLI87iY7Q9MQWaiFFSGzt84mSQq25ONltTgbl8YDQSAzXqpJEK7bGL1UJn0f59vPrwdtd6sDLjLoo18tQXf55upmTadJDsELpH2vqYuTAmYzDg951PKFP6pEizIJQd8NgnHTND6z6ExRXV0ouUjWTkAKABeAC9Rfjac43AjkXnHdgSy3v5cBets3VXqfpPBqiGf90awg4dW9UkvRiJy46GbH3UcJ86hWVaCMjedsUcqDSZ1DlP2mfBhzu5dvu1i6eW2YNLhM3fWOdzKS6Qov14wxYYd8saS38hIlcPtS4l9B7hFC3JXJGpstll7a7WNrVMwunmnmDc5duVpZxTCl8FI01jhn5Bl4JzaEV7CKMThLji1gyZuXcIv4033NqZLITGUx3ClPCBKO3vRUimJql5blI9GrWyirWHoflH73ZTeZXEQ10}   \end{equation} for the Sobolev norms on noise coefficients. The latter is a Banach space when endowed with the norm  \begin{equation} \sifpoierjsodfgupoefasdfgjsdfgjsdfgjsdjflxncvzxnvasdjfaopsruosihjsdgghajsdflahgfsif f\sifpoierjsodfgupoefasdfgjsdfgjsdfgjsdjflxncvzxnvasdjfaopsruosihjsdgghajsdflahgfsif_{\mathbb{W}^{s,p}}:=\left( \sifpoierjsodfgupoefasdfgjsdfgjsdfgjsdjflxncvzxnvasdjfaopsruosihjsdfghajsdflahgfsif_{\TT^3} \big\sifpoierjsodfgupoefasdfgjsdfgjsdfgjsdjflxncvzxnvasdjfaopsruosihjsdgghajsdflahgfsif \mathcal{F}^{-1}\big[(1+|\xi|^2)^{s/2}\mathcal{F}f\big]\big\sifpoierjsodfgupoefasdfgjsdfgjsdfgjsdjflxncvzxnvasdjfaopsruosihjsdgghajsdflahgfsif_{l^2( \mathcal{H},\mathcal{Y})}^p \,dx\right)^{1/p}.  \label{EQ11} \end{equation}
We abbreviate $\mathbb{W}^{0,p}$ as~$\mathbb{L}^{p}$. Denoting $( \mathcal{P}f)  \mathbf{e}_k=\mathcal{P} (f \mathbf{e}_k)$, we have $\mathcal{P}f\in \mathbb{W}^{s,p}$ whenever $ f\in \mathbb{W}^{s,p}$. \par For $p\in[2,\infty)$ and $r\in[2,6]$, recall an inequality \cite[Lemma~3]{KZ}, which states \begin{equation}\llabel{kopeq8XL1RQ3aUj6Essnj20MA3AsrSVft3F9wzB1qDQVOnHCmmP3dWSbjstoj3oGjadvzqcMB6Y6kD9sZ0bdMjtUThULGTWU9Nmr3E4CNbzUOvThhqL1pxAxTezrHdVMgLYTTrSfxLUXCMrWAbE69K6XHi5re1fx4GDKkiB7f2DXzXez2k2YcYc4QjUyMYR1oDeYNWf74hByFdsWk4cUbCRDXaq4eDWd7qbOt7GOuoklgjJ00J9IlOJxntzFVBCFtpABpVLEE2y5Qcgb35DU4igj4dzzWsoNFwvqjbNFma0amFKivAappzMzrVqYfOulMHafaBk6JreOQBaTEsJBBtHXjn2EUCNleWpcvWJIggWXKsnB3wvmoWK49Nl492ogR6fvc8ffjJmsWJr0jzI9pCBsIUVofDkKHUb7vxpuQUXA6hMUryvxEpcTqlTkzz0qHbXpO8jFuh6nwzVPPzpA8961V78cO2Waw0yGnCHVqBVjTUHlkp6dGHOdvoEE8cw7QDL1o1qg5TXqoV720hhQTyFtpTJDg9E8Dnsp1QiX98ZVQN3sduZqcn9IXozWhFd16IB0K9JeBHvi364kQlFMMJOn0OUBrnvpYyjUBOfsPzxl4zcMnJHdqOjSi6NMn8bR6kPeklTFdVlwDSrhT8Qr0sChNh88j8ZAvvWVD03wtETKKNUdr7WEK1jKSIHFKh2sr1RRVRa8JmBtkWI1ukuZTF2B4p8E7Y3p0DX20JM3XzQtZ3bMCvM4DEAwBFp8qYKpLSEQ13} \sifpoierjsodfgupoefasdfgjsdfgjsdfgjsdjflxncvzxnvasdjfaopsruosihjsdgghajsdflahgfsif |f|^{p/2}\sifpoierjsodfgupoefasdfgjsdfgjsdfgjsdjflxncvzxnvasdjfaopsruosihjsdgghajsdflahgfsif_{r}\leq C\sifpoierjsodfgupoefasdfgjsdfgjsdfgjsdjflxncvzxnvasdjfaopsruosihjsdgghajsdflahgfsif |f|^{p/2}\sifpoierjsodfgupoefasdfgjsdfgjsdfgjsdjflxncvzxnvasdjfaopsruosihjsdgghajsdflahgfsif_{2}^{1-\alpha}\sifpoierjsodfgupoefasdfgjsdfgjsdfgjsdjflxncvzxnvasdjfaopsruosihjsdgghajsdflahgfsif \nabla(|f|^{p/2})\sifpoierjsodfgupoefasdfgjsdfgjsdfgjsdjflxncvzxnvasdjfaopsruosihjsdgghajsdflahgfsif_{2}^{\alpha}, \end{equation} where $\alpha=3(1/2-1/r)$, valid for $f$ with the average~$0$. Particularly, setting $r=6$  yields  \begin{equation}\label{EQ14} \sifpoierjsodfgupoefasdfgjsdfgjsdfgjsdjflxncvzxnvasdjfaopsruosihjsdgghajsdflahgfsif f\sifpoierjsodfgupoefasdfgjsdfgjsdfgjsdjflxncvzxnvasdjfaopsruosihjsdgghajsdflahgfsif_{3p}^p\leq C_p \sifpoierjsodfgupoefasdfgjsdfgjsdfgjsdjflxncvzxnvasdjfaopsruosihjsdgghajsdflahgfsif \nabla(|f|^{p/2})\sifpoierjsodfgupoefasdfgjsdfgjsdfgjsdjflxncvzxnvasdjfaopsruosihjsdgghajsdflahgfsif_{2}^2     \comma p\in[2,\infty)    , \end{equation} under the condition $\sifpoierjsodfgupoefasdfgjsdfgjsdfgjsdjflxncvzxnvasdjfaopsruosihjsdfghajsdflahgfsif_{\TT^3} f=0$. As usual, $C$ denotes a sufficiently large positive constant, whose value may change from line to line. \par In this paper, we only consider probabilistically strong solutions with respect to a given stochastic basis $(\Omega, \cf, (\cf_t)_{t\geq 0}, \PP)$. Namely, when substituted in, the solution satisfies the equation's analytic weak formulation on $(\Omega, \cf, (\cf_t)_{t\geq 0}, \PP)$ up to a stopping time~$\tau$ (see the definition after Theorem~\ref{T01}). \par Assuming that $\sadklfjsdfgsdfgsdfgsdfgdsfgsdfgadfasdf[\vert u_0\sifpoierjsodfgupoefasdfgjsdfgjsdfgjsdjflxncvzxnvasdjfaopsruosihjsdgghajsdflahgfsif_3]$ is small, we assert the existence and pathwise uniqueness of a global $L^{3}$-solution with large probability, as stated next. \par \cole \begin{Theorem} \label{T01} (Global solution with large probability with small initial data in $\LLLot$ or~$\LLLit$) Let $\uu_0\in L^1(\Omega; L^3(\TT^3))$, $\nabla\cdot u_0=0$, and $\sifpoierjsodfgupoefasdfgjsdfgjsdfgjsdjflxncvzxnvasdjfaopsruosihjsdfghajsdflahgfsif_{\TT^3} u_0=0$, and assume that the assumptions~\eqref{EQ03}--\eqref{EQ04} hold, with $\epsilon_{\sigma}\in(0,1]$ being sufficiently small. For every $p_0\in(0,1]$, there is $\epsilon_0\in(0,1]$ such that if    \begin{equation}    \sadklfjsdfgsdfgsdfgsdfgdsfgsdfgadfasdf[\vert u_0\sifpoierjsodfgupoefasdfgjsdfgjsdfgjsdjflxncvzxnvasdjfaopsruosihjsdgghajsdflahgfsif_3]    \leq \epsilon_0     ,    \llabel{o1a5sdRPfTg5R67v1T4eCJ1qg14CTK7u7agjQ0AtZ1Nh6hkSys5CWonIOqgCL3u7feRBHzodSJp7JH8u6RwsYE0mcP4rLaWAtlyRwkHF3eiUyhIiA19ZBu8mywf42nuyX0eljCt3Lkd1eUQEZoOZrA2OqfoQ5CahrByKzFgDOseim0jYBmXcsLAyccCJBTZPEjyzPb5hZKWOxT6dytu82IahtpDm75YDktQvdNjWjIQH1BAceSZKVVP136vL8XhMm1OHKn2gUykFUwN8JMLBqmnvGuwGRoWUoNZY2PnmS5gQMcRYHxLyHuDo8bawaqMNYtonWu2YIOzeB6RwHuGcnfio47UPM5tOjszQBNq7mcofCNjou83emcY81svsI2YDS3SyloBNx5FBVBc96HZEOXUO3W1fIF5jtEMW6KW7D63tH0FCVTZupPlA9aIoN2sf1Bw31ggLFoDO0Mx18ooheEdKgZBCqdqpasaHFhxBrEaRgAuI5dqmWWBMuHfv90ySPtGhFFdYJJLf3Apk5CkSzr0KbVdisQkuSAJEnDTYkjPAEMua0VCtCFfz9R6Vht8UacBe7opAnGa7AbLWjHcsnARGMbn7a9npaMflftM7jvb200TWxUC4lte929joZrAIuIao1ZqdroCL55LT4Q8kNyvsIzPx4i59lKTq2JBBsZbQCECtwarVBMTH1QR6v5srWhRrD4rwf8ik7KHEgeerFVTErONmlQ5LR8vXNZLB39UDzRHZbH9fTBhRwkA2n3pg4IEQ87}   \end{equation} then there exists a stopping time $\tau\in[0,\infty]$ and a unique solution $(u, \tau)$ of  \eqref{EQ01} on $(\Omega, \mathcal{F},(\mathcal{F}_t)_{t\geq 0},\mathbb{P})$ with the initial condition $u_0$ such that   \begin{align}   \begin{split}     \sadklfjsdfgsdfgsdfgsdfgdsfgsdfgadfasdf\biggl[     \sup_{0\leq s\leq \tau}\sifpoierjsodfgupoefasdfgjsdfgjsdfgjsdjflxncvzxnvasdjfaopsruosihjsdgghajsdflahgfsif\uu(s,\cdot)\sifpoierjsodfgupoefasdfgjsdfgjsdfgjsdjflxncvzxnvasdjfaopsruosihjsdgghajsdflahgfsif_3^3     +\sifpoierjsodfgupoefasdfgjsdfgjsdfgjsdjflxncvzxnvasdjfaopsruosihjsdfghajsdflahgfsif_0^{\tau}      \sum_{j}    \sifpoierjsodfgupoefasdfgjsdfgjsdfgjsdjflxncvzxnvasdjfaopsruosihjsdfghajsdflahgfsif_{\TT^3} | \nabla (|\uu_j(s,x)|^{3/2})|^2 \,dx ds     \biggr]     \leq      C\epsilon_0^{3}   \end{split}    \label{EQ15}   \end{align} and   \begin{equation}    \mathbb{P}[\tau<\infty]    \leq p_0.    \llabel{grHxdfEFuz6REtDqPdwN7HTVtcE18hW6yn4GnnCE3MEQ51iPsGZ2GLbtCSthuzvPFeE28MM23ugTCdj7z7AvTLa1AGLiJ5JwWCiDPyMqa8tAKQZ9cfP42kuUzV3h6GsGFoWm9hcfj51dGtWyZzC5DaVt2Wi5IIsgDB0cXLM1FtExERIZIZ0RtQUtWcUCmFmSjxvWpZcgldopk0D7aEouRkuIdOZdWFORuqbPY6HkWOVi7FuVMLWnxpSaNomkrC5uIZK9CjpJyUIeO6kgb7tr2SCYx5F11S6XqOImrs7vv0uvAgrb9hGPFnkRMj92HgczJ660kHbBBlQSIOY7FcX0cuyDlLjbU3F6vZkGbaKaMufjuxpn4Mi457MoLNW3eImcj6OOSe59afAhglt9SBOiFcYQipj5uN19NKZ5Czc231wxGx1utgJB4ueMxx5lrs8gVbZs1NEfI02RbpkfEOZE4eseo9teNRUAinujfeJYaEhns0Y6XRUF1PCf5eEAL9DL6a2vmBAU5AuDDtyQN5YLLWwPWGjMt4hu4FIoLCZLxeBVY5lZDCD5YyBwOIJeHVQsKobYdqfCX1tomCbEj5m1pNx9pnLn5A3g7Uv777YUgBRlNrTyjshaqBZXeAFtjyFlWjfc57t2fabx5Ns4dclCMJcTlqkfquFDiSdDPeX6mYLQzJzUmH043MlgFedNmXQPjAoba07MYwBaC4CnjI4dwKCZPO9wx3en8AoqX7JjN8KlqjQ5cbMSdhRFstQ8Qr2ve2EQ16}   \end{equation} If   \begin{equation}     \sifpoierjsodfgupoefasdfgjsdfgjsdfgjsdjflxncvzxnvasdjfaopsruosihjsdgghajsdflahgfsif u_0\sifpoierjsodfgupoefasdfgjsdfgjsdfgjsdjflxncvzxnvasdjfaopsruosihjsdgghajsdflahgfsif_{\LLLit}\leq \epsilon_0     ,    \label{EQ17}   \end{equation} then we can arrange that   \begin{equation}    \mathbb{P}[\tau>0]    =1    .    \llabel{HT0uO5WjTAiiIWn1CWrU1BHBMvJ3ywmAdqNDLY8lbxXMx0DDvco3RL9Qz5eqywVYqENnO8MH0PYzeVNi3yb2msNYYWzG2DCPoG1VbBxe9oZGcTU3AZuEKbkp6rNeTX0DSMczd91nbSVDKEkVazIqNKUQapNBP5B32EyprwPFLvuPiwRPl1GTdQBZEAw3d90v8P5CPAnX4Yo2q7syr5BW8HcT7tMiohaBW9U4qrbumEQ6XzMKR2BREFXk3ZOMVMYSw9SF5ekq0myNKGnH0qivlRA18CbEzidOiuyZZ6kRooJkLQ0EwmzsKlld6KrKJmRxls12KG2bv8vLxfJwrIcU6Hxpq6pFy7OimmodXYtKt0VVH22OCAjfdeTBAPvPloKQzLEOQlqdpzxJ6JIzUjnTqYsQ4BDQPW6784xNUfsk0aM78qzMuL9MrAcuVVKY55nM7WqnB2RCpGZvHhWUNg93F2eRT8UumC62VH3ZdJXLMScca1mxoOO6oOLOVzfpOBOX5EvKuLz5sEW8a9yotqkcKbDJNUslpYMJpJjOWUy2U4YVKH6kVC1Vx1uvykOyDszo5bzd36qWH1kJ7JtkgV1JxqrFnqmcUyZJTp9oFIcFAk0ITA93SrLaxO9oUZ3jG6fBRL1iZ7ZE6zj8G3MHu86Ayjt3flYcmTkjiTSYvCFtJLqcJPtN7E3POqGOKe03K3WV0epWXDQC97YSbADZUNp81GFfCPbj3iqEt0ENXypLvfoIz6zoFoF9lkIunXjYyYL52UEQ12}   \end{equation} \end{Theorem} \colb \par To be precise, the pair $(\uu,\tau)$ above is called a solution if $\tau$ is a stopping time on $(\Omega, \cf, (\cf_t)_{t\geq 0}, \PP)$, and the process $u\in L^3(\Omega; C([0,\tau], L^3))$ is progressively measurable, divergence-free, and satisfies   \begin{align}    \begin{split}      (u_j( t \wedge \tau),\sifpoierjsodfgupoefasdfgjsdfgjsdfgjsdjflxncvzxnvasdjfaopsruosihjsdkghajsdflahgfsif)      &=      (u_{j, 0},\sifpoierjsodfgupoefasdfgjsdfgjsdfgjsdjflxncvzxnvasdjfaopsruosihjsdkghajsdflahgfsif)    + \sifpoierjsodfgupoefasdfgjsdfgjsdfgjsdjflxncvzxnvasdjfaopsruosihjsdfghajsdflahgfsif_{0}^{t\wedge\tau}          (u_j (r), \Delta\sifpoierjsodfgupoefasdfgjsdfgjsdfgjsdjflxncvzxnvasdjfaopsruosihjsdkghajsdflahgfsif )      \,dr    \\&\indeq    + \sifpoierjsodfgupoefasdfgjsdfgjsdfgjsdjflxncvzxnvasdjfaopsruosihjsdfghajsdflahgfsif_{0}^{t\wedge\tau}       \bigl(\bigl(\mathcal{P}  (u_m(r) u(r))\bigr)_j ,\partial_{m} \sifpoierjsodfgupoefasdfgjsdfgjsdfgjsdjflxncvzxnvasdjfaopsruosihjsdkghajsdflahgfsif\bigr)      \,dr      +\sifpoierjsodfgupoefasdfgjsdfgjsdfgjsdjflxncvzxnvasdjfaopsruosihjsdfghajsdflahgfsif_0^{t\wedge\tau} \bigl(\sigma_j(r, \uu(r)),\sifpoierjsodfgupoefasdfgjsdfgjsdfgjsdjflxncvzxnvasdjfaopsruosihjsdkghajsdflahgfsif\bigr)\,dW(r)\quad      \PP\mbox{-a.s.}     \comma j=1,2,3     ,    \end{split}    \llabel{bRBjxkQUSU9mmXtzIHOCz1KH49ez6PzqWF223C0Iz3CsvuTR9sVtQCcM1eopDPy2lEEzLU0USJtJb9zgyGyfiQ4foCx26k4jLE0ula6aSIrZQHER5HVCEBL55WCtB2LCmveTDzVcp7URgI7QuFbFw9VTxJwGrzsVWM9sMJeJNd2VGGFsiWuqC3YxXoJGKwIo71fgsGm0PYFBzX8eX7pf9GJb1oXUs1q06KPLsMucNytQbL0Z0Qqm1lSPj9MTetkL6KfsC6ZobYhc2quXy9GPmZYj1GoeifeJ3pRAfn6Ypy6jNs4Y5nSEpqN4mRmamAGfYHhSaBrLsDTHCSElUyRMh66XU7hNzpZVC5VnV7VjL7kvWKf7P5hj6t1vugkLGdNX8bgOXHWm6W4YEmxFG4WaNEbGKsv0p4OG0NrduTeZaxNXqV4BpmOdXIq9abPeDPbUZ4NXtohbYegCfxBNttEwcDYSD637jJ2ms6Ta1J2xZPtKnPwAXAtJARc8n5d93TZi7q6WonEDLwWSzeSueYFX8cMhmY6is15pXaOYBbVfSChaLkBRKs6UOqG4jDVabfbdtnyfiDBFI7uhB39FJ6mYrCUUTf2X38J43KyZg87igFR5Rz1t3jH9xlOg1h7P7Ww8wjMJqH3l5J5wU8eH0OogRCvL7fJJg1ugRfMXIGSuEEfbh3hdNY3x197jRqePcdusbfkuJhEpwMvNBZVzLuqxJ9b1BTfYkRJLjOo1aEPIXvZAjvXnefhKGsJGawqjtU7r6MEQ89}   \end{align} for all $\sifpoierjsodfgupoefasdfgjsdfgjsdfgjsdjflxncvzxnvasdjfaopsruosihjsdkghajsdflahgfsif\in C^{\infty}(\TT^3)$ and all $t\in [0,\infty)$.  The pathwise uniqueness is understood in the sense that if both $(u,\tau)$ and $(v, \eta)$ are solutions subject to $(\Omega, \mathcal{F},(\mathcal{F}_t)_{t\geq 0},\mathbb{P})$ and the same Wiener process $\WW$, then we have $ \PP(\uu( t)=v(t),~\forall t \in[0, \tau\wedge \eta])=1 $. \colb \par Note that $\tau$ depends on~$\epsilon_0$ and~$u_0$. Also, observe that the statement applies to small initial data $\uu_0\in L^3(\Omega; L^3(\TT^3))$. In cases where $\uu_0\in L^1(\Omega; L^3(\TT^3))$ or $L^3(\Omega; L^3(\TT^3))$, we may use Markov's inequality in $\omega$, select a large enough portion of the probability space $\Omega_1$, and replace $u_0$ with $u_0 \mathds{1}_{\Omega_1}$, and simply assume 
that \eqref{EQ17} holds, where $\epsilon_0\in(0,1]$ is a sufficiently small constant to be determined.  This is because we can assign $\tau=0$ on~$\Omega_1^{\text{c}}$, and then, after the solution is constructed, redefine $u(0)=u_0$ on~$\Omega_1^{\text{c}}$. \par If the parameter $\epsilon_{\sigma}$ is not  necessarily small, we obtain the global existence and uniqueness on an interval with a deterministic bound. We fix $T>0$ for the next theorem, but $T$ is arbitrary.  \par \cole \begin{Theorem} \label{T02} (Local solution with small initial data) Let $\uu_0\in L^\infty(\Omega; L^3(\TT^3))$, $\nabla\cdot u_0=0$, and $\sifpoierjsodfgupoefasdfgjsdfgjsdfgjsdjflxncvzxnvasdjfaopsruosihjsdfghajsdflahgfsif_{\TT^3} u_0=0$, and assume that the assumptions~\eqref{EQ03}--\eqref{EQ04} hold, where~$\epsilon_{\sigma}>0$. Then there exists  $\epsilon_0\in(0,1]$ such that if    \begin{equation}     \sifpoierjsodfgupoefasdfgjsdfgjsdfgjsdjflxncvzxnvasdjfaopsruosihjsdgghajsdflahgfsif u_0\sifpoierjsodfgupoefasdfgjsdfgjsdfgjsdjflxncvzxnvasdjfaopsruosihjsdgghajsdflahgfsif_{\LLLit}\leq \epsilon_0     ,    \llabel{PoydEH26203mGiJhFnTNCDBYlnPoKO6PuXU3uu9mSg41vmakk0EWUpSUtGBtDe6dKdxZNTFuTi1fMcMhq7POvf0hgHl8fqvI3RK39fn9MaCZgow6e1iXjKC5lHOlpGpkKXdDxtz0HxEfSMjXYL8Fvh7dmJkE8QAKDo1FqMLHOZ2iL9iIm3LKvaYiNK9sb48NxwYNR0nx2t5bWCkx2a31ka8fUIaRGzr7oigRX5sm9PQ7Sr5StZEYmp8VIWShdzgDI9vRF5J81x33nNefjBTVvGPvGsxQhAlGFbe1bQi6JapOJJaceGq1vvb8rF2F3M68eDlzGtXtVm5y14vmwIXa2OGYhxUsXJ0qgl5ZGAtHPZdoDWrSbBSuNKi6KWgr39s9tc7WM4Aws1PzI5cCO7Z8y9lMTLAdwhzMxz9hjlWHjbJ5CqMjhty9lMn4rc76AmkKJimvH9rOtbctCKrsiB04cFVDl1gcvfWh65nxy9ZS4WPyoQByr3vfBkjTZKtEZ7rUfdMicdyCVqnD036HJWMtYfL9fyXxO7mIcFE1OuLQsAQNfWv6kV8Im7Q6GsXNCV0YPoCjnWn6L25qUMTe71vahnHDAoXAbTczhPcfjrjW5M5G0nzNM5TnlJWOPLhM6U2ZFxwpg4NejP8UQ09JX9n7SkEWixERwgyFvttzp4Asv5FTnnMzLVhFUn56tFYCxZ1BzQ3ETfDlCad7VfoMwPmngrDHPfZV0aYkOjrZUw799etoYuBMIC4ovEY8DOLNURVQ5lEQ87-2}   \end{equation} then there exists a stopping time $\tau\in(0,T)$ and a unique solution $(u,\tau)$ of  \eqref{EQ01} on $(\Omega, \mathcal{F},(\mathcal{F}_t)_{t\in[0,T]},\mathbb{P})$ with the initial condition $u_0$ such that   \begin{align}   \begin{split}     \sadklfjsdfgsdfgsdfgsdfgdsfgsdfgadfasdf\biggl[     \sup_{0\leq s\leq \tau}\sifpoierjsodfgupoefasdfgjsdfgjsdfgjsdjflxncvzxnvasdjfaopsruosihjsdgghajsdflahgfsif\uu(s,\cdot)\sifpoierjsodfgupoefasdfgjsdfgjsdfgjsdjflxncvzxnvasdjfaopsruosihjsdgghajsdflahgfsif_3^3     +\sifpoierjsodfgupoefasdfgjsdfgjsdfgjsdjflxncvzxnvasdjfaopsruosihjsdfghajsdflahgfsif_0^{\tau}      \sum_{j}    \sifpoierjsodfgupoefasdfgjsdfgjsdfgjsdjflxncvzxnvasdjfaopsruosihjsdfghajsdflahgfsif_{\TT^3} | \nabla (|\uu_j(s,x)|^{3/2})|^2 \,dx ds     \biggr]     \leq      C\epsilon_0^{3}     ,   \end{split}    \label{EQ15-2}   \end{align}   where the constant $C$ may depend on $T$. \end{Theorem} \colb \par The proof of Theorem~\ref{T02} is a modification of the one for Theorem~\ref{T01}. Since the modification is straight-forward, we omit the detailed proof. \colb \par \subsection{Auxiliary results} In \cite{KXZ}, we investigated the stochastic heat equation  \begin{align}   \begin{split}     \partial_t\uu( t,x)     &=\Delta \uu( t,x) + \nabla \cdot f( t,x) + g( t,x)\dot{\WW}(t),     \\     \uu( 0,x)&= \uu_0 ( x) \Pas     ,   \end{split}   \label{EQ18}  \end{align} where  $u\colon \colr \Omega\times\colb [0,T]\times{\mathcal{D}}\to \mathbb{R}$, and $\mathcal{D}=\mathbb{T}^{3}$. We proved in \cite{KXZ} that \eqref{EQ18} has a unique strong solution up to any deterministic time $T$ and provided an energy estimate of the solution. (See \cite{KXZ} or \cite{KX} for the definition of the strong solution.) The next lemma is a natural consequence of \cite[Theorem~4.1]{KXZ} and \cite[Theorem~3.1]{KWX}. \par \cole \begin{Lemma}   \label{L01}   Let $0<T<\infty$, $2\leq p<\infty$, and ~$3p/(p+1)\leq q\leq p$. Suppose that $u_0\in L^p(\Omega, L^p(\mathcal{D}))$,    $f\in L^p(\Omega\times[0,T], L^{q}(\mathcal{D}))$,    and $g\in L^p(\Omega\times[0,T], \mathbb{L}^p(\mathcal{D}))$. Then there exists a unique global solution $\uu$ of~\eqref{EQ18} in $L^p(\Omega; C([0,T], L^p(\mathcal{D})))$ such that   \begin{align}
    \begin{split}       &\sadklfjsdfgsdfgsdfgsdfgdsfgsdfgadfasdf\biggl[\sup_{0\leq t\leq T}\sifpoierjsodfgupoefasdfgjsdfgjsdfgjsdjflxncvzxnvasdjfaopsruosihjsdgghajsdflahgfsif\uu(t,\cdot)\sifpoierjsodfgupoefasdfgjsdfgjsdfgjsdjflxncvzxnvasdjfaopsruosihjsdgghajsdflahgfsif_p^p+\sifpoierjsodfgupoefasdfgjsdfgjsdfgjsdjflxncvzxnvasdjfaopsruosihjsdfghajsdflahgfsif_0^{T}\sum_{j=1}^{3} \sifpoierjsodfgupoefasdfgjsdfgjsdfgjsdjflxncvzxnvasdjfaopsruosihjsdfghajsdflahgfsif_{\mathcal{D}} | \nabla (|u(t,x)|^{p/2})|^2 \,dx dt\biggr]       \\&\indeq       \leq C       \sadklfjsdfgsdfgsdfgsdfgdsfgsdfgadfasdf\biggl[       \sifpoierjsodfgupoefasdfgjsdfgjsdfgjsdjflxncvzxnvasdjfaopsruosihjsdgghajsdflahgfsif\uu_0\sifpoierjsodfgupoefasdfgjsdfgjsdfgjsdjflxncvzxnvasdjfaopsruosihjsdgghajsdflahgfsif_p^p       +\sifpoierjsodfgupoefasdfgjsdfgjsdfgjsdjflxncvzxnvasdjfaopsruosihjsdfghajsdflahgfsif_0^{T}\sifpoierjsodfgupoefasdfgjsdfgjsdfgjsdjflxncvzxnvasdjfaopsruosihjsdgghajsdflahgfsif f(s,\cdot)\sifpoierjsodfgupoefasdfgjsdfgjsdfgjsdjflxncvzxnvasdjfaopsruosihjsdgghajsdflahgfsif_{q}^p\,ds       +\sifpoierjsodfgupoefasdfgjsdfgjsdfgjsdjflxncvzxnvasdjfaopsruosihjsdfghajsdflahgfsif_0^{T} \sifpoierjsodfgupoefasdfgjsdfgjsdfgjsdjflxncvzxnvasdjfaopsruosihjsdgghajsdflahgfsif g(s, \cdot)\sifpoierjsodfgupoefasdfgjsdfgjsdfgjsdjflxncvzxnvasdjfaopsruosihjsdgghajsdflahgfsif_{\mathbb{L}^{p}}^{p}  \,ds       \biggr]       ,     \end{split}     \label{EQ19}   \end{align} \end{Lemma} \colb \par We also need the following strengthening of~\eqref{EQ19}. \par \cole \begin{Lemma} \label{L02} Let $0<T<\infty$ and $2\leq p<\infty$. Suppose that $u_0\in L^p(\Omega, L^p(\mathcal{D}))$,    $f\in L^p(\Omega\times[0,T], L^{p}(\mathcal{D}))$,  and $g\in L^p(\Omega\times[0,T], \mathbb{L}^p(\mathcal{D}))$. Then there exists a unique global solution $\uu$ of~\eqref{EQ18} in $L^p(\Omega; C([0,T], L^p(\mathcal{D})))$ such that    \begin{align}   \begin{split}   &\sadklfjsdfgsdfgsdfgsdfgdsfgsdfgadfasdf\biggl[\sup_{0\leq t\leq T}\sifpoierjsodfgupoefasdfgjsdfgjsdfgjsdjflxncvzxnvasdjfaopsruosihjsdgghajsdflahgfsif\uu(t,\cdot)\sifpoierjsodfgupoefasdfgjsdfgjsdfgjsdjflxncvzxnvasdjfaopsruosihjsdgghajsdflahgfsif_p^p-\sifpoierjsodfgupoefasdfgjsdfgjsdfgjsdjflxncvzxnvasdjfaopsruosihjsdgghajsdflahgfsif\uu_0\sifpoierjsodfgupoefasdfgjsdfgjsdfgjsdjflxncvzxnvasdjfaopsruosihjsdgghajsdflahgfsif_p^p+\sifpoierjsodfgupoefasdfgjsdfgjsdfgjsdjflxncvzxnvasdjfaopsruosihjsdfghajsdflahgfsif_0^{T} \sifpoierjsodfgupoefasdfgjsdfgjsdfgjsdjflxncvzxnvasdjfaopsruosihjsdfghajsdflahgfsif_{\mathcal{D}} | \nabla (|\uu(t,x)|^{p/2})|^2 \,dx dt\biggr]   \\&\indeq   \leq C   \sadklfjsdfgsdfgsdfgsdfgdsfgsdfgadfasdf\biggl[   {   	\sifpoierjsodfgupoefasdfgjsdfgjsdfgjsdjflxncvzxnvasdjfaopsruosihjsdfghajsdflahgfsif_0^{T} \sifpoierjsodfgupoefasdfgjsdfgjsdfgjsdjflxncvzxnvasdjfaopsruosihjsdfghajsdflahgfsif_{\mathcal{D}}|f(t,x)|^2 |u(t,x)|^{p-2}\,dxdt   }   {   	+ \sifpoierjsodfgupoefasdfgjsdfgjsdfgjsdjflxncvzxnvasdjfaopsruosihjsdfghajsdflahgfsif_0^{T}\sifpoierjsodfgupoefasdfgjsdfgjsdfgjsdjflxncvzxnvasdjfaopsruosihjsdfghajsdflahgfsif_{\mathcal{D}} |\uu(t)|^{p-2} \sifpoierjsodfgupoefasdfgjsdfgjsdfgjsdjflxncvzxnvasdjfaopsruosihjsdgghajsdflahgfsif g(t,x)\sifpoierjsodfgupoefasdfgjsdfgjsdfgjsdjflxncvzxnvasdjfaopsruosihjsdgghajsdflahgfsif_{l^2}^2\,dxdt   }   \biggr],   \end{split}   \label{EQ20}   \end{align} where $C$ is a positive constant that depends only on~$p$. \end{Lemma} \colb \par We emphasize that in Lemma~\ref{L02} the constant $C$ does not depend on~$T$. \par \begin{proof}[Proof of Lemma~\ref{L02}]  We follow the ideas in \cite[Theorem~4.1]{KXZ}, smoothing $u_0$, $f$, and $g$ by standard mollifiers and performing the estimates on~$u^{\epsilon}$. Note that the existence of $u^{\epsilon}$ is guaranteed by \cite[Theorem 4.1.4]{R} for every $\epsilon>0$ and $u^{\epsilon}$ has a regular modification. \par As in~\cite[Theorem~4.1]{KXZ}, the estimate \eqref{EQ20} is derived from the It\^o expansion for~$|u^{\epsilon}|^p$, which asserts  \begin{align}   \begin{split}    &\sifpoierjsodfgupoefasdfgjsdfgjsdfgjsdjflxncvzxnvasdjfaopsruosihjsdgghajsdflahgfsif\ueps(t)\sifpoierjsodfgupoefasdfgjsdfgjsdfgjsdjflxncvzxnvasdjfaopsruosihjsdgghajsdflahgfsif_{p}^{p}    - \sifpoierjsodfgupoefasdfgjsdfgjsdfgjsdjflxncvzxnvasdjfaopsruosihjsdgghajsdflahgfsif\ueps_0\sifpoierjsodfgupoefasdfgjsdfgjsdfgjsdjflxncvzxnvasdjfaopsruosihjsdgghajsdflahgfsif_{p}^{p}  +\frac{4(p-1)}{p}\sifpoierjsodfgupoefasdfgjsdfgjsdfgjsdjflxncvzxnvasdjfaopsruosihjsdfghajsdflahgfsif_0^t \sifpoierjsodfgupoefasdfgjsdfgjsdfgjsdjflxncvzxnvasdjfaopsruosihjsdfghajsdflahgfsif_{\TT^\dd} | \nabla (|\ueps(r)|^{p/2})|^2 \,dx dr \\&\indeq \leq p\sifpoierjsodfgupoefasdfgjsdfgjsdfgjsdjflxncvzxnvasdjfaopsruosihjsdfghajsdflahgfsif_0^t\left|\sifpoierjsodfgupoefasdfgjsdfgjsdfgjsdjflxncvzxnvasdjfaopsruosihjsdfghajsdflahgfsif_{\TT^\dd} |\ueps(r)|^{p-2}\ueps(r) \partial_{j} f^{\epsilon}_j(r)\,dx\right| \,dr \\&\indeq\indeq +\frac{p(p-1)}{2}\sifpoierjsodfgupoefasdfgjsdfgjsdfgjsdjflxncvzxnvasdjfaopsruosihjsdfghajsdflahgfsif_0^t\sifpoierjsodfgupoefasdfgjsdfgjsdfgjsdjflxncvzxnvasdjfaopsruosihjsdfghajsdflahgfsif_{\TT^\dd} |\ueps(r)|^{p-2}\sifpoierjsodfgupoefasdfgjsdfgjsdfgjsdjflxncvzxnvasdjfaopsruosihjsdgghajsdflahgfsif g^{\epsilon}(r)\sifpoierjsodfgupoefasdfgjsdfgjsdfgjsdjflxncvzxnvasdjfaopsruosihjsdgghajsdflahgfsif_{l^2}^2\,dx dr + p\left|\sifpoierjsodfgupoefasdfgjsdfgjsdfgjsdjflxncvzxnvasdjfaopsruosihjsdfghajsdflahgfsif_0^t\sifpoierjsodfgupoefasdfgjsdfgjsdfgjsdjflxncvzxnvasdjfaopsruosihjsdfghajsdflahgfsif_{\TT^\dd} |\ueps(r)|^{p-2}\ueps(r) g^{\epsilon}(r) \,dxd\WW_r\right| \\&\indeq =   I_1 + I_2 + I_3 \commaone t\in [0,T].     \end{split} \llabel{ti1iSNZAdwWr6Q8oPFfae5lAR9gDRSiHOeJOWwxLv20GoMt2Hz7YcalyPZxeRuFM07gaV9UIz7S43k5TrZiDMt7pENCYiuHL7gac7GqyN6Z1ux56YZh2dyJVx9MeUOMWBQfl0EmIc5Zryfy3irahCy9PiMJ7ofoOpdennsLixZxJtCjC9M71vO0fxiR51mFIBQRo1oWIq3gDPstD2ntfoX7YUoS5kGuVIGMcfHZe37ZoGA1dDmkXO2KYRLpJjIIomM6Nuu8O0jO5NabUbRnZn15khG94S21V4Ip457ooaiPu2jhIzosWFDuO5HdGrdjvvtTLBjovLLiCo6L5LwaPmvD6Zpal69Ljn11reT2CPmvjrL3xHmDYKuv5TnpC1fMoURRToLoilk0FEghakm5M9cOIPdQlGDLnXerCykJC10FHhvvnYaTGuqUrfTQPvwEqiHOvOhD6AnXuvGlzVAvpzdOk36ymyUoFbAcAABItOes52Vqd0Yc7U2gBt0WfFVQZhrJHrlBLdCx8IodWpAlDS8CHBrNLzxWp6ypjuwWmgXtoy1vPbrauHyMNbkUrZD6Ee2fzIDtkZEtiLmgre1woDjuLBBSdasYVcFUhyViCxB15yLtqlqoUhgL3bZNYVkorzwa3650qWhF22epiXcAjA4ZV4bcXxuB3NQNp0GxW2Vs1zjtqe2pLEBiS30E0NKHgYN50vXaK6pNpwdBX2Yv7V0UddTcPidRNNCLG47Fc3PLBxK3Bex1XzyXcj0Z6aJk0HKEQ103} \end{align} Note that in our model the drift function is $\partial_{j}f_j^{\epsilon}$ instead of~$f^{\epsilon}$. We obtain the last term in \eqref{EQ20} from $I_2$ and $I_3$, which was already explained in \cite[p.~170 and the bottom of~p.~172]{KXZ}. However, we handle the term $I_1$ differently from \cite[p.~170]{KXZ}. Using integration by parts, we have for $p\in[2,\infty)$,   \begin{align}   \begin{split}    I_1     &=     p(p-1)     \sifpoierjsodfgupoefasdfgjsdfgjsdfgjsdjflxncvzxnvasdjfaopsruosihjsdfghajsdflahgfsif_{0}^{t}     \left|\sifpoierjsodfgupoefasdfgjsdfgjsdfgjsdjflxncvzxnvasdjfaopsruosihjsdfghajsdflahgfsif_{\mathcal{D}} |u^{\epsilon}(r)|^{p-2}\partial_{j}u^{\epsilon}(r) f^{\epsilon}_j(r)\,dx\right| \,dr    \\&    = 2(p-1)      \left|\sifpoierjsodfgupoefasdfgjsdfgjsdfgjsdjflxncvzxnvasdjfaopsruosihjsdfghajsdflahgfsif_{\mathcal{D}} \frac{|u^{\epsilon}|^{p/2}}{u^{\epsilon}}      \partial_{j}(|u^{\epsilon}|^{p/2}) f^{\epsilon}_j          \,dx\right| \,dr   \\&   \leq   \delta\sifpoierjsodfgupoefasdfgjsdfgjsdfgjsdjflxncvzxnvasdjfaopsruosihjsdfghajsdflahgfsif_{0}^{t}\sifpoierjsodfgupoefasdfgjsdfgjsdfgjsdjflxncvzxnvasdjfaopsruosihjsdfghajsdflahgfsif_{\mathcal{D}}   |\nabla(|u^{\epsilon}|^{p/2})|^2   + 2 C_{\delta,p}   \sifpoierjsodfgupoefasdfgjsdfgjsdfgjsdjflxncvzxnvasdjfaopsruosihjsdfghajsdflahgfsif_{0}^{t}\sifpoierjsodfgupoefasdfgjsdfgjsdfgjsdjflxncvzxnvasdjfaopsruosihjsdfghajsdflahgfsif_{\mathcal{D}}   |f^{\epsilon}_j|^2   |u^{\epsilon}|^{p-2}   \,dx dr   \comma t\in(0,T]   ,   \end{split}    \llabel{uQnwdDhPQ1QrwA05v9c3pnzttztx2IirWCZBoS5xlOKCiD3WFh4dvCLQANAQJGgyvODNTDFKjMc0RJPm4HUSQkLnTQ4Y6CCMvNjARZblir7RFsINzHiJlcgfxSCHtsZOG1VuOzk5G1CLtmRYIeD35BBuxZJdYLOCwS9lokSNasDLj5h8yniu7hu3cdizYh1PdwEl3m8XtyXQRCAbweaLiN8qA9N6DREwy6gZexsA4fGEKHKQPPPKMbksY1jM4h3JjgSUOnep1wRqNGAgrL4c18Wv4kchDgRx7GjjIBzcKQVf7gATrZxOy6FF7y93iuuAQt9TKRxS5GOTFGx4Xx1U3R4s7U1mpabpDHgkicxaCjkhnobr0p4codyxTCkVj8tW4iP2OhTRF6kU2k2ooZJFsqY4BFSNI3uW2fjOMFf7xJveilbUVTArCTvqWLivbRpg2wpAJOnlRUEPKhj9hdGM0MigcqQwkyunBJrTLDcPgnOSCHOsSgQsR35MB7BgkPk6nJh01PCxdDsw514O648VD8iJ54FW6rs6SyqGzMKfXopoe4eo52UNB4Q8f8NUz8u2nGOAXHWgKtGAtGGJsbmz2qjvSvGBu5e4JgLAqrmgMmS08ZFsxQm28M3z4Ho1xxjj8UkbMbm8M0cLPL5TS2kIQjZKb9QUx2Ui5Aflw1SLDGIuWUdCPjywVVM2ct8cmgOBS7dQViXR8Fbta1mtEFjTO0kowcK2d6MZiW8PrKPI1sXWJNBcREVY4H5QQGHbplPbwdEQ22}   \end{align} for $\delta>0$. The first term can be absorbed into the dissipative term provided that $\delta>0$ is sufficiently small. Hence, we obtain~\eqref{EQ20} for $u^{\epsilon}$ and then for the solution $u$ of \eqref{EQ18} by passing to the weak limit.  \end{proof} \par Finally, we also recall~\cite[Lemma~3.2]{KWX}. \par \cole \begin{Lemma}   \label{L03}   Let $T>0$ and $p\geq 2$, and let $\{u_n\}_{n\in\NNp}$ be a sequence of scalar-valued processes such that   $\{\nabla (|\uu_{n}|^{p/2})\}_{n\in\NNp}$ is bounded in $L^2(\Omega\times[0,T], L^2)$   and $\uu_{n}$ converges to $u$ in $L^p(\Omega, L^{\infty}([0,T], L^p))$ as $n\to\infty$. Then,   \begin{align}
    \sadklfjsdfgsdfgsdfgsdfgdsfgsdfgadfasdf\biggl[\sifpoierjsodfgupoefasdfgjsdfgjsdfgjsdjflxncvzxnvasdjfaopsruosihjsdfghajsdflahgfsif_0^{T} \sifpoierjsodfgupoefasdfgjsdfgjsdfgjsdjflxncvzxnvasdjfaopsruosihjsdfghajsdflahgfsif_{\mathcal{D}} |\nabla (|u|^{p/2})|^2 \,dx dt\biggr]     \leq     \liminf_{n\to \infty}     \sadklfjsdfgsdfgsdfgsdfgdsfgsdfgadfasdf\biggl[\sifpoierjsodfgupoefasdfgjsdfgjsdfgjsdjflxncvzxnvasdjfaopsruosihjsdfghajsdflahgfsif_0^{T} \sifpoierjsodfgupoefasdfgjsdfgjsdfgjsdjflxncvzxnvasdjfaopsruosihjsdfghajsdflahgfsif_{\mathcal{D}} | \nabla (|\uu_{n}|^{p/2})|^2 \,dx dt\biggr]     .     \llabel{TxpOI5OQZAKyiix7QeyYI91Ea16rKXKL2ifQXQPdPNL6EJiHcKrBs2qGtQbaqedOjLixjGiNWr1PbYSZeSxxFinaK9EkiCHV2a13f7G3G3oDKK0ibKVy453E2nFQS8Hnqg0E32ADddEVnmJ7HBc1t2K2ihCzZuy9kpsHn8KouARkvsHKPy8YodOOqBihF1Z3CvUFhmjgBmuZq7ggWLg5dQB1kpFxkk35GFodk00YD13qIqqbLwyQCcyZRwHAfp79oimtCc5CV8cEuwUw7k8Q7nCqWkMgYrtVRIySMtZUGCHXV9mr9GHZol0VEeIjQvwgw17pDhXJSFUcYbqUgnGV8IFWbS1GXaz0ZTt81w7EnIhFF72v2PkWOXlkrw6IPu5679vcW1f6z99lM2LI1Y6Naaxfl18gT0gDptVlCN4jfGSbCro5Dv78CxaukYiUIWWyYDRw8z7KjPx7ChC7zJvb1b0rFd7nMxk091wHvy4u5vLLsJ8NmAkWtxuf4P5NwP23b06sFNQ6xgDhuRGbK7j2O4gy4p4BLtop3h2kfyI9wO4AaEWb36YyHYiI1S3COJ7aN1r0sQOrCAC4vL7yrCGkIRlNuGbOuuk1awLDK2zlKa40hyJnDV4iFxsqO001rqCeOAO2es7DRaCpUG54F2i97xSQrcbPZ6K8Kudn9e6SYo396Fr8LUxyXOjdFsMrl54EhT8vrxxF2phKPbszrlpMAubERMGQAaCBu2LqwGasprfIZOiKVVbuVae6abaufy9KcFEQ23}   \end{align} \end{Lemma} \colb   \startnewsection{Global solutions of the truncated difference equations}{sec03} \colb \subsection{Decomposition into smooth data and a construction of $L^{6}$ solutions} \par Throughout this section, assume that $u_0$ is an initial datum, which satisfies \eqref{EQ17},  for some fixed $\epsilon_0>0$, is divergence-free, and has zero average over~$\DD=\mathbb{T}^{3}$. \par \colb In the following statement, we decompose the initial datum into small regular components. This lemma is an essential building block for constructing the desired solution of the SNSE. \par \cole \begin{Lemma}[Decomposition of initial data] \label{L05} There exists a sequence of functions    \begin{equation}    v_{0}^{(0)},v^{(1)}_{0},v_{0}^{(2)},\ldots \in \LLL{\infty}{\infty}    \label{EQ24}   \end{equation} that are divergence-free and have average~$0$ such that   \begin{equation}    \sifpoierjsodfgupoefasdfgjsdfgjsdfgjsdjflxncvzxnvasdjfaopsruosihjsdgghajsdflahgfsif   v_{0}^{(0)}\sifpoierjsodfgupoefasdfgjsdfgjsdfgjsdjflxncvzxnvasdjfaopsruosihjsdgghajsdflahgfsif_{\LLLit}    \leq 2\epsilon_0    \llabel{k6cBlZ5rKUjhtWE1Cnt9RmdwhJRySGVSOVTv9FY4uzyAHSp6yT9s6R6oOi3aqZlL7bIvWZ18cFaiwptC1ndFyp4oKxDfQz28136a8zXwsGlYsh9Gp3TalnrRUKttBKeFr4543qU2hh3WbYw09g2WLIXzvQzMkj5f0xLseH9dscinGwuPJLP1gEN5WqYsSoWPeqjMimTybHjjcbn0NO5hzP9W40r2w77TAoz70N1au09bocDSxGc3tvKLXaC1dKgw9H3o2kEoulIn9TSPyL2HXO7tSZse01Z9HdslDq0tmSOAVqtA1FQzEMKSbakznw839wnH1DpCjGIk5X3B6S6UI7HIgAaf9EV33Bkkuo3FyEi8Ty2ABPYzSWjPj5tYZETYzg6Ix5tATPMdlGke67Xb7FktEszyFycmVhGJZ29aPgzkYj4cErHCdP7XFHUO9zoy4AZaiSROpIn0tp7kZzUVHQtm3ip3xEd41By72uxIiY8BCLbOYGoLDwpjuza6iPakZdhaD3xSXyjpdOwoqQqJl6RFglOtX67nm7s1lZJmGUrdIdXQ7jps7rcdACYZMsBKANxtkqfNhktsbBf2OBNZ5pfoqSXtd3cHFLNtLgRoHrnNlwRnylZNWVNfHvOB1nUAyjtxTWW4oCqPRtuVuanMkLvqbxpNi0xYnOkcdFBdrw1Nu7cKybLjCF7P4dxj0Sbz9faVCWkVFos9t2aQIPKORuEjEMtbSHsYeG5Z7uMWWAwRnR8FwFCzXVVxnFUfyKLNk4EQ25}   \end{equation} and   \begin{equation}    \sifpoierjsodfgupoefasdfgjsdfgjsdfgjsdjflxncvzxnvasdjfaopsruosihjsdgghajsdflahgfsif   v_{0}^{(k)}\sifpoierjsodfgupoefasdfgjsdfgjsdfgjsdjflxncvzxnvasdjfaopsruosihjsdgghajsdflahgfsif_{\LLLit}    \leq \frac{\epsilon_0}{4^{k}}     \comma k=1,2,3,\ldots     ;    \label{EQ26}   \end{equation} additionally, we have   \begin{equation}    u_{0} = v_{0}^{(0)} + v_{0}^{(1)} + v_{0}^{(2)} + \cdots \quad\mbox{ in }L_x^3    ,    \label{EQ27}   \end{equation} almost surely. \end{Lemma} \colb \par The proof shows that we may take $v_{0}^{(j)}\in L_{\omega}^{\infty}C_{x}^{k}$ for every $j,k\in\mathbb{N}_0$. \par \begin{proof}[Proof of Lemma~\ref{L05}] Let $\sifpoierjsodfgupoefasdfgjsdfgjsdfgjsdjflxncvzxnvasdjfaopsruosihjsdkghajsdflahgfsif_{\epsilon}$ be a standard mollifier. Then,   \begin{equation} \lim_{n\to\infty}   \sifpoierjsodfgupoefasdfgjsdfgjsdfgjsdjflxncvzxnvasdjfaopsruosihjsdgghajsdflahgfsif  u_{0}* \sifpoierjsodfgupoefasdfgjsdfgjsdfgjsdjflxncvzxnvasdjfaopsruosihjsdkghajsdflahgfsif_{1/n}- u_{0}\sifpoierjsodfgupoefasdfgjsdfgjsdfgjsdjflxncvzxnvasdjfaopsruosihjsdgghajsdflahgfsif_{L_{x}^{3}}    =0    \tae \omega\in\Omega    .    \llabel{eOIlyn3ClI5HP8XP6S4KFfIl62VlbXgcauth861pUWUx2aQTWgrZwcAx52TkqoZXVg0QGrBrrpeiwuWyJtd9ooD8tUzAdLSnItarmhPAWBmnmnsbxLIqX4RQSTyoFDIikpeILhWZZ8icJGa91HxRb97knWhp9sAVzPo8560pRN2PSMGMMFK5XW52OnWIyoYngxWno868SKbbu1Iq1SyPkHJVCvseVGWrhUdewXw6CSY1be3hD9PKha1y0SRwyxiAGzdCMVMmiJaemmP8xrbJXbKLDYE1FpXUKADtF9ewhNefd2XRutTl1HYJVp5cAhM1JfK7UIcpkdTbEndM6FWHA72PgLHzXlUo39oW90BuDeJSlnVRvz8VDV48tId4DtgFOOa47LEH8QwnRGNBM0RRULluASzjxxwGIBHmVyyLdkGww5eEgHFvsFUnzl0vgOaQDCVEz64r8UvVHTtDykrEuFaS35p5yn6QZUcX3mfETExz1kvqEpOVVEFPIVpzQlMOIZ2yTTxIUOm0fWL1WoxCtlXWs9HU4EF0IZ1WDv3TP42LN7TrSuR8uMv1tLepvZoeoKLxf9zMJ6PUIn1S8I4KY13wJTACh5Xl8O5g0ZGwDdtu68wvrvnDCoqYjJ3nFKWMAK8VOeGo4DKxnEOyBwgmttcES8dmToAD0YBFlyGRBpBbo8tQYBwbSX2lcYnU0fhAtmyR3CKcUAQzzETNgbghHT64KdOfLqFWuk07tDkzfQ1dgBcw0LSYlr79U81QPqrdfHEQ28}   \end{equation} Therefore, there exists $n_0\in \NNp$ and an $\mathcal{F}_0$-measurable set $\Omega_{0}$ such that
$\PP(\Omega_0)\geq 1/2$, and   \begin{equation}    \sifpoierjsodfgupoefasdfgjsdfgjsdfgjsdjflxncvzxnvasdjfaopsruosihjsdgghajsdflahgfsif  u_{0}* \sifpoierjsodfgupoefasdfgjsdfgjsdfgjsdjflxncvzxnvasdjfaopsruosihjsdkghajsdflahgfsif_{1/n_0}- u_{0}\sifpoierjsodfgupoefasdfgjsdfgjsdfgjsdjflxncvzxnvasdjfaopsruosihjsdgghajsdflahgfsif_{L_{x}^{3}}    \leq \frac{\epsilon_0}{8}    \tae \omega\in\Omega_0.    \llabel{1tb8kKnDl52FhCj7TXiP7GFC7HJKfXgrP4KOOg18BM001mJPTpubQr61JQu6oGr4baj60kzdXoDgAOX2DBkLymrtN6T7us2Cp6eZm1aVJTY8vYPOzMnsAqs3RL6xHumXNAB5eXnZRHaiECOaaMBwAb15iFWGucZlU8JniDNKiPGWzq41iBj1kqbakZFSvXqvSiRbLTriSy8QYOamQUZhOrGHYHWguPBzlAhuao59RKUtrF5KbjsKseTPXhUqRgnNALVtaw4YJBtK9fN7bN9IEwKLTYGtnCcc2nfMcx7VoBt1IC5teMHX4g3JK4JsdeoDl1Xgbm9xWDgZ31PchRS1R8W1hap5Rh6JjyTNXSCUscxK4275D72gpRWxcfAbZY7Apto5SpTzO1dPAVyZJiWCluOjOtEwxUB7cTtEDqcAbYGdZQZfsQ1AtHyxnPL5K7D91u03s8K20rofZ9w7TjxyG7qbCAhssUZQuPK7xUeK7F4HKfrCEPJrgWHDZQpvRkO8XveaSBOXSeeXV5jkgzLUTmMbomaJfxu8gArndzSIB0YQSXvcZW8voCOoOHyrEuGnS2fnGEjjaLzZIocQegwHfSFKjW2LbKSnIcG9WnqZya6qAYMSh2MmEAsw18nsJFYAnbrxZT45ZwBsBvK9gSUgyBk3dHqdvYULhWgGKaMfFk78mP20meVaQp2NWIb6hVBSeSVwnEqbq6ucnX8JLkIRJbJEbwEYwnvLBgM94Gplclu2s3Um15EYAjs1GLnhzG8vmhEQ29}   \end{equation} Clearly, the function $v_{0}^{(0)} = ( u_{0}* \sifpoierjsodfgupoefasdfgjsdfgjsdfgjsdjflxncvzxnvasdjfaopsruosihjsdkghajsdflahgfsif_{1/n_0}) \mathds{1}_{\Omega_0}$ belongs to~$\LLLii$. Furthermore, it satisfies $\sifpoierjsodfgupoefasdfgjsdfgjsdfgjsdjflxncvzxnvasdjfaopsruosihjsdgghajsdflahgfsif v_{0}^{(0)}\sifpoierjsodfgupoefasdfgjsdfgjsdfgjsdjflxncvzxnvasdjfaopsruosihjsdgghajsdflahgfsif_{\LLLit}\leq 2\epsilon_0$ and $\sifpoierjsodfgupoefasdfgjsdfgjsdfgjsdjflxncvzxnvasdjfaopsruosihjsdgghajsdflahgfsif u_{0} - v_{0}^{(0)}\sifpoierjsodfgupoefasdfgjsdfgjsdfgjsdjflxncvzxnvasdjfaopsruosihjsdgghajsdflahgfsif_{L_{x}^{3}}\leq \epsilon_0/8$ on~$\Omega_0$. Note that the standard mollification preserves the divergence and zero average conditions. We have thus constructed~$v_{0}^{(0)}$. We proceed by induction,  assuming that $   v_{0}^{(0)},v^{(1)}_{0},v_{0}^{(2)},\ldots,v_{0}^{(m)} \in \LLL{\infty}{\infty}$ and $\Omega_0,\Omega_1,\ldots,\Omega_m\in\mathcal{F}_0$ have been constructed, for some $m\in\NNz$, such that \eqref{EQ26} holds, as well as $\PP(\Omega_j)\geq 1-1/2^{j+1}$ and   \begin{equation}    \sifpoierjsodfgupoefasdfgjsdfgjsdfgjsdjflxncvzxnvasdjfaopsruosihjsdgghajsdflahgfsif u_{0} - v_{0}^{(0)}-\ldots-v_{0}^{(j)}\sifpoierjsodfgupoefasdfgjsdfgjsdfgjsdjflxncvzxnvasdjfaopsruosihjsdgghajsdflahgfsif_{L_{x}^{3}}\leq \frac{\epsilon_0}{2^{j+3}}       \tae \omega\in\Omega_j     ,    \llabel{ghsQcEDE1KnaHwtuxOgUDLBE59FLxIpvuKfJEUTQSEaZ6huBCaKXrlnir1XmLKH3hVPrqixmTkRzh0OGpOboN6KLCE0GaUdtanZ9Lvt1KZeN5GQcLQLL0P9GXuakHm6kqk7qmXUVH2bUHgav0Wp6Q8JyITzlpqW0Yk1fX8gjGcibRarmeSi8lw03WinNXw1gvvcDeDPSabsVwZu4haO1V2DqwkJoRShjMBgryglA93DBdS0mYAcEl5aEdpIIDT5mbSVuXo8NlY24WCA6dfCVF6Ala6iNs7GChOvFAhbxw9Q71ZRC8yRi1zZdMrpt73douogkAkGGE487Vii4OfwJesXURdzVLHU0zms8W2ZtziY5mw9aBZIwk5WNmvNM2HdjnewMR8qp2VvupcV4PcjOGeu35u5cQXNTykfTZXAJHUnSs4zxfHwf10ritJYoxRto5OMFPhakRgzDYPm02mG18vmfV11Nn87zSX59DE0cN99uEUz2rTh1FP8xjrmq2Z7utpdRJ2DdYkjy9JYkoc38KduZ9vydOwkO0djhXSxSvHwJoXE79f8qhiBr8KYTxOfcYYFsMyj0HvK3ayUwt4nA5H76bwUqyJQodOu8UGjbt6vlcxYZt6AUxwpYr18uOv62vjnwFrCrfZ4nlvJuh2SpVLOvpOlZnPTG07VReixBmXBxOBzpFW5iBIO7RVmoGnJu8AxolYAxlJUrYKVKkpaIkVCuPiDO8IHPUndzeLPTILBP5BqYyDLZDZadbjcJAT644VEQ108}   \end{equation} for $j=1,2,\ldots,m$. By   \begin{equation}     \lim_{n\to\infty}   \sifpoierjsodfgupoefasdfgjsdfgjsdfgjsdjflxncvzxnvasdjfaopsruosihjsdgghajsdflahgfsif  v* \sifpoierjsodfgupoefasdfgjsdfgjsdfgjsdjflxncvzxnvasdjfaopsruosihjsdkghajsdflahgfsif_{1/n}- v\sifpoierjsodfgupoefasdfgjsdfgjsdfgjsdjflxncvzxnvasdjfaopsruosihjsdgghajsdflahgfsif_{L_{x}^{3}}    =0    \tae \omega\in\Omega    ,    \llabel{p6byb1g4dE7YdzkeOYLhCReOmmxF9zsu0rp8Ajzd2vHeo7L5zVnL8IQWnYATKKV1f14s2JgeCb3v9UJdjNNVBINix1q5oyrSBM2Xtgrv8RQMaXka4AN9iNinzfHxGpA57uAE4jMfg6S6eNGKvJL3tyH3qwdPrx2jFXW2WihpSSxDraA7PXgjK6GGlOg5PkRd2n53eEx4NyGhd8ZRkONMQqLq4sERG0CssQkdZUaOvWrplaBOWrSwSG1SM8Iz9qkpdv0CRMsGcZLAz4Gk70eO7k6df4uYnR6T5DuKOTsay0DawWQvn2UOOPNqQT7H4HfiKYJclRqM2g9lcQZcvCNBP2BbtjvVYjojrrh78tWR886ANdxeASVPhK3uPrQRs6OSW1BwWM0yNG9iBRI7opGCXkhZpEo2JNtkyYOpCY9HL3o7Zu0J9FTz6tZGLn8HAeso9umpyucs4l3CA6DCQ0m0llFPbc8z5Ad2lGNwSgAXeNHTNpwdS6e3ila2tlbXN7c1itXaDZFakdfJkz7TzaO4kbVhnYHfTda9C3WCbtwMXHWxoCCc4Ws2CUHBsNLFEfjS4SGI4I4hqHh2nCaQ4nMpnzYoYE5fDsXhCHJzTQOcbKmvEplWUndVUorrqiJzRqTdIWSQBL96DFUd64k5gvQh0djrGlw795xV6KzhTl5YFtCrpybHH86h3qnLyzyycGoqmCbfh9hprBCQpFeCxhUZ2oJF3aKgQH8RyImF9tEksgPFMMJTAIyz3ohWjHxMR86KJOEQ28old}   \end{equation} where   \begin{equation}    v=u_{0} - v_{0}^{(0)}-\cdots -v_{0}^{(k)}    ,    \llabel{NKTc3uyRNnSKHlhb11Q9Cwrf8iiXqyYL4zh9s8NTEve539GzLgvhDN7FeXo5kAWAT6VrwhtDQwytuHOa5UIOExbMpV2AHpuuCHWItfOruxYfFqsaP8ufHF16CEBXKtj6ohsuvT8BBPDNgGfKQg6MBK2x9jqRbHmjIUEKBIm0bbKacwqIXijrFuq9906Vym3Ve1gBdMy9ihnbA3gBo5aBKK5gfJSmNeCWwOMt9xutzwDkXIY7nNhWdDppZUOq2Ae0aW7A6XoIcTSLNDZyf2XjBcUweQTZtcuXIDYsDhdAu3VMBBBKWIcFNWQdOu3Fbc6F8VN77DaIHE3MZluLYvBmNZ2wEauXXDGpeKRnwoUVB2oMVVehW0ejGgbgzIw9FwQhNYrFI4pTlqrWnXzz2qBbalv3snl2javzUSncpwhcGJ0Di3Lr3rs6F236obLtDvN9KqApOuold3secxqgSQNZNfw5tBGXPdvW0k6G4Byh9V3IicOnR2obfx3jrwt37u82fwxwjSmOQq0pq4qfvrN4kFWhPHRmylxBx1zCUhsDNYINvLdtVDG35kTMT0ChPEdjSG4rWN6v5IIMTVB5ycWuYOoU6SevyecOTfZJvBjSZZkM68vq4NOpjX0oQ7rvMvmyKftbioRl5c4ID72iFH0VbQzhjHU5Z9EVMX81PGJssWedmhBXKDAiqwUJVGj2rIS92AntBn1QPR3tTJrZ1elVoiKUstzA8fCCgMwfw4jKbDberBRt6T8OZynNOqXc53PgfLEQ106}   \end{equation} there exists $n_{m+1}\in\mathbb{N}$ such that   \begin{equation}    \sifpoierjsodfgupoefasdfgjsdfgjsdfgjsdjflxncvzxnvasdjfaopsruosihjsdgghajsdflahgfsif  v* \sifpoierjsodfgupoefasdfgjsdfgjsdfgjsdjflxncvzxnvasdjfaopsruosihjsdkghajsdflahgfsif_{1/n_{m+1}}- v\sifpoierjsodfgupoefasdfgjsdfgjsdfgjsdjflxncvzxnvasdjfaopsruosihjsdgghajsdflahgfsif_{L_{x}^{3}}    \leq \frac{\epsilon_0}{2^{m+4}}    \tae \omega\in\Omega_{m+1}    .    \llabel{K9oKe1pPrYBBZYuuiCwXzA6kaGbtwGpmRTmKviwHEzRjhTefripvLAXk3PkLNDg5odcomQj9LYIVawVmLpKrto0F6Ns7MmkcTL9Tr8fOT4uNNJvZThOQwCOCRBHRTxhSBNaIizzbKIBEcWSMYEhDkRtPWGKtUmo26acLbBnI4t2P11eRiPP99nj4qQ362UNAQaHJPPY1OgLhN8sta9eJzPgmE4zQgB0mlAWBa4Emu7mnfYgbNLzddGphhJV9hyAOGCNjxJ83Hg6CAUTnusW9pQrWv1DfVlGnWxMBbe9WwLtOdwDERmlxJ8LTqKWTtsR0cDXAfhRX1zXlAUuwzqnO2o7rtoiSMrOKLCqjoq1tUGGiIxuspoiitjaNRngtxS0r98rwXF7GNiepzEfAO2sYktIdgH1AGcRrd2w89xoOKyNnLaLRU03suU3JbS8dok8tw9NQSY4jXY625KCcPLyFRlSp759DeVbY5b69jYOmdfb99j15lvLvjskK2gEwlRxOtWLytZJ1yZ5Pit35SOiivz4F8tqMJIgQQiOobSpeprt2vBVqhvzkLlf7HXA4soMXjWdMS7LeRDiktUifLJHukestrvrl7mYcSOB7nKWMD0xBqkbxFgTTNIweyVIG6Uy3dL0C3MzFxsBE7zUhSetBQcX7jn22rr0yL1ErbpLRm3ida5MdPicdnMOiZCyGd2MdKUbxsaI9TtnHXqAQBjuN5I4Q6zz4dSWYUrhxTCuBgBUT992uczEmkqK1ouCaHJBR0QEQ107}   \end{equation} To conclude the induction step, we set $v_{0}^{(m+1)}= ( u_{0}* \sifpoierjsodfgupoefasdfgjsdfgjsdfgjsdjflxncvzxnvasdjfaopsruosihjsdkghajsdflahgfsif_{1/n_{m+1}}) \mathds{1}_{\Omega_{m+1}}$. \end{proof} \par For the rest of the paper, we fix $p_0\in(0,1)$ and $\epsilon_0\in (0, 1/2)$ to be determined, and a sequence \eqref{EQ24} satisfying the assertions of Lemma~\ref{L05}. Let $\MM_k$ be a sequence of constants such that   \begin{equation}    \sifpoierjsodfgupoefasdfgjsdfgjsdfgjsdjflxncvzxnvasdjfaopsruosihjsdgghajsdflahgfsif   v_{0}^{(k)}\sifpoierjsodfgupoefasdfgjsdfgjsdfgjsdjflxncvzxnvasdjfaopsruosihjsdgghajsdflahgfsif_{\LLLis}    \leq \MM_k     \comma k=0,1,2,3,\ldots     .    \label{EQ31}   \end{equation} Our solution shall be constructed as a sum of solutions to Navier-Stokes-like systems, with the initial data~$v_{0}^{(k)}$, for $k=0,1,2,\ldots$, which we present next. Setting $\umo=0$ and   \begin{equation}    \uk = \vz + \cdots + \vk     \comma k\in\NNz    ,    \label{EQ33}   \end{equation} we intend to solve, iteratively for $k\in\NNz$,   \begin{align}     \begin{split}        &\partial_t \vk  -\Delta \vk      \\&\indeq        =          -    \mathcal{P}\bigl( \vk\cdot \nabla \vk\bigr)          -    \mathcal{P}\bigl( \ukm\cdot \nabla \vk\bigr)          -    \mathcal{P}\bigl( \vk\cdot \nabla \ukm\bigr)          \\&\indeq\indeq          +    \sigma(t, \uk( t,x))\dot{\WW}(t)          -    \sigma(t, \ukm( t,x))\dot{\WW}(t),      \\&   \nabla\cdot v^{(k)}( t,x) = 0          ,           \\&       v^{(k)}( 0,x)=v_{0}^{(k)} (x)  \Pas    \commaone x\in\TT^3   .   \end{split}    \label{EQ32}   \end{align} The reason for addressing this system of SPDEs is that the sum $\uk$ solves the SNSE with the initial data $v_{0}^{(0)} + \cdots + v_{0}^{(k)}$; see the proof of the main theorem below. In order to solve \eqref{EQ32}, we however use smooth cutoff functions and instead first consider, for a fixed $k\in\NNz$, \begin{align}     \begin{split}       &\partial_t \vk  -\Delta \vk      \\&\indeq        =          -          \sifpoierjsodfgupoefasdfgjsdfgjsdfgjsdjflxncvzxnvasdjfaopsruosihjsdlghajsdflahgfsif_k^2 \sifpoierjsodfgupoefasdfgjsdfgjsdfgjsdjflxncvzxnvasdjfaopsruosihjsdkghajsdflahgfsif_k^2          \mathcal{P}\bigl( \vk\cdot \nabla \vk\bigr)          -          \sifpoierjsodfgupoefasdfgjsdfgjsdfgjsdjflxncvzxnvasdjfaopsruosihjsdlghajsdflahgfsif_k^2 \sifpoierjsodfgupoefasdfgjsdfgjsdfgjsdjflxncvzxnvasdjfaopsruosihjsdkghajsdflahgfsif_k^2 \zeta_{k-1}    \mathcal{P}\bigl( \ukm\cdot \nabla \vk\bigr)          -          \sifpoierjsodfgupoefasdfgjsdfgjsdfgjsdjflxncvzxnvasdjfaopsruosihjsdlghajsdflahgfsif_k^2 \sifpoierjsodfgupoefasdfgjsdfgjsdfgjsdjflxncvzxnvasdjfaopsruosihjsdkghajsdflahgfsif_k^2 \zeta_{k-1}    \mathcal{P}\bigl( \vk\cdot \nabla \ukm\bigr)        \\&\indeq\indeq          +          \sifpoierjsodfgupoefasdfgjsdfgjsdfgjsdjflxncvzxnvasdjfaopsruosihjsdlghajsdflahgfsif_k^2 \sifpoierjsodfgupoefasdfgjsdfgjsdfgjsdjflxncvzxnvasdjfaopsruosihjsdkghajsdflahgfsif_k^2 \zeta_{k-1}    \sigma(t, \uk( t,x))    \dot{\WW}(t)          -         \sifpoierjsodfgupoefasdfgjsdfgjsdfgjsdjflxncvzxnvasdjfaopsruosihjsdlghajsdflahgfsif_k^2 \sifpoierjsodfgupoefasdfgjsdfgjsdfgjsdjflxncvzxnvasdjfaopsruosihjsdkghajsdflahgfsif_k^2 \zeta_{k-1}    \sigma(t, \ukm( t,x))    \dot{\WW}(t),      \\&   \nabla\cdot v^{(k)}( t,x) = 0          ,           \\&       v^{(k)}( 0,x)=v_{0}^{(k)} (x)  \Pas    \commaone x\in\TT^3  ,   \end{split}    \label{EQ34}   \end{align} where $\{u^{(k)}\}_{k\in\NNp}$ are as in~\eqref{EQ33}, while $\{ \sifpoierjsodfgupoefasdfgjsdfgjsdfgjsdjflxncvzxnvasdjfaopsruosihjsdlghajsdflahgfsif_k\}_{k\in\NNp}$ and $\{ \sifpoierjsodfgupoefasdfgjsdfgjsdfgjsdjflxncvzxnvasdjfaopsruosihjsdkghajsdflahgfsif_k\}_{k\in\NNp}$ in \eqref{EQ34} are defined using a smooth function, $\theta\colon[0,\infty)\to [0,1]$, such that $\theta\equiv 1$ on $[0,1]$ and  $\theta\equiv 0$ on~$[2,\infty)$. To be specific,   \begin{equation}    \sifpoierjsodfgupoefasdfgjsdfgjsdfgjsdjflxncvzxnvasdjfaopsruosihjsdlghajsdflahgfsif_k    :=    \theta\left (\frac{\sifpoierjsodfgupoefasdfgjsdfgjsdfgjsdjflxncvzxnvasdjfaopsruosihjsdgghajsdflahgfsif v^{(k)}(t, \cdot)\sifpoierjsodfgupoefasdfgjsdfgjsdfgjsdjflxncvzxnvasdjfaopsruosihjsdgghajsdflahgfsif_6}{M_k}\right)    \comma    \zeta_{k-1}    :=\mathds{1}_{k=0}+\mathds{1}_{k>0}\Pi_{i=0}^{k-1}\sifpoierjsodfgupoefasdfgjsdfgjsdfgjsdjflxncvzxnvasdjfaopsruosihjsdlghajsdflahgfsif_i    \label{EQ35}   \end{equation} and   \begin{equation}    \sifpoierjsodfgupoefasdfgjsdfgjsdfgjsdjflxncvzxnvasdjfaopsruosihjsdkghajsdflahgfsif_k    :=    \theta\left (\frac{2^{k}\sifpoierjsodfgupoefasdfgjsdfgjsdfgjsdjflxncvzxnvasdjfaopsruosihjsdgghajsdflahgfsif v^{(k)}(t, \cdot)\sifpoierjsodfgupoefasdfgjsdfgjsdfgjsdjflxncvzxnvasdjfaopsruosihjsdgghajsdflahgfsif_3}{\bar{\epsilon}}\right)    ,    \label{EQ36}   \end{equation} where $\bar{\epsilon}\in (2\epsilon_0, 1)$, and $\{M_k\}_{k\in\NNp}$
is an increasing sequence of positive numbers. We shall clarify how large $M_k$ needs to be and the relationship between $\bar{\epsilon}$ and $\epsilon_0$ (see the definition in Lemma~\ref{L05}) further below.  \par For every $T>0$ and $k\in\NNz$, we claim the following. \par \cole \begin{Lemma}[An $L^{6}$ solution] \label{L06} Let $T>0$ and $k\in\mathbb{N}_0$ be fixed, let $\vk_0$ be as above, and assume that $v^{(0)}, v^{(1)}, \ldots, v^{(k-1)}$ are given. Then~\eqref{EQ34} has a unique strong solution $\vk\in L^6(\Omega; C([0,T], L^6))$, which satisfies   \begin{align}   \begin{split}   &\sadklfjsdfgsdfgsdfgsdfgdsfgsdfgadfasdf\biggl[\sup_{0\leq t\leq T}\sifpoierjsodfgupoefasdfgjsdfgjsdfgjsdjflxncvzxnvasdjfaopsruosihjsdgghajsdflahgfsif\vk(t,\cdot)\sifpoierjsodfgupoefasdfgjsdfgjsdfgjsdjflxncvzxnvasdjfaopsruosihjsdgghajsdflahgfsif_6^6+\sifpoierjsodfgupoefasdfgjsdfgjsdfgjsdjflxncvzxnvasdjfaopsruosihjsdfghajsdflahgfsif_0^{T}\sum_{i=1}^{3} \sifpoierjsodfgupoefasdfgjsdfgjsdfgjsdjflxncvzxnvasdjfaopsruosihjsdfghajsdflahgfsif_{\TT^3} | \nabla (|\vk_i(t,x)|^{3})|^2 \,dx dt\biggr]   \leq C   \sadklfjsdfgsdfgsdfgsdfgdsfgsdfgadfasdf\biggl[   \sifpoierjsodfgupoefasdfgjsdfgjsdfgjsdjflxncvzxnvasdjfaopsruosihjsdgghajsdflahgfsif\vk_0\sifpoierjsodfgupoefasdfgjsdfgjsdfgjsdjflxncvzxnvasdjfaopsruosihjsdgghajsdflahgfsif_6^6   \biggr]+C   ,   \end{split}   \label{EQ37}   \end{align} where the constant $C$ depends on $T$, $k$, and $\bar{\epsilon}$ (see \eqref{EQ36}).  \end{Lemma} \colb \par \begin{proof}[Proof of Lemma~\ref{L06}] Fix $T$ and $k$ as in the statement, and consider the iterative scheme   \begin{align}   \begin{split}   &\partial_t \Vj  -\Delta \Vj   \\&\indeq   =   -   \sifpoierjsodfgupoefasdfgjsdfgjsdfgjsdjflxncvzxnvasdjfaopsruosihjsdlghajsdflahgfsif_j\sifpoierjsodfgupoefasdfgjsdfgjsdfgjsdjflxncvzxnvasdjfaopsruosihjsdlghajsdflahgfsif_{j-1}\sifpoierjsodfgupoefasdfgjsdfgjsdfgjsdjflxncvzxnvasdjfaopsruosihjsdkghajsdflahgfsif_j\sifpoierjsodfgupoefasdfgjsdfgjsdfgjsdjflxncvzxnvasdjfaopsruosihjsdkghajsdflahgfsif_{j-1}   \mathcal{P}\bigl( \Vjm\cdot \nabla \Vjm\bigr)   -   \sifpoierjsodfgupoefasdfgjsdfgjsdfgjsdjflxncvzxnvasdjfaopsruosihjsdlghajsdflahgfsif_j\sifpoierjsodfgupoefasdfgjsdfgjsdfgjsdjflxncvzxnvasdjfaopsruosihjsdlghajsdflahgfsif_{j-1}\sifpoierjsodfgupoefasdfgjsdfgjsdfgjsdjflxncvzxnvasdjfaopsruosihjsdkghajsdflahgfsif_j\sifpoierjsodfgupoefasdfgjsdfgjsdfgjsdjflxncvzxnvasdjfaopsruosihjsdkghajsdflahgfsif_{j-1}\zeta   \mathcal{P}\bigl( w\cdot \nabla \Vjm\bigr)          \\&\indeq\indeq   -   \sifpoierjsodfgupoefasdfgjsdfgjsdfgjsdjflxncvzxnvasdjfaopsruosihjsdlghajsdflahgfsif_j\sifpoierjsodfgupoefasdfgjsdfgjsdfgjsdjflxncvzxnvasdjfaopsruosihjsdlghajsdflahgfsif_{j-1}\sifpoierjsodfgupoefasdfgjsdfgjsdfgjsdjflxncvzxnvasdjfaopsruosihjsdkghajsdflahgfsif_j\sifpoierjsodfgupoefasdfgjsdfgjsdfgjsdjflxncvzxnvasdjfaopsruosihjsdkghajsdflahgfsif_{j-1}\zeta   \mathcal{P}\bigl( \Vjm\cdot \nabla w\bigr)   \\&\indeq\indeq   +   \sifpoierjsodfgupoefasdfgjsdfgjsdfgjsdjflxncvzxnvasdjfaopsruosihjsdlghajsdflahgfsif_j\sifpoierjsodfgupoefasdfgjsdfgjsdfgjsdjflxncvzxnvasdjfaopsruosihjsdlghajsdflahgfsif_{j-1}\sifpoierjsodfgupoefasdfgjsdfgjsdfgjsdjflxncvzxnvasdjfaopsruosihjsdkghajsdflahgfsif_j\sifpoierjsodfgupoefasdfgjsdfgjsdfgjsdjflxncvzxnvasdjfaopsruosihjsdkghajsdflahgfsif_{j-1}\zeta   \Bigl(   \sigma(t, \Vjm( t,x)+w(t,x))   -   \sigma(t, w( t,x))   \Bigr)   \dot{\WW}(t),   \\&   \nabla\cdot V^{(j)}( t,x) = 0,   \\&   V^{(j)}( 0,x)=v_{0}^{(k)} (x)  \Pas   \commaone x\in\TT^3\commaone j=0,1,2,\ldots,   \end{split}   \label{EQ38}   \end{align} with $V^{(-1)}(t)\equiv 0$. Note that we have replaced $\ukm$ by~$w$, abbreviated $\zeta_{k-1}$ as $\zeta$, and denoted   \begin{equation}   \sifpoierjsodfgupoefasdfgjsdfgjsdfgjsdjflxncvzxnvasdjfaopsruosihjsdlghajsdflahgfsif_i   =   \theta\left (\frac{\sifpoierjsodfgupoefasdfgjsdfgjsdfgjsdjflxncvzxnvasdjfaopsruosihjsdgghajsdflahgfsif V^{(i)}(t, \cdot)\sifpoierjsodfgupoefasdfgjsdfgjsdfgjsdjflxncvzxnvasdjfaopsruosihjsdgghajsdflahgfsif_6}{M_k}\right)   \comma   \sifpoierjsodfgupoefasdfgjsdfgjsdfgjsdjflxncvzxnvasdjfaopsruosihjsdkghajsdflahgfsif_{i}   =   \theta\left (\frac{2^{k}\sifpoierjsodfgupoefasdfgjsdfgjsdfgjsdjflxncvzxnvasdjfaopsruosihjsdgghajsdflahgfsif V^{(i)}(t, \cdot)\sifpoierjsodfgupoefasdfgjsdfgjsdfgjsdjflxncvzxnvasdjfaopsruosihjsdgghajsdflahgfsif_3}{\bar{\epsilon}}\right)   \comma i\in\{j-1, j\}.    \llabel{nv1artFiekBu49ND9kK9eKBOgPGzqfKJ67NsKz3BywIwYxEoWYf6AKuyVPj8B9D6quBkFCsKHUDCksDYK3vs0Ep3gM2EwlPGjRVX6cxlbVOfAll7g6yL9PWyo58h0e07HO0qz8kbe85ZBVCYOKxNNLa4aFZ7mw7moACU1q1lpfmE5qXTA0QqVMnRsbKzHo5vX1tpMVZXCznmSOM73CRHwQPTlvVN7lKXI06KT6MTjO3Yb87pgozoxydVJHPL3k2KRyx3b0yPBsJmNjETPJi4km2fxMh35MtRoirNE9bU7lMo4bnj9GgYA6vsEsONRtNmDFJej96STn3lJU2u16oTEXogvMqwhD0BKr1CisVYbA2wkfX0n4hD5Lbr8l7ErfuN8OcUjqeqzCCyx6hPAyMrLeB8CwlkThixdIzviEWuwI8qKa0VZEqOroDUPGphfIOFSKZ3icda7Vh3ywUSzkkW8SfU1yHN0A14znyPULl6hpzlkq7SKNaFqg9Yhj2hJ3pWSmi9XgjapmMZ6HV8yjigpSNlI9T8eLhc1eRRgZ885eNJ8w3secl5ilCdozV1BoOIk9gDZNY5qgVQcFeTDVxhPmwPhEU41Lq35gCzPtc2oPugVKOp5Gsf7DFBlektobd2yuDtElXxmj1usDJJ6hj0HBVFanTvabFAVwM51nUH60GvT9fAjTO4MQVzNNAQiwSlSxf2pQ8qvtdjnvupLATIwym4nEYESfMavUgZoyehtoe9RTN15EI1aKJSCnr4MjiYhBEQ39}   \end{equation} Also observe that $\zeta$ provides a uniform control in $\Omega\times[0,T]$ on the $L^6$-norm of~$w$. With a slight variation of the approach in \cite[Lemmas 5.4--5.5]{KXZ}, we can show that \eqref{EQ38} has a unique strong $L^6$ solution $V^{(j)}$ with a continuous modification on a common time interval $[0,t^{\ast}]$ for all $j\in \NNz$. Moreover, \eqref{EQ37} holds for $V^{(j)}$ up to~$t^{\ast}$. These results follow from Lemma~\ref{L01} and a fixed-point argument: We conduct iteration on the product~$\sifpoierjsodfgupoefasdfgjsdfgjsdfgjsdjflxncvzxnvasdjfaopsruosihjsdlghajsdflahgfsif_j\sifpoierjsodfgupoefasdfgjsdfgjsdfgjsdjflxncvzxnvasdjfaopsruosihjsdkghajsdflahgfsif_j$. Utilizing Lemma~\ref{L01} (with $p=6$ and $q=3$), we obtain the existence of a global iterative solution satisfying \eqref{EQ37} uniformly. Then, we employ Lemma~\ref{L01} to the equation for the difference of two consecutive iterates, obtaining an exponential rate of convergence in $L^6_{\omega}L^{\infty}_tL^6_{x}$ on the interval $[0,t^{\ast}]$ for a small but deterministic $t^{\ast}>0$. This allows us to pass to the limit and prove that the fixed-point $V^{(j)}$ solves \eqref{EQ38} on $[0,t^{\ast}]$, as the spatial domain is~$\TT^3$. The time of existence $t^{\ast}$ may depend on $k$ and $\bar{\epsilon}$ but is uniform with respect to~$j$. Then we pass to the limit in \eqref{EQ37} and prove that it holds for every $V^{(j)}$ uniformly up to the time $t^{\ast}$ by Lemma~\ref{L03}.    Next, as in the proof of \cite[Theorem 5.1]{KXZ}, we establish the convergence of $\{V^{(j)}\}_{j\in\NNp}$ in $L^6_{\omega}L^{\infty}_tL^6_{x}$. Denote \begin{equation} \zj:=\Vjp-\Vj , \llabel{0A7vnSAYnZ1cXOI1V7yja0R9jCTwxMUiMI5l2sTXnNRnVi1KczLG3MgJoEktlKoU13tsaqjrHYVzfb1yyxunpbRA56brW45IqhfKo0zj04IcGrHirwyH2tJbFr3leRdcpstvXe2yJlekGVFCe2aD4XPOuImtVoazCKO3uRIm2KFjtm5RGWCvkozi75YWNsbhORnxzRzw99TrFhjhKbfqLAbe2v5n9mD2VpNzlMnntoiFZB2ZjXBhhsK8K6cGiSbRkkwfWeYJXdRBBxyqjEVF5lr3dFrxGlTcsbyAENcqA981IQ4UGpBk0gBeJ6Dn9Jhkne5f518umOuLnIaspzcRfoC0StSy0DF8NNzF2UpPtNG50tqKTk2e51yUbrsznQbeIuiY5qaSGjcXiEl45B5PnyQtnUOMHiskTC2KsWkjha6loMfgZKG3nHph0gnNQ7q0QxsQkgQwKwyhfP5qFWwNaHxSKTA63ClhGBgarujHnGKf46FQtVtSPgEgTeY6fJGmB3qgXxtR8RTCPB18kQajtt6GDrKb1VYLV3RgWIrAyZf69V8VM7jHOb7zLvaXTTVI0ONKMBAHOwOZ7dPkyCgUS74HlnFZMHabr8mlHbQNSwwdomOL6q5wvRexVejvVHkCEdXm3cU54juZSKng8wcj6hR1FnZJbkmgKXJgFm5qZ5SubXvPKDBOCGf4srh1a5FL0vYfRjJwUm2sfCogRhabxyc0RgavaRbkjzlteRGExbEMMhLZbh3axosCqu7kZ1Pt6YEQ40} \end{equation} which satisfies \begin{align} \begin{split} &\partial_t\zj -\Delta \zj +  \sifpoierjsodfgupoefasdfgjsdfgjsdfgjsdjflxncvzxnvasdjfaopsruosihjsdlghajsdflahgfsif_{j+1}\sifpoierjsodfgupoefasdfgjsdfgjsdfgjsdjflxncvzxnvasdjfaopsruosihjsdlghajsdflahgfsif_{j} \sifpoierjsodfgupoefasdfgjsdfgjsdfgjsdjflxncvzxnvasdjfaopsruosihjsdkghajsdflahgfsif_{j+1}\sifpoierjsodfgupoefasdfgjsdfgjsdfgjsdjflxncvzxnvasdjfaopsruosihjsdkghajsdflahgfsif_{j}\mathcal{P}\bigl(( \Vj \cdot \nabla)\Vj \bigr)-\sifpoierjsodfgupoefasdfgjsdfgjsdfgjsdjflxncvzxnvasdjfaopsruosihjsdlghajsdflahgfsif_{j}\sifpoierjsodfgupoefasdfgjsdfgjsdfgjsdjflxncvzxnvasdjfaopsruosihjsdlghajsdflahgfsif_{j-1}\sifpoierjsodfgupoefasdfgjsdfgjsdfgjsdjflxncvzxnvasdjfaopsruosihjsdkghajsdflahgfsif_{j}\sifpoierjsodfgupoefasdfgjsdfgjsdfgjsdjflxncvzxnvasdjfaopsruosihjsdkghajsdflahgfsif_{j-1} \mathcal{P} \bigl( ( \Vjm \cdot \nabla)\Vjm \bigr) \\&\indeq\indeq + \sifpoierjsodfgupoefasdfgjsdfgjsdfgjsdjflxncvzxnvasdjfaopsruosihjsdlghajsdflahgfsif_{j+1}\sifpoierjsodfgupoefasdfgjsdfgjsdfgjsdjflxncvzxnvasdjfaopsruosihjsdlghajsdflahgfsif_{j}\sifpoierjsodfgupoefasdfgjsdfgjsdfgjsdjflxncvzxnvasdjfaopsruosihjsdkghajsdflahgfsif_{j+1}\sifpoierjsodfgupoefasdfgjsdfgjsdfgjsdjflxncvzxnvasdjfaopsruosihjsdkghajsdflahgfsif_{j} \zeta \mathcal{P}\bigl( w\cdot \nabla \Vj\bigr)- \sifpoierjsodfgupoefasdfgjsdfgjsdfgjsdjflxncvzxnvasdjfaopsruosihjsdlghajsdflahgfsif_j\sifpoierjsodfgupoefasdfgjsdfgjsdfgjsdjflxncvzxnvasdjfaopsruosihjsdlghajsdflahgfsif_{j-1}\sifpoierjsodfgupoefasdfgjsdfgjsdfgjsdjflxncvzxnvasdjfaopsruosihjsdkghajsdflahgfsif_{j}\sifpoierjsodfgupoefasdfgjsdfgjsdfgjsdjflxncvzxnvasdjfaopsruosihjsdkghajsdflahgfsif_{j-1} \zeta \mathcal{P}\bigl( w\cdot \nabla \Vjm\bigr) \\&\indeq\indeq +\sifpoierjsodfgupoefasdfgjsdfgjsdfgjsdjflxncvzxnvasdjfaopsruosihjsdlghajsdflahgfsif_{j+1}\sifpoierjsodfgupoefasdfgjsdfgjsdfgjsdjflxncvzxnvasdjfaopsruosihjsdlghajsdflahgfsif_{j}\sifpoierjsodfgupoefasdfgjsdfgjsdfgjsdjflxncvzxnvasdjfaopsruosihjsdkghajsdflahgfsif_{j+1}\sifpoierjsodfgupoefasdfgjsdfgjsdfgjsdjflxncvzxnvasdjfaopsruosihjsdkghajsdflahgfsif_{j} \zeta \mathcal{P}\bigl( \Vj\cdot \nabla w\bigr) - \sifpoierjsodfgupoefasdfgjsdfgjsdfgjsdjflxncvzxnvasdjfaopsruosihjsdlghajsdflahgfsif_j\sifpoierjsodfgupoefasdfgjsdfgjsdfgjsdjflxncvzxnvasdjfaopsruosihjsdlghajsdflahgfsif_{j-1}\sifpoierjsodfgupoefasdfgjsdfgjsdfgjsdjflxncvzxnvasdjfaopsruosihjsdkghajsdflahgfsif_{j}\sifpoierjsodfgupoefasdfgjsdfgjsdfgjsdjflxncvzxnvasdjfaopsruosihjsdkghajsdflahgfsif_{j-1} \zeta \mathcal{P}\bigl( \Vjm\cdot \nabla w\bigr) \\&\indeq =\Bigl( \sifpoierjsodfgupoefasdfgjsdfgjsdfgjsdjflxncvzxnvasdjfaopsruosihjsdlghajsdflahgfsif_{j+1}\sifpoierjsodfgupoefasdfgjsdfgjsdfgjsdjflxncvzxnvasdjfaopsruosihjsdlghajsdflahgfsif_{j}\sifpoierjsodfgupoefasdfgjsdfgjsdfgjsdjflxncvzxnvasdjfaopsruosihjsdkghajsdflahgfsif_{j+1}\sifpoierjsodfgupoefasdfgjsdfgjsdfgjsdjflxncvzxnvasdjfaopsruosihjsdkghajsdflahgfsif_{j} \zeta \big( \sigma(t, \Vj+w) - \sigma(t, w) \big) - \sifpoierjsodfgupoefasdfgjsdfgjsdfgjsdjflxncvzxnvasdjfaopsruosihjsdlghajsdflahgfsif_j\sifpoierjsodfgupoefasdfgjsdfgjsdfgjsdjflxncvzxnvasdjfaopsruosihjsdlghajsdflahgfsif_{j-1}\sifpoierjsodfgupoefasdfgjsdfgjsdfgjsdjflxncvzxnvasdjfaopsruosihjsdkghajsdflahgfsif_{j}\sifpoierjsodfgupoefasdfgjsdfgjsdfgjsdjflxncvzxnvasdjfaopsruosihjsdkghajsdflahgfsif_{j-1}\zeta \big( \sigma(t, \Vjm+w) - \sigma(t, w) \big) \Bigr)\dot{\WW}(t), \\ & \nabla\cdot \zj = 0, \\& \zj( 0)= {0}   \Pas \comma t\in [0,t^{\ast}] \commaone j=0,1,2,\ldots. \end{split} \llabel{8zJXtvmvPvAr3LSWDjbVPN7eNu20r8Bw2ivnkzMda93zWWiUBHwQzahUiji2TrXI8v2HNShbTKLeKW83WrQKO4TZm57yzoVYZJytSg2Wx4YafTHAxS7kacIPQJGYdDk0531u2QIKfREWYcMKMUT7fdT9EkIfUJ3pMW59QLFmu02YHJaa2Er6KSIwTBGDJYZwvfSJQby7fdFWdfT9zU27ws5oU5MUTDJzKFNojdXRyBaYybTvnhh2dV77oFFlt4H0RNZjVJ5BJpyIqAOWWcefdR27nGkjmoEFHjanXf1ONEcytoINtD90ONandawDRKi2DJzAqYHGCTB0pzdBa3OotPq1QVFvaYNTVz2sZJ6eyIg2N7PgilKLF9NzcrhuLeCeXwb6cMFExflJSE8Ev9WHgQ1Brp7ROMACwvAnATqGZHwkdHA5fbABXo6EWHsoW6HQYvvjcZgRkOWAbVA0zBfBaWwlIV05Z6E2JQjOeHcZGJuq90ac5Jh9h0rLKfIHtl8tPrtRdqql8TZGUgdNySBHoNrQCsxtgzuGAwHvyNxpMmwKQuJFKjtZr6Y4HdmrCbnF52gA0328aVuzEbplXZd7EJEEC939HQthaMsupTcxVaZ32pPdbPIj2x8AzxjYXSq8LsofqmgSqjm8G4wUbQ28LuAabwI0cFWNfGnzpVzsUeHsL9zoBLlg5jXQXnR0giRmCLErqlDIPYeYXduUJE0BsbkKbjpdcPLiek8NWrIjsfapHh4GYvMFbA67qyex7sHgHGEQ41} \end{align} Componentwise, we may rewrite the first equation as \begin{align} \begin{split} &\partial_t \zj_m -\Delta \zj_m = \sum_{i}\partial_{i} f_{im}  + g_m \dot{\WW}(t) \comma m=1,2,3 , \end{split} \llabel{3GlW0y1WD35mIo5gEUbObrbknjgUQyko7g2yrEOfovQfAk6UVDHGl7GV3LvQmradEUOJpuuztBBnrmefilt1sGSf5O0aw2Dc0hRaHGalEqIpfgPyNQoLHp2LAIUp77FygrjC8qBbuxBkYX8NTmUvyT7YnBgv5K7vq5NefB5ye4TMuCfmE2JF7hgqwI7dmNx2CqZuLFthzIlB1sjKA8WGDKcDKvabk9yp28TFP0rg0iA9CBD36c8HLkZnO2S6ZoafvLXb8gopYa085EMRbAbQjGturIXlTE0Gz0tYSVUseCjDvrQ2bvfiIJCdfCAcWyIO7mlycs5RjioIZt7qyB7pL9pyG8XDTzJxHs0yhVVAr8ZQRqsZCHHADFTwvJHeHOGvLJHuTfNa5j12ZkTvGqOyS8826D2rj7rHDTLN7Ggmt9Mzcygwxnj4JJeQb7eMmwRnSuZLU8qUNDLrdgC70bhEPgpb7zk5a32N1IbJhf8XvGRmUFdvIUkwPFbidJPLlNGe1RQRsK2dVNPM7A3YhdhB1R6N5MJi5S4R498lwY9I8RHxQKLlAk8W3Ts7WFUoNwI9KWnztPxrZLvNwZ28EYOnoufxz6ip9aSWnNQASriwYC1sOtSqXzot8k4KOz78LG6GMNCExoMh9wl5vbsmnnq6Hg6WToJun74JxyNBXyVpvxNB0N8wymK3reReEzFxbK92xELs950SNgLmviRC1bFHjDCke3SgtUdC4cONb4EF24D1VDBHlWATyswjyDOWibTHqXEQ42} \end{align} where  \begin{align} \begin{split} f_{im} &=  -\sifpoierjsodfgupoefasdfgjsdfgjsdfgjsdjflxncvzxnvasdjfaopsruosihjsdlghajsdflahgfsif_{j+1}\sifpoierjsodfgupoefasdfgjsdfgjsdfgjsdjflxncvzxnvasdjfaopsruosihjsdlghajsdflahgfsif_{j}\sifpoierjsodfgupoefasdfgjsdfgjsdfgjsdjflxncvzxnvasdjfaopsruosihjsdkghajsdflahgfsif_{j+1}\sifpoierjsodfgupoefasdfgjsdfgjsdfgjsdjflxncvzxnvasdjfaopsruosihjsdkghajsdflahgfsif_{j} \bigl(\mathcal{P}( \Vj_i \Vj) \bigr)_m+\sifpoierjsodfgupoefasdfgjsdfgjsdfgjsdjflxncvzxnvasdjfaopsruosihjsdlghajsdflahgfsif_{j}\sifpoierjsodfgupoefasdfgjsdfgjsdfgjsdjflxncvzxnvasdjfaopsruosihjsdlghajsdflahgfsif_{j-1} \sifpoierjsodfgupoefasdfgjsdfgjsdfgjsdjflxncvzxnvasdjfaopsruosihjsdkghajsdflahgfsif_{j}\sifpoierjsodfgupoefasdfgjsdfgjsdfgjsdjflxncvzxnvasdjfaopsruosihjsdkghajsdflahgfsif_{j-1} \bigl(\mathcal{P} ( \Vjm_i \Vjm) \bigr)_m \\&\indeq - \sifpoierjsodfgupoefasdfgjsdfgjsdfgjsdjflxncvzxnvasdjfaopsruosihjsdlghajsdflahgfsif_{j+1}\sifpoierjsodfgupoefasdfgjsdfgjsdfgjsdjflxncvzxnvasdjfaopsruosihjsdlghajsdflahgfsif_{j}\sifpoierjsodfgupoefasdfgjsdfgjsdfgjsdjflxncvzxnvasdjfaopsruosihjsdkghajsdflahgfsif_{j+1}\sifpoierjsodfgupoefasdfgjsdfgjsdfgjsdjflxncvzxnvasdjfaopsruosihjsdkghajsdflahgfsif_{j} \zeta \bigl(\mathcal{P} (w_i \Vj)\bigr)_m + \sifpoierjsodfgupoefasdfgjsdfgjsdfgjsdjflxncvzxnvasdjfaopsruosihjsdlghajsdflahgfsif_j\sifpoierjsodfgupoefasdfgjsdfgjsdfgjsdjflxncvzxnvasdjfaopsruosihjsdlghajsdflahgfsif_{j-1}\sifpoierjsodfgupoefasdfgjsdfgjsdfgjsdjflxncvzxnvasdjfaopsruosihjsdkghajsdflahgfsif_{j}\sifpoierjsodfgupoefasdfgjsdfgjsdfgjsdjflxncvzxnvasdjfaopsruosihjsdkghajsdflahgfsif_{j-1} \zeta \bigl( \mathcal{P}(w_i \Vjm)\bigr)_m \\&\indeq -\sifpoierjsodfgupoefasdfgjsdfgjsdfgjsdjflxncvzxnvasdjfaopsruosihjsdlghajsdflahgfsif_{j+1}\sifpoierjsodfgupoefasdfgjsdfgjsdfgjsdjflxncvzxnvasdjfaopsruosihjsdlghajsdflahgfsif_{j}\sifpoierjsodfgupoefasdfgjsdfgjsdfgjsdjflxncvzxnvasdjfaopsruosihjsdkghajsdflahgfsif_{j+1}\sifpoierjsodfgupoefasdfgjsdfgjsdfgjsdjflxncvzxnvasdjfaopsruosihjsdkghajsdflahgfsif_{j} \zeta \bigl( \mathcal{P}(\Vj_i w)\bigr)_m + \sifpoierjsodfgupoefasdfgjsdfgjsdfgjsdjflxncvzxnvasdjfaopsruosihjsdlghajsdflahgfsif_j\sifpoierjsodfgupoefasdfgjsdfgjsdfgjsdjflxncvzxnvasdjfaopsruosihjsdlghajsdflahgfsif_{j-1}\sifpoierjsodfgupoefasdfgjsdfgjsdfgjsdjflxncvzxnvasdjfaopsruosihjsdkghajsdflahgfsif_{j}\sifpoierjsodfgupoefasdfgjsdfgjsdfgjsdjflxncvzxnvasdjfaopsruosihjsdkghajsdflahgfsif_{j-1} \zeta \bigl(\mathcal{P} (\Vjm_i w)\bigr)_m \comma m=1,2,3, \end{split}    \llabel{t3aG6mkfGJVWv40lexPnIcy5ckRMD3owVBdxQm6CvLaAgxiJtEsSlZFwDoYP2nRYbCdXRz5HboVTU8NPgNViWeXGVQZ7bjOy1LRy9faj9n2iE1S0mci0YD3HgUxzLatb92MhCpZKLJqHTSFRMn3KVkpcFLUcF0X66ivdq01cVqkoQqu1u2Cpip5EV7AgMORcfZjLx7Lcv9lXn6rS8WeK3zTLDPB61JVWwMiKEuUZZ4qiK1iQ8N0832TS4eLW4zeUyonzTSofna74RQVKiu9W3kEa3gH8xdiOhAcHsIQCsEt0Qi2IHw9vq9rNPlh1y3wORqrJcxU4i55ZHTOoGP0zEqlB3lkwGGRn7TOoKfGZu5BczGKFeoyIBtjNb8xfQEKduOnJVOZh8PUVaRonXBkIjBT9WWor7A3WfXxA2f2VlXZS1Ttsab4n6R3BKX0XJTmlkVtcWTMCsiFVyjfcrzeJk5MBxwR7zzVOnjlLzUz5uLeqWjDul7OnYICGG9iRybTsYJXfrRnub3p16JBQd0zQOkKZK6DeVgpXRceOExLY3WKrXYyIe7dqMqanCCTjFW71LQ89mQw1gAswnYSMeWlHz7ud7xBwxF3m8usa66yr0nSdsYwuqwXdD0fRjFpeLOe0rcsIuMGrSOqREW5plybq3rFrk7YmLURUSSVYGruD6ksnLXBkvVS2q0ljMPpIL27QdZMUPbaOoLqt3bhn6RX9hPAdQRp9PI4fBkJ8uILIArpTl4E6jrUYwuFXiFYaDVvrDbEQ43} \end{align} i.e., \begin{align} \begin{split} f_{im} &=  - \sifpoierjsodfgupoefasdfgjsdfgjsdfgjsdjflxncvzxnvasdjfaopsruosihjsdlghajsdflahgfsif_{j} (\sifpoierjsodfgupoefasdfgjsdfgjsdfgjsdjflxncvzxnvasdjfaopsruosihjsdlghajsdflahgfsif_{j+1}-\sifpoierjsodfgupoefasdfgjsdfgjsdfgjsdjflxncvzxnvasdjfaopsruosihjsdlghajsdflahgfsif_{j}) \sifpoierjsodfgupoefasdfgjsdfgjsdfgjsdjflxncvzxnvasdjfaopsruosihjsdkghajsdflahgfsif_{j+1}\sifpoierjsodfgupoefasdfgjsdfgjsdfgjsdjflxncvzxnvasdjfaopsruosihjsdkghajsdflahgfsif_{j}\bigl(\mathcal{P}( \Vj_i \Vj) \bigr)_m - (\sifpoierjsodfgupoefasdfgjsdfgjsdfgjsdjflxncvzxnvasdjfaopsruosihjsdlghajsdflahgfsif_{j})^2 (\sifpoierjsodfgupoefasdfgjsdfgjsdfgjsdjflxncvzxnvasdjfaopsruosihjsdkghajsdflahgfsif_{j+1}-\sifpoierjsodfgupoefasdfgjsdfgjsdfgjsdjflxncvzxnvasdjfaopsruosihjsdkghajsdflahgfsif_{j}) \sifpoierjsodfgupoefasdfgjsdfgjsdfgjsdjflxncvzxnvasdjfaopsruosihjsdkghajsdflahgfsif_{j}\bigl(\mathcal{P}( \Vj_i \Vj) \bigr)_m \\&\indeq -\sifpoierjsodfgupoefasdfgjsdfgjsdfgjsdjflxncvzxnvasdjfaopsruosihjsdlghajsdflahgfsif_{j} (\sifpoierjsodfgupoefasdfgjsdfgjsdfgjsdjflxncvzxnvasdjfaopsruosihjsdlghajsdflahgfsif_{j}-\sifpoierjsodfgupoefasdfgjsdfgjsdfgjsdjflxncvzxnvasdjfaopsruosihjsdlghajsdflahgfsif_{j-1}) (\sifpoierjsodfgupoefasdfgjsdfgjsdfgjsdjflxncvzxnvasdjfaopsruosihjsdkghajsdflahgfsif_{j})^2\bigl(\mathcal{P}( \Vj_i \Vj) \bigr)_m -\sifpoierjsodfgupoefasdfgjsdfgjsdfgjsdjflxncvzxnvasdjfaopsruosihjsdlghajsdflahgfsif_{j}\sifpoierjsodfgupoefasdfgjsdfgjsdfgjsdjflxncvzxnvasdjfaopsruosihjsdlghajsdflahgfsif_{j-1}\sifpoierjsodfgupoefasdfgjsdfgjsdfgjsdjflxncvzxnvasdjfaopsruosihjsdkghajsdflahgfsif_{j} (\sifpoierjsodfgupoefasdfgjsdfgjsdfgjsdjflxncvzxnvasdjfaopsruosihjsdkghajsdflahgfsif_{j}-\sifpoierjsodfgupoefasdfgjsdfgjsdfgjsdjflxncvzxnvasdjfaopsruosihjsdkghajsdflahgfsif_{j-1}) \bigl(\mathcal{P}( \Vj_i \Vj) \bigr)_m \\&\indeq -\sifpoierjsodfgupoefasdfgjsdfgjsdfgjsdjflxncvzxnvasdjfaopsruosihjsdlghajsdflahgfsif_{j} \sifpoierjsodfgupoefasdfgjsdfgjsdfgjsdjflxncvzxnvasdjfaopsruosihjsdlghajsdflahgfsif_{j-1}\sifpoierjsodfgupoefasdfgjsdfgjsdfgjsdjflxncvzxnvasdjfaopsruosihjsdkghajsdflahgfsif_{j}\sifpoierjsodfgupoefasdfgjsdfgjsdfgjsdjflxncvzxnvasdjfaopsruosihjsdkghajsdflahgfsif_{j-1} \bigl(\mathcal{P}( \zjm_i \Vj) \bigr)_m -\sifpoierjsodfgupoefasdfgjsdfgjsdfgjsdjflxncvzxnvasdjfaopsruosihjsdlghajsdflahgfsif_{j} \sifpoierjsodfgupoefasdfgjsdfgjsdfgjsdjflxncvzxnvasdjfaopsruosihjsdlghajsdflahgfsif_{j-1}\sifpoierjsodfgupoefasdfgjsdfgjsdfgjsdjflxncvzxnvasdjfaopsruosihjsdkghajsdflahgfsif_{j}\sifpoierjsodfgupoefasdfgjsdfgjsdfgjsdjflxncvzxnvasdjfaopsruosihjsdkghajsdflahgfsif_{j-1} \bigl(\mathcal{P}( \Vjm_i \zjm) \bigr)_m \\&\indeq - \sifpoierjsodfgupoefasdfgjsdfgjsdfgjsdjflxncvzxnvasdjfaopsruosihjsdlghajsdflahgfsif_{j} (\sifpoierjsodfgupoefasdfgjsdfgjsdfgjsdjflxncvzxnvasdjfaopsruosihjsdlghajsdflahgfsif_{j+1}-\sifpoierjsodfgupoefasdfgjsdfgjsdfgjsdjflxncvzxnvasdjfaopsruosihjsdlghajsdflahgfsif_{j})\sifpoierjsodfgupoefasdfgjsdfgjsdfgjsdjflxncvzxnvasdjfaopsruosihjsdkghajsdflahgfsif_{j+1}\sifpoierjsodfgupoefasdfgjsdfgjsdfgjsdjflxncvzxnvasdjfaopsruosihjsdkghajsdflahgfsif_{j}\zeta \bigl(\mathcal{P}( w_i \Vj) \bigr)_m - (\sifpoierjsodfgupoefasdfgjsdfgjsdfgjsdjflxncvzxnvasdjfaopsruosihjsdlghajsdflahgfsif_{j})^2 (\sifpoierjsodfgupoefasdfgjsdfgjsdfgjsdjflxncvzxnvasdjfaopsruosihjsdkghajsdflahgfsif_{j+1}-\sifpoierjsodfgupoefasdfgjsdfgjsdfgjsdjflxncvzxnvasdjfaopsruosihjsdkghajsdflahgfsif_{j})\sifpoierjsodfgupoefasdfgjsdfgjsdfgjsdjflxncvzxnvasdjfaopsruosihjsdkghajsdflahgfsif_{j}\zeta \bigl(\mathcal{P}( w_i \Vj) \bigr)_m \\&\indeq - \sifpoierjsodfgupoefasdfgjsdfgjsdfgjsdjflxncvzxnvasdjfaopsruosihjsdlghajsdflahgfsif_{j} (\sifpoierjsodfgupoefasdfgjsdfgjsdfgjsdjflxncvzxnvasdjfaopsruosihjsdlghajsdflahgfsif_{j}-\sifpoierjsodfgupoefasdfgjsdfgjsdfgjsdjflxncvzxnvasdjfaopsruosihjsdlghajsdflahgfsif_{j-1})(\sifpoierjsodfgupoefasdfgjsdfgjsdfgjsdjflxncvzxnvasdjfaopsruosihjsdkghajsdflahgfsif_{j})^2\zeta \bigl(\mathcal{P}( w_i \Vj) \bigr)_m - \sifpoierjsodfgupoefasdfgjsdfgjsdfgjsdjflxncvzxnvasdjfaopsruosihjsdlghajsdflahgfsif_{j}\sifpoierjsodfgupoefasdfgjsdfgjsdfgjsdjflxncvzxnvasdjfaopsruosihjsdlghajsdflahgfsif_{j-1} (\sifpoierjsodfgupoefasdfgjsdfgjsdfgjsdjflxncvzxnvasdjfaopsruosihjsdkghajsdflahgfsif_{j}-\sifpoierjsodfgupoefasdfgjsdfgjsdfgjsdjflxncvzxnvasdjfaopsruosihjsdkghajsdflahgfsif_{j-1})\sifpoierjsodfgupoefasdfgjsdfgjsdfgjsdjflxncvzxnvasdjfaopsruosihjsdkghajsdflahgfsif_{j}\zeta \bigl(\mathcal{P}( w_i \Vj) \bigr)_m \\&\indeq - \sifpoierjsodfgupoefasdfgjsdfgjsdfgjsdjflxncvzxnvasdjfaopsruosihjsdlghajsdflahgfsif_{j}\sifpoierjsodfgupoefasdfgjsdfgjsdfgjsdjflxncvzxnvasdjfaopsruosihjsdlghajsdflahgfsif_{j-1}\sifpoierjsodfgupoefasdfgjsdfgjsdfgjsdjflxncvzxnvasdjfaopsruosihjsdkghajsdflahgfsif_{j}\sifpoierjsodfgupoefasdfgjsdfgjsdfgjsdjflxncvzxnvasdjfaopsruosihjsdkghajsdflahgfsif_{j-1}\zeta\bigl(\mathcal{P}( w_i \zjm) \bigr)_m - \sifpoierjsodfgupoefasdfgjsdfgjsdfgjsdjflxncvzxnvasdjfaopsruosihjsdlghajsdflahgfsif_{j} (\sifpoierjsodfgupoefasdfgjsdfgjsdfgjsdjflxncvzxnvasdjfaopsruosihjsdlghajsdflahgfsif_{j+1}-\sifpoierjsodfgupoefasdfgjsdfgjsdfgjsdjflxncvzxnvasdjfaopsruosihjsdlghajsdflahgfsif_{j})\sifpoierjsodfgupoefasdfgjsdfgjsdfgjsdjflxncvzxnvasdjfaopsruosihjsdkghajsdflahgfsif_{j+1}\sifpoierjsodfgupoefasdfgjsdfgjsdfgjsdjflxncvzxnvasdjfaopsruosihjsdkghajsdflahgfsif_{j}\zeta \bigl(\mathcal{P}( \Vj_i w) \bigr)_m \\&\indeq -(\sifpoierjsodfgupoefasdfgjsdfgjsdfgjsdjflxncvzxnvasdjfaopsruosihjsdlghajsdflahgfsif_{j})^2 (\sifpoierjsodfgupoefasdfgjsdfgjsdfgjsdjflxncvzxnvasdjfaopsruosihjsdkghajsdflahgfsif_{j+1}-\sifpoierjsodfgupoefasdfgjsdfgjsdfgjsdjflxncvzxnvasdjfaopsruosihjsdkghajsdflahgfsif_{j})\sifpoierjsodfgupoefasdfgjsdfgjsdfgjsdjflxncvzxnvasdjfaopsruosihjsdkghajsdflahgfsif_{j}\zeta  \bigl(\mathcal{P}( \Vj_i w) \bigr)_m -\sifpoierjsodfgupoefasdfgjsdfgjsdfgjsdjflxncvzxnvasdjfaopsruosihjsdlghajsdflahgfsif_{j} (\sifpoierjsodfgupoefasdfgjsdfgjsdfgjsdjflxncvzxnvasdjfaopsruosihjsdlghajsdflahgfsif_{j}-\sifpoierjsodfgupoefasdfgjsdfgjsdfgjsdjflxncvzxnvasdjfaopsruosihjsdlghajsdflahgfsif_{j-1})(\sifpoierjsodfgupoefasdfgjsdfgjsdfgjsdjflxncvzxnvasdjfaopsruosihjsdkghajsdflahgfsif_{j})^2 \zeta \bigl(\mathcal{P}( \Vj_i w) \bigr)_m \\&\indeq -\sifpoierjsodfgupoefasdfgjsdfgjsdfgjsdjflxncvzxnvasdjfaopsruosihjsdlghajsdflahgfsif_{j}\sifpoierjsodfgupoefasdfgjsdfgjsdfgjsdjflxncvzxnvasdjfaopsruosihjsdlghajsdflahgfsif_{j-1} (\sifpoierjsodfgupoefasdfgjsdfgjsdfgjsdjflxncvzxnvasdjfaopsruosihjsdkghajsdflahgfsif_{j}-\sifpoierjsodfgupoefasdfgjsdfgjsdfgjsdjflxncvzxnvasdjfaopsruosihjsdkghajsdflahgfsif_{j-1})\sifpoierjsodfgupoefasdfgjsdfgjsdfgjsdjflxncvzxnvasdjfaopsruosihjsdkghajsdflahgfsif_{j}\zeta \bigl(\mathcal{P}( \Vj_i w) \bigr)_m -\sifpoierjsodfgupoefasdfgjsdfgjsdfgjsdjflxncvzxnvasdjfaopsruosihjsdlghajsdflahgfsif_{j} \sifpoierjsodfgupoefasdfgjsdfgjsdfgjsdjflxncvzxnvasdjfaopsruosihjsdlghajsdflahgfsif_{j-1}\sifpoierjsodfgupoefasdfgjsdfgjsdfgjsdjflxncvzxnvasdjfaopsruosihjsdkghajsdflahgfsif_{j}\sifpoierjsodfgupoefasdfgjsdfgjsdfgjsdjflxncvzxnvasdjfaopsruosihjsdkghajsdflahgfsif_{j-1}\zeta \bigl(\mathcal{P}( \zjm_i w) \bigr)_m \comma m=1,2,3, \end{split} \llabel{2zVpvGg6zFYojSbMBhr4pW8OwDNUao2mhDTScei90KrsmwaBnNUsHe6RpIq1hXFNPm0iVsnGkbCJr8Vmegl416tU2nnollOtcFUM7c4GC8ClaslJ0N8XfCuRaR2sYefjVriJNj1f2tyvqJyQNX1FYmTl5N17tkbBTPuF471AH0Fo71REILJp4VsqiWTTtkAd5RkkJH3RiRNKePesR0xqFqnQjGUIniVgLGCl2He7kmqhEV4PFdCdGpEP9nBmcvZ0pLYGidfn65qEuDfMz2vcq4DMzN6mBFRtQP0yDDFxjuZiZPE3Jj4hVc2zrrcROnFPeOP1pZgnsHAMRK4ETNF23KtfGem2kr5gf5u8NcuwfJCav6SvQ2n18P8RcIkmMSD0wrVR1PYx7kEkZJsJ7Wb6XIWDE0UnqtZPAqEETS3EqNNf38DEk6NhXV9c3sevM32WACSj3eNXuq9GhPOPChd7v1T6gqRinehWk8wLoaawHVvbU4902yObCT6zm2aNf8xUwPOilrR3v8RcNWEk7EvIAI8okPAYxPiUlZ4mwzsJo6ruPmYN6tylDEeeoTmlBKmnVuBB7HnU7qKn353SndtoL82gDifcmjLhHx3gi0akymhuaFTzRnMibFGU5W5x6510NKi85u8JTLYcbfOMn0auD0tvNHwSAWzE3HWcYTId2HhXMLiGiykAjHCnRX4uJJlctQ3yLoqi9ju7Kj84EFU49udeA93xZfZCBW4bSKpycf6nncmvnhKb0HjuKWp6b88pGCEQ44} \end{align} and  \begin{align} \begin{split} g_m  &= \sifpoierjsodfgupoefasdfgjsdfgjsdfgjsdjflxncvzxnvasdjfaopsruosihjsdlghajsdflahgfsif_{j+1} \sifpoierjsodfgupoefasdfgjsdfgjsdfgjsdjflxncvzxnvasdjfaopsruosihjsdlghajsdflahgfsif_{j} \sifpoierjsodfgupoefasdfgjsdfgjsdfgjsdjflxncvzxnvasdjfaopsruosihjsdkghajsdflahgfsif_{j+1}\sifpoierjsodfgupoefasdfgjsdfgjsdfgjsdjflxncvzxnvasdjfaopsruosihjsdkghajsdflahgfsif_{j}\zeta  \big(  \sigma_m(t, \Vj+w)  -  \sigma_m(t, w)  \big)  \\&\indeq  -
 \sifpoierjsodfgupoefasdfgjsdfgjsdfgjsdjflxncvzxnvasdjfaopsruosihjsdlghajsdflahgfsif_{j} \sifpoierjsodfgupoefasdfgjsdfgjsdfgjsdjflxncvzxnvasdjfaopsruosihjsdlghajsdflahgfsif_{j-1}\sifpoierjsodfgupoefasdfgjsdfgjsdfgjsdjflxncvzxnvasdjfaopsruosihjsdkghajsdflahgfsif_{j}\sifpoierjsodfgupoefasdfgjsdfgjsdfgjsdjflxncvzxnvasdjfaopsruosihjsdkghajsdflahgfsif_{j-1} \zeta   \big(  \sigma_m(t, \Vjm+w)  -  \sigma_m(t, w)  \big) \\& = \sifpoierjsodfgupoefasdfgjsdfgjsdfgjsdjflxncvzxnvasdjfaopsruosihjsdlghajsdflahgfsif_{j} (\sifpoierjsodfgupoefasdfgjsdfgjsdfgjsdjflxncvzxnvasdjfaopsruosihjsdlghajsdflahgfsif_{j+1}-\sifpoierjsodfgupoefasdfgjsdfgjsdfgjsdjflxncvzxnvasdjfaopsruosihjsdlghajsdflahgfsif_{j})\sifpoierjsodfgupoefasdfgjsdfgjsdfgjsdjflxncvzxnvasdjfaopsruosihjsdkghajsdflahgfsif_{j+1}\sifpoierjsodfgupoefasdfgjsdfgjsdfgjsdjflxncvzxnvasdjfaopsruosihjsdkghajsdflahgfsif_{j}\zeta \big( \sigma_m(t, \Vj+w) - \sigma_m(t, w) \big) \\&\indeq + (\sifpoierjsodfgupoefasdfgjsdfgjsdfgjsdjflxncvzxnvasdjfaopsruosihjsdlghajsdflahgfsif_{j})^2 (\sifpoierjsodfgupoefasdfgjsdfgjsdfgjsdjflxncvzxnvasdjfaopsruosihjsdkghajsdflahgfsif_{j+1}-\sifpoierjsodfgupoefasdfgjsdfgjsdfgjsdjflxncvzxnvasdjfaopsruosihjsdkghajsdflahgfsif_{j})\sifpoierjsodfgupoefasdfgjsdfgjsdfgjsdjflxncvzxnvasdjfaopsruosihjsdkghajsdflahgfsif_{j}\zeta \big( \sigma_m(t, \Vj+w) - \sigma_m(t, w) \big) \\&\indeq + \sifpoierjsodfgupoefasdfgjsdfgjsdfgjsdjflxncvzxnvasdjfaopsruosihjsdlghajsdflahgfsif_{j} (\sifpoierjsodfgupoefasdfgjsdfgjsdfgjsdjflxncvzxnvasdjfaopsruosihjsdlghajsdflahgfsif_{j}-\sifpoierjsodfgupoefasdfgjsdfgjsdfgjsdjflxncvzxnvasdjfaopsruosihjsdlghajsdflahgfsif_{j-1}) (\sifpoierjsodfgupoefasdfgjsdfgjsdfgjsdjflxncvzxnvasdjfaopsruosihjsdkghajsdflahgfsif_{j})^2\zeta \big( \sigma_m(t, \Vj+w) - \sigma_m(t, w) \big) \\&\indeq +\sifpoierjsodfgupoefasdfgjsdfgjsdfgjsdjflxncvzxnvasdjfaopsruosihjsdlghajsdflahgfsif_{j} \sifpoierjsodfgupoefasdfgjsdfgjsdfgjsdjflxncvzxnvasdjfaopsruosihjsdlghajsdflahgfsif_{j-1}(\sifpoierjsodfgupoefasdfgjsdfgjsdfgjsdjflxncvzxnvasdjfaopsruosihjsdkghajsdflahgfsif_{j}-\sifpoierjsodfgupoefasdfgjsdfgjsdfgjsdjflxncvzxnvasdjfaopsruosihjsdkghajsdflahgfsif_{j-1}) \sifpoierjsodfgupoefasdfgjsdfgjsdfgjsdjflxncvzxnvasdjfaopsruosihjsdkghajsdflahgfsif_{j}\zeta \big( \sigma_m(t, \Vj+w) - \sigma_m(t, w) \big) \\&\indeq + \sifpoierjsodfgupoefasdfgjsdfgjsdfgjsdjflxncvzxnvasdjfaopsruosihjsdlghajsdflahgfsif_{j} \sifpoierjsodfgupoefasdfgjsdfgjsdfgjsdjflxncvzxnvasdjfaopsruosihjsdlghajsdflahgfsif_{j-1} \sifpoierjsodfgupoefasdfgjsdfgjsdfgjsdjflxncvzxnvasdjfaopsruosihjsdkghajsdflahgfsif_{j}\sifpoierjsodfgupoefasdfgjsdfgjsdfgjsdjflxncvzxnvasdjfaopsruosihjsdkghajsdflahgfsif_{j-1}\zeta \bigl(\sigma_m(t, \Vj+w )-\sigma_m(t, \Vjm+w )\bigr) \comma m=1,2,3. \end{split} \llabel{3U7kmCO1eY8jvEbu59zmGZsZh93NwvJYbkEgDpJBjgQeQUH9kCaz6ZGpcpgrHr79IeQvTIdp35mwWmafRgjDvXS7aFgmNIWmjvopqUuxFrBYmoa45jqkRgTBPPKLgoMLjiwIZ2I4F91C6x9aeW7Tq9CeM62kef7MUbovxWyxgIDcL8Xszu2pZTcbjaK0fKzEyznV0WFYxbFOZJYzBCXtQ4uxU96TnN0CGBhWEFZr60rIgw2f9x0fW3kUB4AOfctvL5I0ANOLdw7h8zK12STKy2ZdewoXYPZLVVvtraCxAJmN7MrmIarJtfTddDWE9At6mhMPCVNUOOSZYtGkPvxpsGeRguDvtWTHMHf3Vyr6W3xvcpi0z2wfwQ1DL1wHedTqXlyojGIQAdEEKv7Tak7cAilRfvrlm82NjNg9KDSvNoQiNhng2tnBSVwd8P4o3oLqrzPNHZmkQItfj61TcOQPJblsBYq3NulNfrConZ6kZ2VbZ0psQAaUCiMaoRpFWfviTxmeyzmc5QsEl1PNOZ4xotciInwc6IFbpwsMeXxy8lJ4A6OV0qRzrSt3PMbvRgOS5obkaFU9pOdMPdjFz1KRXRKDVUjveW3d9shi3jzKBTqZkeSXqbzboWTc5yRRMoBYQPCaeZ23HWk9xfdxJYxHYuNMNGY4XLVZoPUQxJAliDHOKycMAcTpGHIktjlIV25YYoRC74thSsJClD76yxM6BRhgfS0UH4wXVF0x1M6IbemsTKSWlsG9pk95kZSdHU31c5EQ45} \end{align}  Utilizing the cutoff functions and the assumptions~\eqref{EQ03}--\eqref{EQ04}, we obtain \begin{equation} \sadklfjsdfgsdfgsdfgsdfgdsfgsdfgadfasdf\biggl[\sifpoierjsodfgupoefasdfgjsdfgjsdfgjsdjflxncvzxnvasdjfaopsruosihjsdfghajsdflahgfsif_0^{t^{\ast}}\sifpoierjsodfgupoefasdfgjsdfgjsdfgjsdjflxncvzxnvasdjfaopsruosihjsdgghajsdflahgfsif  f \sifpoierjsodfgupoefasdfgjsdfgjsdfgjsdjflxncvzxnvasdjfaopsruosihjsdgghajsdflahgfsif_{3}^6 \,ds + \sifpoierjsodfgupoefasdfgjsdfgjsdfgjsdjflxncvzxnvasdjfaopsruosihjsdfghajsdflahgfsif_0^{t^{\ast}}\sifpoierjsodfgupoefasdfgjsdfgjsdfgjsdjflxncvzxnvasdjfaopsruosihjsdgghajsdflahgfsif g\sifpoierjsodfgupoefasdfgjsdfgjsdfgjsdjflxncvzxnvasdjfaopsruosihjsdgghajsdflahgfsif_{\mathbb{L}^6}^6 \,ds \biggr] \leq C_{k, \bar{\epsilon}} \,t^{\ast}\, \sadklfjsdfgsdfgsdfgsdfgdsfgsdfgadfasdf\biggl[\sup_{s\in[0,t^{\ast}]} \sifpoierjsodfgupoefasdfgjsdfgjsdfgjsdjflxncvzxnvasdjfaopsruosihjsdgghajsdflahgfsif \zjm\sifpoierjsodfgupoefasdfgjsdfgjsdfgjsdjflxncvzxnvasdjfaopsruosihjsdgghajsdflahgfsif_{6}^{6}  + \sup_{s\in[0,t^{\ast}]} \sifpoierjsodfgupoefasdfgjsdfgjsdfgjsdjflxncvzxnvasdjfaopsruosihjsdgghajsdflahgfsif \zj\sifpoierjsodfgupoefasdfgjsdfgjsdfgjsdjflxncvzxnvasdjfaopsruosihjsdgghajsdflahgfsif_{6}^{6}  \biggr],    \llabel{BpQeFx5za7hWPlLjDYdKH1pOkMo1Tvhxxz5FLLu71DUNeUXtDFC7CZ2473sjERebaYt2sEpV9wDJ8RGUqQmboXwJnHKFMpsXBvAsX8NYRZMwmZQctltsqofi8wxn6IW8jc68ANBwz8f4gWowkmZPWlwfKpM1fpdo0yTRIKHMDgTl3BUBWr6vHUzFZbqxnwKkdmJ3lXzIwkw7JkuJcCkgvFZ3lSo0ljVKu9Syby46zDjM6RXZIDPpHqEfkHt9SVnVtWdyYNwdmMm7SPwmqhO6FX8tzwYaMvjzpBSNJ1z368900v2i4y2wQjZhwwFUjq0UNmk8J8dOOG3QlDzp8AWpruu4D9VRlpVVzQQg1caEqevP0sFPHcwtKI3Z6nY79iQabga0i9mRVGbvlTAgV6PUV8EupPQ6xvGbcn7dQjV7Ckw57NPWUy9XnwF9elebZ8UYJDx3xCBYCIdPCE2D8eP90u49NY9Jxx9RI4Fea0QCjs5TLodJFphykczBwoe97PohTql1LMs37cKhsHO5jZxqpkHtLbFDnvfTxjiykLVhpwMqobqDM9A0f1n4i5SBc6trqVXwgQBEgH8lISLPLO52EUvi1myxknL0RBebO2YWw8Jhfo1lHlUMiesstdWw4aSWrYvOsn5Wn3wfwzHRHxFg0hKFuNVhjzXbg56HJ9VtUwalOXfT8oiFY1CsUCgCETCIvLR0AgThCs9TaZl6ver8hRtedkAUrkInSbcI8nyEjZsVOSztBbh7WjBgfaAFt4J6CTEQ46} \end{equation} where the constant $C_{k, \bar{\epsilon}}$ is positively correlated with $k$ and~$1/\bar{\epsilon}$. Then, we infer from Lemma~\ref{L01} that \eqref{EQ38} has a fixed point $v\in L^6_{\omega}L^{\infty}_tL^6_{x}$ on $[0,t^{\ast}]$ if $t^{\ast}$ is sufficiently small with respect to $C_{k, \bar{\epsilon}}$, and the rate of convergence is exponential. Showing that this fixed point of \eqref{EQ38} solves \eqref{EQ34} involves taking the limit in the identity  \begin{align} \begin{split} &(\Vj ( s),\varrho) = (v_{0},\varrho)+\sifpoierjsodfgupoefasdfgjsdfgjsdfgjsdjflxncvzxnvasdjfaopsruosihjsdfghajsdflahgfsif_0^s (\Vj (r),\Delta\varrho)\,dr +\sum_{i}\sifpoierjsodfgupoefasdfgjsdfgjsdfgjsdjflxncvzxnvasdjfaopsruosihjsdfghajsdflahgfsif_0^s \bigl(\sifpoierjsodfgupoefasdfgjsdfgjsdfgjsdjflxncvzxnvasdjfaopsruosihjsdlghajsdflahgfsif_j\sifpoierjsodfgupoefasdfgjsdfgjsdfgjsdjflxncvzxnvasdjfaopsruosihjsdlghajsdflahgfsif_{j-1}\sifpoierjsodfgupoefasdfgjsdfgjsdfgjsdjflxncvzxnvasdjfaopsruosihjsdkghajsdflahgfsif_j\sifpoierjsodfgupoefasdfgjsdfgjsdfgjsdjflxncvzxnvasdjfaopsruosihjsdkghajsdflahgfsif_{j-1} \mathcal{P}\bigl( \Vjm_i \Vjm \bigr),\partial_{i}\varrho\bigr)\,dr \\&\indeq + \sum_{i}\sifpoierjsodfgupoefasdfgjsdfgjsdfgjsdjflxncvzxnvasdjfaopsruosihjsdfghajsdflahgfsif_0^s \bigl(\sifpoierjsodfgupoefasdfgjsdfgjsdfgjsdjflxncvzxnvasdjfaopsruosihjsdlghajsdflahgfsif_j\sifpoierjsodfgupoefasdfgjsdfgjsdfgjsdjflxncvzxnvasdjfaopsruosihjsdlghajsdflahgfsif_{j-1}\sifpoierjsodfgupoefasdfgjsdfgjsdfgjsdjflxncvzxnvasdjfaopsruosihjsdkghajsdflahgfsif_j\sifpoierjsodfgupoefasdfgjsdfgjsdfgjsdjflxncvzxnvasdjfaopsruosihjsdkghajsdflahgfsif_{j-1}\zeta \mathcal{P}\bigl( w_i \Vjm \bigr),\partial_{i}\varrho\bigr)\,dr + \sum_{i}\sifpoierjsodfgupoefasdfgjsdfgjsdfgjsdjflxncvzxnvasdjfaopsruosihjsdfghajsdflahgfsif_0^s \bigl(\sifpoierjsodfgupoefasdfgjsdfgjsdfgjsdjflxncvzxnvasdjfaopsruosihjsdlghajsdflahgfsif_j\sifpoierjsodfgupoefasdfgjsdfgjsdfgjsdjflxncvzxnvasdjfaopsruosihjsdlghajsdflahgfsif_{j-1}\sifpoierjsodfgupoefasdfgjsdfgjsdfgjsdjflxncvzxnvasdjfaopsruosihjsdkghajsdflahgfsif_j\sifpoierjsodfgupoefasdfgjsdfgjsdfgjsdjflxncvzxnvasdjfaopsruosihjsdkghajsdflahgfsif_{j-1}\zeta \mathcal{P}\bigl( \Vjm_i w \bigr),\partial_{i}\varrho\bigr)\,dr \\&\indeq +\sifpoierjsodfgupoefasdfgjsdfgjsdfgjsdjflxncvzxnvasdjfaopsruosihjsdfghajsdflahgfsif_0^s \big(\sifpoierjsodfgupoefasdfgjsdfgjsdfgjsdjflxncvzxnvasdjfaopsruosihjsdlghajsdflahgfsif_j\sifpoierjsodfgupoefasdfgjsdfgjsdfgjsdjflxncvzxnvasdjfaopsruosihjsdlghajsdflahgfsif_{j-1}\sifpoierjsodfgupoefasdfgjsdfgjsdfgjsdjflxncvzxnvasdjfaopsruosihjsdkghajsdflahgfsif_j\sifpoierjsodfgupoefasdfgjsdfgjsdfgjsdjflxncvzxnvasdjfaopsruosihjsdkghajsdflahgfsif_{j-1} \zeta \sigma(r, \Vjm+w )-\sifpoierjsodfgupoefasdfgjsdfgjsdfgjsdjflxncvzxnvasdjfaopsruosihjsdlghajsdflahgfsif_j\sifpoierjsodfgupoefasdfgjsdfgjsdfgjsdjflxncvzxnvasdjfaopsruosihjsdlghajsdflahgfsif_{j-1}\sifpoierjsodfgupoefasdfgjsdfgjsdfgjsdjflxncvzxnvasdjfaopsruosihjsdkghajsdflahgfsif_j\sifpoierjsodfgupoefasdfgjsdfgjsdfgjsdjflxncvzxnvasdjfaopsruosihjsdkghajsdflahgfsif_{j-1}\zeta\sigma(r, w ), \varrho \big)\,d\WW_r \comma  (s,\omega)\text{-a.e.}, \end{split} \label{EQ47} \end{align} for all test functions $\varrho\in C^{\infty}(\TT^3)$. Using the exponential rate of convergence (see~\cite[Lemma 5.2]{KXZ}) and the Dominated Convergence Theorem, we conclude \begin{align} \begin{split} &\sifpoierjsodfgupoefasdfgjsdfgjsdfgjsdjflxncvzxnvasdjfaopsruosihjsdfghajsdflahgfsif_0^s (\Vj ,\Delta\varrho)\,dr +\sum_{i}\sifpoierjsodfgupoefasdfgjsdfgjsdfgjsdjflxncvzxnvasdjfaopsruosihjsdfghajsdflahgfsif_0^s \bigl(\sifpoierjsodfgupoefasdfgjsdfgjsdfgjsdjflxncvzxnvasdjfaopsruosihjsdlghajsdflahgfsif_j\sifpoierjsodfgupoefasdfgjsdfgjsdfgjsdjflxncvzxnvasdjfaopsruosihjsdlghajsdflahgfsif_{j-1}\sifpoierjsodfgupoefasdfgjsdfgjsdfgjsdjflxncvzxnvasdjfaopsruosihjsdkghajsdflahgfsif_j\sifpoierjsodfgupoefasdfgjsdfgjsdfgjsdjflxncvzxnvasdjfaopsruosihjsdkghajsdflahgfsif_{j-1} \mathcal{P}\bigl( \Vjm_i \Vjm \bigr),\partial_{i}\varrho\bigr)\,dr \\&\indeq + \sum_{i}\sifpoierjsodfgupoefasdfgjsdfgjsdfgjsdjflxncvzxnvasdjfaopsruosihjsdfghajsdflahgfsif_0^s \bigl(\sifpoierjsodfgupoefasdfgjsdfgjsdfgjsdjflxncvzxnvasdjfaopsruosihjsdlghajsdflahgfsif_j\sifpoierjsodfgupoefasdfgjsdfgjsdfgjsdjflxncvzxnvasdjfaopsruosihjsdlghajsdflahgfsif_{j-1}\sifpoierjsodfgupoefasdfgjsdfgjsdfgjsdjflxncvzxnvasdjfaopsruosihjsdkghajsdflahgfsif_j\sifpoierjsodfgupoefasdfgjsdfgjsdfgjsdjflxncvzxnvasdjfaopsruosihjsdkghajsdflahgfsif_{j-1}\zeta \mathcal{P}\bigl( w_i \Vjm \bigr),\partial_{i}\varrho\bigr)\,dr + \sum_{i}\sifpoierjsodfgupoefasdfgjsdfgjsdfgjsdjflxncvzxnvasdjfaopsruosihjsdfghajsdflahgfsif_0^s \bigl(\sifpoierjsodfgupoefasdfgjsdfgjsdfgjsdjflxncvzxnvasdjfaopsruosihjsdlghajsdflahgfsif_j\sifpoierjsodfgupoefasdfgjsdfgjsdfgjsdjflxncvzxnvasdjfaopsruosihjsdlghajsdflahgfsif_{j-1} \sifpoierjsodfgupoefasdfgjsdfgjsdfgjsdjflxncvzxnvasdjfaopsruosihjsdkghajsdflahgfsif_j\sifpoierjsodfgupoefasdfgjsdfgjsdfgjsdjflxncvzxnvasdjfaopsruosihjsdkghajsdflahgfsif_{j-1}\zeta\mathcal{P}\bigl( \Vjm_i w \bigr),\partial_{i}\varrho\bigr)\,dr \\&\indeq \rightarrow \sifpoierjsodfgupoefasdfgjsdfgjsdfgjsdjflxncvzxnvasdjfaopsruosihjsdfghajsdflahgfsif_0^s (v ,\Delta\sifpoierjsodfgupoefasdfgjsdfgjsdfgjsdjflxncvzxnvasdjfaopsruosihjsdkghajsdflahgfsif)\,dr +\sum_{i}\sifpoierjsodfgupoefasdfgjsdfgjsdfgjsdjflxncvzxnvasdjfaopsruosihjsdfghajsdflahgfsif_0^s \bigl(\psi^2 \sifpoierjsodfgupoefasdfgjsdfgjsdfgjsdjflxncvzxnvasdjfaopsruosihjsdkghajsdflahgfsif^2\mathcal{P}\bigl( v_i v \bigr),\partial_{i}\varrho\bigr)\,dr \\&\indeq + \sum_{i}\sifpoierjsodfgupoefasdfgjsdfgjsdfgjsdjflxncvzxnvasdjfaopsruosihjsdfghajsdflahgfsif_0^s \bigl(\psi^2\sifpoierjsodfgupoefasdfgjsdfgjsdfgjsdjflxncvzxnvasdjfaopsruosihjsdkghajsdflahgfsif^2\zeta \mathcal{P}\bigl( w_i v \bigr),\partial_{i}\varrho\bigr)\,dr + \sum_{i}\sifpoierjsodfgupoefasdfgjsdfgjsdfgjsdjflxncvzxnvasdjfaopsruosihjsdfghajsdflahgfsif_0^s \bigl(\psi^2\sifpoierjsodfgupoefasdfgjsdfgjsdfgjsdjflxncvzxnvasdjfaopsruosihjsdkghajsdflahgfsif^2\zeta \mathcal{P}\bigl( v_i w \bigr),\partial_{i}\varrho\bigr)\,dr \comma  (s,\omega)\text{-a.e.} , \end{split} \llabel{UCU543rbavpOMyelWYWhVBRGow5JRh2nMfUcoBkBXUQ7UlO5rYfHDMceWou3RoFWtbaKh70oHBZn7unRpRh3SIpp0Btqk5vhXCU9BHJFx7qPxB55a7RkOyHmSh5vwrDqt0nF7toPJUGqHfY5uAt5kQLP6ppnRjMHk3HGqZ0OBugFFxSnASHBI7agVfqwfgaAleH9DMnXQQTAAQM8qz9trz86VR2gOMMVuMgf6tGLZWEKqvkMEOgUzMxgN4CbQ8fWY9Tk73Gg90jy9dJbOvddVZmqJjb5qQ5BSFfl2tNPRC86tI0PIdLDUqXKO1ulgXjPVlfDFkFh42W0jwkkH8dxIkjy6GDgeM9mbTYtUS4ltyAVuor6w7InwCh6GG9Km3YozbuVqtsXTNZaqmwkzoKxE9O0QBQXhxN5Lqr6x7SxmvRwTSBGJY5uo5wSNGp3hCcfQNafXWjxeAFyCxUfM8c0kKkwgpsvwVe4tFsGUIzoWFYfnQAUT9xclTfimLCJRXFAmHe7VbYOaFBPjjeF6xI3CzOVvimZ32pt5uveTrhU6y8wjwAyIU3G15HMybdauGckOFnq6a5HaR4DOojrNAjdhSmhOtphQpc9jXX2u5rwPHzW032fi2bz160Ka4FDjd1yVFSMTzSvF1YkRzdzbYbI0qjKMNXBFtXoCZdj9jD5AdSrNBdunlTDIaA4UjYSx6DK1X16i3yiQuq4zooHvHqNgT2VkWGBVA4qeo8HH70FflAqTDBKi461GvMgzd7WriqtFqEQ48} \end{align} for all $s\in [0,t^{\ast}]$ as~$j\rightarrow\infty$, where \begin{equation} \psi = \theta\left (\frac{\sifpoierjsodfgupoefasdfgjsdfgjsdfgjsdjflxncvzxnvasdjfaopsruosihjsdgghajsdflahgfsif v(t, \cdot)\sifpoierjsodfgupoefasdfgjsdfgjsdfgjsdjflxncvzxnvasdjfaopsruosihjsdgghajsdflahgfsif_6}{M_k}\right)   \comma \sifpoierjsodfgupoefasdfgjsdfgjsdfgjsdjflxncvzxnvasdjfaopsruosihjsdkghajsdflahgfsif = \theta\left (\frac{2^{k}\sifpoierjsodfgupoefasdfgjsdfgjsdfgjsdjflxncvzxnvasdjfaopsruosihjsdgghajsdflahgfsif v(t, \cdot)\sifpoierjsodfgupoefasdfgjsdfgjsdfgjsdjflxncvzxnvasdjfaopsruosihjsdgghajsdflahgfsif_3}{\bar{\epsilon}}\right) .    \llabel{24GYcyifYkWHv7EI0aq5JKlfNDCNmWom3VyXJsNt4WP8yGgAoATOkVWZ4ODLtkza9PadGCGQ2FCH6EQppksxFKMWAfY0JdaSYgo7hhGwHttbb4z5qrcdc9CnAmxqY6m8uGf7DZQ6FBUPPiOxgsQ0CZlPYPBa75OiV6tZOBpfYuNcbj4VUpbTKXZRJf36EA0LDgAdfdOpSbg1ynCPUVoRWxeWQMKSmuh3JHqX15APJJX2v0W6lm0llC8hlss1NLWaNhRBAqfIuzkx2sp01oDrYsRywFrNbz1hGpq99FwUzlfcQkTsbCvGIIgmfHhTrM1ItDgCMzYttQRjzFxXIgI7FMAp1kllwJsGodXAT2PgoIp9VonFkwZVQifq9ClAQ4YBwFR4nCyRAg84MLJunx8uKTF3FzlGEQtl32y174wLXZm62xX5xGoaCHvgZFEmyDIzj3q10RZrsswByA2WlOADDDQVin8PTFLGmwi6pgRZQ6A5TLlmnFVtNiJbnUkLyvq9zSBP6eJJq7P6RFaim6KXPWaxm6W7fM83uKD6kNj7vhg4ppZ4ObMaSaPH0oqxABG8vqrqT6QiRGHBCCN1ZblTY4zq8lFqLCkghxDUuZw7MXCD4psZcEX9RlCwf0CCG8bgFtiUv3mQeLWJoyFkv6hcSnMmKbiQukLFpYAqo5Fjf9RRRtqS6XWVoIYVDMla5c7cWKJLUqcvtiIOeVCU7xJdC5W5bk3fQbyZjtUDmegbgI179dlU3u3cvWoAIowbEZ0xP2EQ49} \end{equation} Moreover, by the BDG inequality, \begin{align} \begin{split} &\sadklfjsdfgsdfgsdfgsdfgdsfgsdfgadfasdf\biggl[\sup_{s\in[0,t^{\ast}]}\biggl|\sifpoierjsodfgupoefasdfgjsdfgjsdfgjsdjflxncvzxnvasdjfaopsruosihjsdfghajsdflahgfsif_0^s  \Bigl( \sifpoierjsodfgupoefasdfgjsdfgjsdfgjsdjflxncvzxnvasdjfaopsruosihjsdlghajsdflahgfsif_j\sifpoierjsodfgupoefasdfgjsdfgjsdfgjsdjflxncvzxnvasdjfaopsruosihjsdlghajsdflahgfsif_{j-1}\sifpoierjsodfgupoefasdfgjsdfgjsdfgjsdjflxncvzxnvasdjfaopsruosihjsdkghajsdflahgfsif_j\sifpoierjsodfgupoefasdfgjsdfgjsdfgjsdjflxncvzxnvasdjfaopsruosihjsdkghajsdflahgfsif_{j-1}\zeta \bigl( \sigma(r, \Vjm +w )-\sigma(r,w)\bigr) -\psi^2\sifpoierjsodfgupoefasdfgjsdfgjsdfgjsdjflxncvzxnvasdjfaopsruosihjsdkghajsdflahgfsif^2\zeta \bigl( \sigma(r,v+w)-\sigma(r,w)\bigr) ,\varrho \Bigr) \,d\WW_r\biggr|\biggr] \\&\indeq \leq  C\sadklfjsdfgsdfgsdfgsdfgdsfgsdfgadfasdf\biggl[\biggl(\sifpoierjsodfgupoefasdfgjsdfgjsdfgjsdjflxncvzxnvasdjfaopsruosihjsdfghajsdflahgfsif_0^{t^{\ast}}  (\sifpoierjsodfgupoefasdfgjsdfgjsdfgjsdjflxncvzxnvasdjfaopsruosihjsdlghajsdflahgfsif_{j}-\psi)^2(\sifpoierjsodfgupoefasdfgjsdfgjsdfgjsdjflxncvzxnvasdjfaopsruosihjsdlghajsdflahgfsif_{j-1})^2(\sifpoierjsodfgupoefasdfgjsdfgjsdfgjsdjflxncvzxnvasdjfaopsruosihjsdkghajsdflahgfsif_{j})^2(\sifpoierjsodfgupoefasdfgjsdfgjsdfgjsdjflxncvzxnvasdjfaopsruosihjsdkghajsdflahgfsif_{j-1})^2\zeta^2 \bigl\sifpoierjsodfgupoefasdfgjsdfgjsdfgjsdjflxncvzxnvasdjfaopsruosihjsdgghajsdflahgfsif \bigl( \sigma(r, \Vjm+w )-\sigma(r,w) ,\varrho \bigr) \bigr\sifpoierjsodfgupoefasdfgjsdfgjsdfgjsdjflxncvzxnvasdjfaopsruosihjsdgghajsdflahgfsif_{ l^2}^2\, dr\biggr)^{1/2}\biggr] \\&\indeq\indeq + 
C\sadklfjsdfgsdfgsdfgsdfgdsfgsdfgadfasdf\biggl[\biggl(\sifpoierjsodfgupoefasdfgjsdfgjsdfgjsdjflxncvzxnvasdjfaopsruosihjsdfghajsdflahgfsif_0^{t^{\ast}}  (\psi)^2(\sifpoierjsodfgupoefasdfgjsdfgjsdfgjsdjflxncvzxnvasdjfaopsruosihjsdlghajsdflahgfsif_{j-1})^2(\sifpoierjsodfgupoefasdfgjsdfgjsdfgjsdjflxncvzxnvasdjfaopsruosihjsdkghajsdflahgfsif_{j}-\sifpoierjsodfgupoefasdfgjsdfgjsdfgjsdjflxncvzxnvasdjfaopsruosihjsdkghajsdflahgfsif)^2(\sifpoierjsodfgupoefasdfgjsdfgjsdfgjsdjflxncvzxnvasdjfaopsruosihjsdkghajsdflahgfsif_{j-1})^2\zeta^2 \bigl\sifpoierjsodfgupoefasdfgjsdfgjsdfgjsdjflxncvzxnvasdjfaopsruosihjsdgghajsdflahgfsif \bigl( \sigma(r, \Vjm+w )-\sigma(r,w) ,\varrho \bigr) \bigr\sifpoierjsodfgupoefasdfgjsdfgjsdfgjsdjflxncvzxnvasdjfaopsruosihjsdgghajsdflahgfsif_{ l^2}^2\, dr\biggr)^{1/2}\biggr] \\&\indeq\indeq +  C\sadklfjsdfgsdfgsdfgsdfgdsfgsdfgadfasdf\biggl[\biggl(\sifpoierjsodfgupoefasdfgjsdfgjsdfgjsdjflxncvzxnvasdjfaopsruosihjsdfghajsdflahgfsif_0^{t^{\ast}}  (\psi)^2(\sifpoierjsodfgupoefasdfgjsdfgjsdfgjsdjflxncvzxnvasdjfaopsruosihjsdlghajsdflahgfsif_{j-1})^2(\sifpoierjsodfgupoefasdfgjsdfgjsdfgjsdjflxncvzxnvasdjfaopsruosihjsdkghajsdflahgfsif)^2(\sifpoierjsodfgupoefasdfgjsdfgjsdfgjsdjflxncvzxnvasdjfaopsruosihjsdkghajsdflahgfsif_{j-1})^2\zeta^2 \bigl\sifpoierjsodfgupoefasdfgjsdfgjsdfgjsdjflxncvzxnvasdjfaopsruosihjsdgghajsdflahgfsif \bigl( \sigma(r, \Vjm+w ) - \sigma(r, v+w) ,\varrho \bigr) \bigr\sifpoierjsodfgupoefasdfgjsdfgjsdfgjsdjflxncvzxnvasdjfaopsruosihjsdgghajsdflahgfsif_{ l^2}^2\, dr\biggr)^{1/2}\biggr] \\&\indeq\indeq + C\sadklfjsdfgsdfgsdfgsdfgdsfgsdfgadfasdf\biggl[\biggl(\sifpoierjsodfgupoefasdfgjsdfgjsdfgjsdjflxncvzxnvasdjfaopsruosihjsdfghajsdflahgfsif_0^{t^{\ast}} (\sifpoierjsodfgupoefasdfgjsdfgjsdfgjsdjflxncvzxnvasdjfaopsruosihjsdlghajsdflahgfsif_{j-1}-\psi)^2(\psi)^2(\sifpoierjsodfgupoefasdfgjsdfgjsdfgjsdjflxncvzxnvasdjfaopsruosihjsdkghajsdflahgfsif)^2(\sifpoierjsodfgupoefasdfgjsdfgjsdfgjsdjflxncvzxnvasdjfaopsruosihjsdkghajsdflahgfsif_{j-1})^2\zeta^2 \bigl\sifpoierjsodfgupoefasdfgjsdfgjsdfgjsdjflxncvzxnvasdjfaopsruosihjsdgghajsdflahgfsif \bigl( \sigma(r, v+w)-\sigma(r,w) ,\varrho \bigr) \bigr\sifpoierjsodfgupoefasdfgjsdfgjsdfgjsdjflxncvzxnvasdjfaopsruosihjsdgghajsdflahgfsif_{ l^2}^2\, dr\biggr)^{1/2}\biggr] \\&\indeq\indeq + C\sadklfjsdfgsdfgsdfgsdfgdsfgsdfgadfasdf\biggl[\biggl(\sifpoierjsodfgupoefasdfgjsdfgjsdfgjsdjflxncvzxnvasdjfaopsruosihjsdfghajsdflahgfsif_0^{t^{\ast}} (\psi)^4(\sifpoierjsodfgupoefasdfgjsdfgjsdfgjsdjflxncvzxnvasdjfaopsruosihjsdkghajsdflahgfsif)^2(\sifpoierjsodfgupoefasdfgjsdfgjsdfgjsdjflxncvzxnvasdjfaopsruosihjsdkghajsdflahgfsif_{j-1}-\sifpoierjsodfgupoefasdfgjsdfgjsdfgjsdjflxncvzxnvasdjfaopsruosihjsdkghajsdflahgfsif)^2\zeta^2 \bigl\sifpoierjsodfgupoefasdfgjsdfgjsdfgjsdjflxncvzxnvasdjfaopsruosihjsdgghajsdflahgfsif \bigl( \sigma(r, v+w)-\sigma(r,w) ,\varrho \bigr) \bigr\sifpoierjsodfgupoefasdfgjsdfgjsdfgjsdjflxncvzxnvasdjfaopsruosihjsdgghajsdflahgfsif_{ l^2}^2\, dr\biggr)^{1/2}\biggr] :=  \sum_{i=1}^{5}Q_i. \end{split} \label{EQ50} \end{align} We estimate these terms using Minkowski's inequality,  properties of the cutoff functions, and the assumptions on~$\sigma$. First, \begin{align} \begin{split} Q_1&\leq C \sadklfjsdfgsdfgsdfgsdfgdsfgsdfgadfasdf\biggl[ \biggl( \sifpoierjsodfgupoefasdfgjsdfgjsdfgjsdjflxncvzxnvasdjfaopsruosihjsdfghajsdflahgfsif_0^{t^{\ast}} (\sifpoierjsodfgupoefasdfgjsdfgjsdfgjsdjflxncvzxnvasdjfaopsruosihjsdlghajsdflahgfsif_{j}-\psi)^2(\sifpoierjsodfgupoefasdfgjsdfgjsdfgjsdjflxncvzxnvasdjfaopsruosihjsdlghajsdflahgfsif_{j-1}\sifpoierjsodfgupoefasdfgjsdfgjsdfgjsdjflxncvzxnvasdjfaopsruosihjsdkghajsdflahgfsif_{j}\sifpoierjsodfgupoefasdfgjsdfgjsdfgjsdjflxncvzxnvasdjfaopsruosihjsdkghajsdflahgfsif_{j-1})^2 \sifpoierjsodfgupoefasdfgjsdfgjsdfgjsdjflxncvzxnvasdjfaopsruosihjsdgghajsdflahgfsif \Vjm \sifpoierjsodfgupoefasdfgjsdfgjsdfgjsdjflxncvzxnvasdjfaopsruosihjsdgghajsdflahgfsif_{6}^2\, dr \biggr)^{1/2}  \biggr] \leq  C_{k,T} \sqrt{t^{\ast}} \biggl(\sadklfjsdfgsdfgsdfgsdfgdsfgsdfgadfasdf\Big[ \sup_{r\in[0,t^{\ast}]}\sifpoierjsodfgupoefasdfgjsdfgjsdfgjsdjflxncvzxnvasdjfaopsruosihjsdgghajsdflahgfsif\Vj -v\sifpoierjsodfgupoefasdfgjsdfgjsdfgjsdjflxncvzxnvasdjfaopsruosihjsdgghajsdflahgfsif_{6}^6 \Big] \biggr)^{1/6} . \end{split} \llabel{FBMSwazV1XfzVi97mmy5sTJK0hz9O6pDaGctytmHTDYxTUBALNvQefRQuF2OyokVsLJwdqgDhTTJeR7CuPczNLVj1HKml8mwLFr8Gz66n4uA9YTt9oiJGclm0EckA9zkElOB9Js7Gfwhqyglc2RQ9d52aYQvC8ArK7aCLmENPYd27XImGC6L9gOfyL05HMtgR65lBCsWGwFKGBIQiIRBiT95N78wncbk7EFeiBRB216SiHoHJSkNgxqupJmZ1pxEbWcwiJX5NfiYPGD6uWsXTP94uaFVDZuhJH2d0PLOY243xMK47VP6FTyT35zpLxRC6tN89as3ku8eGrdMKWoMIU946FBjksOTe0UxZD4avbTw5mQ3Ry9AfJFjPgvLFKz0olfZdj3O07EavpWfbM3rBGSyOiuxpI4o82JJ42X1GIux8QFh3PhRtY9vjiSL6x76W9y2Zz3YASGRMp7kDhrgma8fWGG0qKLsO5oQr42t1jP1crM2fClRbETdqra5lVG1lKitbXqbdPKcaUV0lv4Lalo8VTXclaUqh5GWCzAnRnlNNcmwaF8ErbwX32rjiHleb4gXSjLROJgG2yb8OCAxN4uy4RsLQjD7U7enwcYCnZxiKdju74vpjBKKjRRl36kXXzvnX2JrD8aPDUWGstgb8CTWYnHRs6y6JCp8Lx1jzCI1mtG26y5zrJ1nFhX67wCzqF8uZQIS0dnYxPeXDyjBz1aYwzDXaxaMIZzJ3C3QRrahpw8sWLxrAsSqZP5Wvv1QF7EQ51} \end{align} Similarly, by the assumption~\eqref{EQ30}, \begin{align} \begin{split} &Q_3+Q_4 \leq   C_{k,T} \sqrt{t^{\ast}} \biggl(\sadklfjsdfgsdfgsdfgsdfgdsfgsdfgadfasdf\Big[ \sup_{r\in[0,t^{\ast}]}\sifpoierjsodfgupoefasdfgjsdfgjsdfgjsdjflxncvzxnvasdjfaopsruosihjsdgghajsdflahgfsif\Vjm -v\sifpoierjsodfgupoefasdfgjsdfgjsdfgjsdjflxncvzxnvasdjfaopsruosihjsdgghajsdflahgfsif_{6}^6 \Big] \biggr)^{1/6}. \end{split} \llabel{JPAVQwuWu69YLwNHUPJ0wjs7RSiVaPrEGgxYaVmSk3Yo1wLn0q0PVeXrzoCIH7vxq5ztOmq6mp4drApdzhwSOlRPDpsClr8FoZUG7vDUYhbScJ6gJb8Q8emG2JG9Oja83owYwjozLa3DB500siGjEHolPuqe4p7T1kQJmU6cHnOo29oroOzTa3j31n8mDL7CIvCpKZUs0jVrb7vHIH7NTtYY7JKvVdGLhA1ONCWoQW1fvjmlH7lSlIm8T1QSdUWhTiMPKDZmm4V7ofRW1dnlqg0Ah1QRjdtKZVzEBNE1eXiRRSLLQPESEDeXbiMMFfxC5FI1zviyNsYHPsGxfGiIuhDPDi0OIHuBTTHOCHyCTkABxuCjgOZs965wfeFwvfRpNLLT3EvgKgkO9jyyvotRRlpDTdn9H5ZnqwWr4OUkIlxtsk0RZdODnsoYid6ctgwwQrxQk1S8ajpPiZJlp5pIAT1t482KxtvQ6D1TVzQ7F3xoz6Hw2phWDlCJg7VcEix6XFIdlOlcNbgODKp86tCHVGrzEcVnBk99sq5XGd1DNFANeggJYjfBWjAbJSchyEuVlENawP0DWoZWKuP4IPtvZbmnRL0472K3bBQIH5SpPxtXy5NJjoWceA7FeT7IwpivQdqLaeZE0QfiMW1KozkdUtRsGH6ryobMpDbfLt0Z2FAXbR3QQwuIizgZFQ4Gh4lY5pt9RMTieqBIkdXI979BGU2yYtJSanOMsDLWydCQfolxJWbbIdbEggZLBKbFmKXoRMEQ52} \end{align} Using that the spatial domain is $\TT^3$, we also have  \begin{align} \begin{split} &Q_2+Q_5 \leq   C_{k,\bar{\epsilon}, T} \sqrt{t^{\ast}} \sadklfjsdfgsdfgsdfgsdfgdsfgsdfgadfasdf\Big[ \sup_{r\in[0,t^{\ast}]}\sifpoierjsodfgupoefasdfgjsdfgjsdfgjsdjflxncvzxnvasdjfaopsruosihjsdgghajsdflahgfsif\Vjm-v\sifpoierjsodfgupoefasdfgjsdfgjsdfgjsdjflxncvzxnvasdjfaopsruosihjsdgghajsdflahgfsif_{3} + \sup_{r\in[0,t^{\ast}]}\sifpoierjsodfgupoefasdfgjsdfgjsdfgjsdjflxncvzxnvasdjfaopsruosihjsdgghajsdflahgfsif\Vj-v\sifpoierjsodfgupoefasdfgjsdfgjsdfgjsdjflxncvzxnvasdjfaopsruosihjsdgghajsdflahgfsif_{3} \Big] \\ &\indeq  \leq   C_{k,\bar{\epsilon}, T} \sqrt{t^{\ast}} \biggl(\sadklfjsdfgsdfgsdfgsdfgdsfgsdfgadfasdf\Big[ \sup_{r\in[0,t^{\ast}]}\sifpoierjsodfgupoefasdfgjsdfgjsdfgjsdjflxncvzxnvasdjfaopsruosihjsdgghajsdflahgfsif\Vjm-v\sifpoierjsodfgupoefasdfgjsdfgjsdfgjsdjflxncvzxnvasdjfaopsruosihjsdgghajsdflahgfsif_{6}^6 \Big] \biggr)^{1/6} +C_{k,\bar{\epsilon}, T} \sqrt{t^{\ast}} \biggl(\sadklfjsdfgsdfgsdfgsdfgdsfgsdfgadfasdf\Big[ \sup_{r\in[0,t^{\ast}]}\sifpoierjsodfgupoefasdfgjsdfgjsdfgjsdjflxncvzxnvasdjfaopsruosihjsdgghajsdflahgfsif\Vj-v\sifpoierjsodfgupoefasdfgjsdfgjsdfgjsdjflxncvzxnvasdjfaopsruosihjsdgghajsdflahgfsif_{6}^6 \Big] \biggr)^{1/6}. \end{split} \llabel{cUyM8NlGnWyuERUtbAs4ZRPHdIWtlbJRtQwoddmlZhI3I8A9K8SyflGzcVjCqGkZnaZrxHNxIcMaeGQdXXxGHFi6AeYBAlo4Q9HZIjJjtOhl4VLmVvcphmMESM8ltxHQQUHjJhYyf5Ndc0i8mHOTNS7yx5hNrJCyJ1ZFj4QeIom7wczw98Bn6SxxoqPtnXp4FyiEb2MCyj2AHaB8FejdIRhqQVfR8rEtz0mq54IZtbSlXdBmEvCuvAf5bYxZ3LEsJYEX8eNmotV2IHlhJE70cs45KVwJR1riFMPEsP3srHa8pqwVNAHusohYIrkNwekfRbDVLm2axu6caKkTXrgBgnQhUA1z8X6MtqvksUfAFVLgTmqPntrgIggjfJfMGfCuByBS7njWfYRNhpHsjFCzM4f6cRDgjPZkbSUHQBnzQwEnS9CxSfn00xmAfwlTv4HIZIZAyXIs4hPOPjQ3v93iTL0JtNJ8baBBWcY18vifUiGKvSQ4gEkZ10yS5lXCwI4oX2gPBisFp7TjKupgVn5oi4uxKt2QP4kbrChS5ZnuWXWep0mOjW1r2IaXvHle8ksF2XQ529gTLs3uvAOf64HOVIqrbLoG5I2n0XskvcKYFIV8yP9tfMEVPR7F0ipDaqwgQxro5EtIWr3tEaSs5CjzfRRALgvmyMhIztVKjStP744RC0TTPQpn8gLVtzpLzEQe2Rck9WuM7XHGA7O7KGwfmZHLhJRNUDEQeBrqfKIt0Y4RW49GKEHYptgLH4F8rZfYCvEQ53} \end{align} Based on the estimates above, the right side of \eqref{EQ50} converges to zero exponentially fast as $j\rightarrow \infty$. Then, \begin{equation} \sifpoierjsodfgupoefasdfgjsdfgjsdfgjsdjflxncvzxnvasdjfaopsruosihjsdfghajsdflahgfsif_0^s (\sifpoierjsodfgupoefasdfgjsdfgjsdfgjsdjflxncvzxnvasdjfaopsruosihjsdlghajsdflahgfsif_j\sifpoierjsodfgupoefasdfgjsdfgjsdfgjsdjflxncvzxnvasdjfaopsruosihjsdlghajsdflahgfsif_{j-1}\sifpoierjsodfgupoefasdfgjsdfgjsdfgjsdjflxncvzxnvasdjfaopsruosihjsdkghajsdflahgfsif_j\sifpoierjsodfgupoefasdfgjsdfgjsdfgjsdjflxncvzxnvasdjfaopsruosihjsdkghajsdflahgfsif_{j-1}\zeta(\sigma(r, \Vjm +w )-\sigma(r, w )),\varrho)\,d\WW_r \xrightarrow{j \to \infty} \sifpoierjsodfgupoefasdfgjsdfgjsdfgjsdjflxncvzxnvasdjfaopsruosihjsdfghajsdflahgfsif_0^s (\psi^2 \sifpoierjsodfgupoefasdfgjsdfgjsdfgjsdjflxncvzxnvasdjfaopsruosihjsdkghajsdflahgfsif^2\zeta(\sigma(r,v+w)-\sigma(r, w )),\varrho)\,d\WW_r,    \llabel{cf1pOyjk8iTES0ujRvFpipcwIvLDgikPuqqk9REdH9YjRUMkr9byFJKLBex0SgDJ2gBIeCX2CUZyyRtGNY3eGOaDp3mwQyV1AjtGLgSC1dDpQCBcocMSM4jqbSWbvx6aSnuMtD05qpwNDlW0tZ1cbjzwU5bUdCGAghCw0nICDFKHRkphbtA6nYld6c5TSkDq3Qxo2jhDxQbmb8nPq3zNZQFJJyuVm1C6rzRDCB1meQy4TtYr5jQVWoOfbrYQ6qakZepHb2b5w4KN3mEHtQKAXsIycbakyID9O8YCmRlEW7fGISs6xazbM6PSBN2Bjtb65zz2NuYo4kUlpIqJVBC4DzuZZN6Zkz0oommnswebstFmlxkKEQEL6bsoYzxx08IQ5Ma7InfdXLQ9jeHSTmigttk4vP7778Hp1o67atRbfcrS2CWzwQ9j0Rjr0VL9vlvkkk6J9bM1XgiYlay8ZEq39Z53jRnXh5mKPPa5tFw7E0nE7CuFIoVlFxguxB1hqlHeOLdb7RKfl0SKJiYekpvRSYnNFf7UVOWBvwpN9mtgGwh2NJCY53IdJXPpYAZ1B1AgSxn61oQVtg7W7QcPC42ecSA5jG4K5H1tQs6TNphOKTBIdGkFSGmV0kzAxavQzjeXGbiSjg3kYZ5LxzF3JNHknrmy4smJ70whEtBeXkSTWEujcAuS0NkHloa7wYgMa5j8g4gi7WZ77Ds5MZZMtN5iJEaCfHJ0sD6zVuX06BP99Fga9GgYMv6YFVOBERy3Xw2SBYEQ54} \end{equation} $(s,\omega)$-a.e.\ for all $s\in [0,t^{\ast}]$. Letting $j\rightarrow \infty$ in \eqref{EQ47} and \eqref{EQ37}, we obtain that $v$ solves~\eqref{EQ34} and satisfies~\eqref{EQ37}. Thus, the existence of a strong solution is established on $[0,t^{\ast}]$. The pathwise uniqueness can be proven in a similar manner. Note that the smallness of $t^{\ast}$ depends only on $k$, $T$, and~$\bar{\epsilon}$. By extending this result for finitely many steps, we obtain the existence and uniqueness of a strong solution to~\eqref{EQ34} as well as the energy estimate \eqref{EQ37} on $[0,T]$.   \end{proof} \par \subsection{Global energy control} To assess the impact of the cutoff functions in the truncated models~\eqref{EQ34}, we derive certain energy estimates of the solution $v^{(k)}$ for all $k\in\NNz$. \par \cole \begin{Lemma}[An $L^{p}$-energy control] \label{L07} Let $\epsilon>0$, $p\in\{3,6\}$, and $k\in\mathbb{N}_0$. Suppose that $\vk_0$ is an initial datum of~\eqref{EQ34} satisfying the assertions in Lemma~\ref{L05}; additionally, assume that $\sifpoierjsodfgupoefasdfgjsdfgjsdfgjsdjflxncvzxnvasdjfaopsruosihjsdgghajsdflahgfsif \ukm\sifpoierjsodfgupoefasdfgjsdfgjsdfgjsdjflxncvzxnvasdjfaopsruosihjsdgghajsdflahgfsif_3\leq\epsilon$ in $\Omega\times [0,\infty)$. If $\epsilon$, $\bar{\epsilon}$, and $\epsilon_{\sigma}$ (see \eqref{EQ36} and the assumptions~\eqref{EQ30} on $\sigma$) are sufficiently small, then there exists a positive constant $C$ that is independent of $k$ so that \begin{align} \sadklfjsdfgsdfgsdfgsdfgdsfgsdfgadfasdf \left[\sup_{t\in[0,\infty)} \sifpoierjsodfgupoefasdfgjsdfgjsdfgjsdjflxncvzxnvasdjfaopsruosihjsdgghajsdflahgfsif \vk(t)\sifpoierjsodfgupoefasdfgjsdfgjsdfgjsdjflxncvzxnvasdjfaopsruosihjsdgghajsdflahgfsif_p^p +\sifpoierjsodfgupoefasdfgjsdfgjsdfgjsdjflxncvzxnvasdjfaopsruosihjsdfghajsdflahgfsif_0^{\infty}\sum_{i=1}^{3} \sifpoierjsodfgupoefasdfgjsdfgjsdfgjsdjflxncvzxnvasdjfaopsruosihjsdfghajsdflahgfsif_{\TT^3} | \nabla (|\vk_i(t,x)|^{p/2})|^2 \,dx dt \right] \leq C\sadklfjsdfgsdfgsdfgsdfgdsfgsdfgadfasdf[\sifpoierjsodfgupoefasdfgjsdfgjsdfgjsdjflxncvzxnvasdjfaopsruosihjsdgghajsdflahgfsif \vk_0\sifpoierjsodfgupoefasdfgjsdfgjsdfgjsdjflxncvzxnvasdjfaopsruosihjsdgghajsdflahgfsif_p^p]. \label{EQ55} \end{align} \cole \end{Lemma} \colb \par \begin{proof}[Proof of Lemma~\ref{L07}] Recall that in Lemma~\ref{L06} we have established the existence of a unique strong solution $\vk$ to~\eqref{EQ34} on $[0,T]$ for an arbitrary $T>0$. Now we apply Lemma~\ref{L02} to~\eqref{EQ34} componentwise, obtaining   \begin{align}   \begin{split}   &\sadklfjsdfgsdfgsdfgsdfgdsfgsdfgadfasdf\biggl[\sup_{0\leq t\leq T}\sifpoierjsodfgupoefasdfgjsdfgjsdfgjsdjflxncvzxnvasdjfaopsruosihjsdgghajsdflahgfsif\vk_j(t,\cdot)\sifpoierjsodfgupoefasdfgjsdfgjsdfgjsdjflxncvzxnvasdjfaopsruosihjsdgghajsdflahgfsif_p^p   +\sifpoierjsodfgupoefasdfgjsdfgjsdfgjsdjflxncvzxnvasdjfaopsruosihjsdfghajsdflahgfsif_0^{T} \sifpoierjsodfgupoefasdfgjsdfgjsdfgjsdjflxncvzxnvasdjfaopsruosihjsdfghajsdflahgfsif_{\TT^3} | \nabla (|\vk_j(t,x)|^{p/2})|^2 \,dx dt\biggr]   \\&\indeq  \leq C   \sadklfjsdfgsdfgsdfgsdfgdsfgsdfgadfasdf\biggl[    \sifpoierjsodfgupoefasdfgjsdfgjsdfgjsdjflxncvzxnvasdjfaopsruosihjsdgghajsdflahgfsif\vk_{0,j}\sifpoierjsodfgupoefasdfgjsdfgjsdfgjsdjflxncvzxnvasdjfaopsruosihjsdgghajsdflahgfsif_p^p    +   \sifpoierjsodfgupoefasdfgjsdfgjsdfgjsdjflxncvzxnvasdjfaopsruosihjsdfghajsdflahgfsif_0^{T} \sifpoierjsodfgupoefasdfgjsdfgjsdfgjsdjflxncvzxnvasdjfaopsruosihjsdfghajsdflahgfsif_{\TT^3}   \bigl(   |f_j^{(1)}|^2   +   |f_j^{(2)}|^2   + |f_j^{(3)}|^2   \bigr)    |\vk_j(t,x)|^{p-2}\,dxdt   \biggr]    \\&\indeq   + C   \sadklfjsdfgsdfgsdfgsdfgdsfgsdfgadfasdf\biggl[   \sifpoierjsodfgupoefasdfgjsdfgjsdfgjsdjflxncvzxnvasdjfaopsruosihjsdfghajsdflahgfsif_0^{T}\sifpoierjsodfgupoefasdfgjsdfgjsdfgjsdjflxncvzxnvasdjfaopsruosihjsdlghajsdflahgfsif_k^2 \sifpoierjsodfgupoefasdfgjsdfgjsdfgjsdjflxncvzxnvasdjfaopsruosihjsdkghajsdflahgfsif_k^2 \zeta_{k-1}\sifpoierjsodfgupoefasdfgjsdfgjsdfgjsdjflxncvzxnvasdjfaopsruosihjsdfghajsdflahgfsif_{\TT^3} |\vk_j(t)|^{p-2}\sifpoierjsodfgupoefasdfgjsdfgjsdfgjsdjflxncvzxnvasdjfaopsruosihjsdgghajsdflahgfsif      \sigma_j(t, \uk( t,x)) -   \sigma_j(t, \ukm( t,x))\sifpoierjsodfgupoefasdfgjsdfgjsdfgjsdjflxncvzxnvasdjfaopsruosihjsdgghajsdflahgfsif_{l^2}^2\,dxdt   \biggr]   \commaone j=1,2,3,   \end{split}   \llabel{ZDxixxWHrrlxjKA3fokPh9Y758fGXEhgbBw82C4JCStUeozJfIuGjPpwp7UxCE5ahG5EGJF3nRLM8CQc00TcmXISIyZNJWKMIzkF5u1nvD8GWYqBt2lNxdvzbXj00EEpUTcw3zvyfab6yQoRjHWRFJzPBuZ61G8w0SAbzpNLIVjWHkWfjylXj6VZvjsTwO3UzBosQ7erXyGsdvcKrYzZGQeAM1u1TNkybHcU71KmpyahtwKEj7Ou0A7epb7v4FdqSAD7c02cGvsiW444pFeh8OdjwM7olsSQoeyZXota8wXr6NSG2sFoGBel3PvMoGgamq3YkaatLidTQ84LYKFfAF15vlZaeTTxvru2xlM2gFBbV80UJQvkebsTqFRfmCSVe34YVHOuKokFXYI2MTZj8BZX0EuD1dImocM93NjZPlPHqEll4z66IvF3TOMb7xuVRYjlVEBGePNUgLqSd4OYNeXudaDQ6BjKUrIpcr5n8QTNztBho3LC3rc30it5CN2TmN88XYeTdqTLPlS97uLMw0NAsMphOuPNisXNIlWfXBGc2hxykg50QTN75t5JNwZR3NH1MnVRZj2PrUYveHPEljGaTIx4sCFzKB0qp3PleK68p85w44l5zZl07brv61KkiAuTSA5dkwYS3F3YF3e1xKEJWoAvVOZVbwNYgF7CKbSi92R0rlWh2akhCoEppr6O2PZJDZN8ZZD4IhHPTMvSDTgOy1lZ0Y86n9aMgkWdeuOZjOi2Fg3ziYaSRCjlzXdQKbEQ56}   \end{align} where    \begin{align}   \begin{split}   &   f^{(1)}_{ij} = - \sifpoierjsodfgupoefasdfgjsdfgjsdfgjsdjflxncvzxnvasdjfaopsruosihjsdlghajsdflahgfsif_k^2\sifpoierjsodfgupoefasdfgjsdfgjsdfgjsdjflxncvzxnvasdjfaopsruosihjsdkghajsdflahgfsif_k^2(\mathcal{P}(\vk_i \vk))_j,   \\&   f^{(2)}_{ij} = - \sifpoierjsodfgupoefasdfgjsdfgjsdfgjsdjflxncvzxnvasdjfaopsruosihjsdlghajsdflahgfsif_k^2\sifpoierjsodfgupoefasdfgjsdfgjsdfgjsdjflxncvzxnvasdjfaopsruosihjsdkghajsdflahgfsif_k^2\zeta_{k-1}(\mathcal{P}(u^{(k-1)}_i \vk))_j,   \\&   f^{(3)}_{ij} = - \sifpoierjsodfgupoefasdfgjsdfgjsdfgjsdjflxncvzxnvasdjfaopsruosihjsdlghajsdflahgfsif_k^2\sifpoierjsodfgupoefasdfgjsdfgjsdfgjsdjflxncvzxnvasdjfaopsruosihjsdkghajsdflahgfsif_k^2\zeta_{k-1}(\mathcal{P}(\vk_i u^{(k-1)}))_j   ,   \end{split}   \llabel{cnb5pKTqrJp1P6oGyxc9vZZRZeFr5TsSZzGl7HWuIGM0yReYDw3lMuxgAdF6dpp8ZVRcl7uqH8OBMbzL6dKBflWCWdlVhycV5nEpv2JSkD0ccMpoIR38QpeZj9j0ZoPmqXRTxBs8w9Q5epR3tN5jbvbrbSK7U4W4PJ0ovnB0opRpCYNPso834PwtSRqvir4DRqujaJq32QUTG1Pgbp6nJM2CUnENdJCr3ZGBHEgBtdsTd84gM22gKBN7QnmRtJgKUIGEeKx64yAGKGezeJNmpeQkLR389HH9fXLBcE6T4GjVZLIdLQIiQtkBk9G9FzHWIGm91M7SW029tzNUX3HLrOUtvG5QZnDqyM6ESTxfoUVylEQ99nTCSkHA8sfxrONeFp9QLDnhLBPibiujcJc8QzZ2KzDoDHg252clhDcaQ1cnxG9aJljFqmADsfDFA0wDO3CZrQ1a2IGtqKbjciqzRSd0fjSJA1rsie9iqOr5xgVljy6afNuooOyIVlT21vJWfKUdeLbcq1MwF9NR9xQnp6TqgElSk50p43HsdCl7VKkZd12Ijx43vI72QyQvUm77BV23a6Wh6IXdP9n67StlZllbRiDyGNr0g9S4AHAVga0XofkXFZwgGtsW2J492NC7FAd8AVzIE0SwEaNEI8v9ele8EfNYg3uWVH3JMgi7vGf4N0akxmBAIjpx4dXlxQRGJZerTMzBxY9JAtmZCjH9064Q4uzKxgmpCQg8x06NYx02vknEtYX5O2vgP3gcspGswFEQ57}   \end{align} for $i,j=1,2,3$. Starting with $f^{(2)}$, we have   \begin{align}   \begin{split}   &\sifpoierjsodfgupoefasdfgjsdfgjsdfgjsdjflxncvzxnvasdjfaopsruosihjsdfghajsdflahgfsif_{\TT^3} |f^{(2)}_j|^2 |\vk_j|^{p-2} \,dx    \leq   C    \sifpoierjsodfgupoefasdfgjsdfgjsdfgjsdjflxncvzxnvasdjfaopsruosihjsdgghajsdflahgfsif |f^{(2)}_j|^2\sifpoierjsodfgupoefasdfgjsdfgjsdfgjsdjflxncvzxnvasdjfaopsruosihjsdgghajsdflahgfsif_{3p/(2p+2)}   \sifpoierjsodfgupoefasdfgjsdfgjsdfgjsdjflxncvzxnvasdjfaopsruosihjsdgghajsdflahgfsif |\vk_j|^{p-2}\sifpoierjsodfgupoefasdfgjsdfgjsdfgjsdjflxncvzxnvasdjfaopsruosihjsdgghajsdflahgfsif_{3p/(p-2)}   \\&\indeq   =   C   \sifpoierjsodfgupoefasdfgjsdfgjsdfgjsdjflxncvzxnvasdjfaopsruosihjsdgghajsdflahgfsif f^{(2)}_j\sifpoierjsodfgupoefasdfgjsdfgjsdfgjsdjflxncvzxnvasdjfaopsruosihjsdgghajsdflahgfsif_{3p/(p+1)}^2   \sifpoierjsodfgupoefasdfgjsdfgjsdfgjsdjflxncvzxnvasdjfaopsruosihjsdgghajsdflahgfsif \vk_j\sifpoierjsodfgupoefasdfgjsdfgjsdfgjsdjflxncvzxnvasdjfaopsruosihjsdgghajsdflahgfsif_{3p}^{p-2}   \leq   C   \sifpoierjsodfgupoefasdfgjsdfgjsdfgjsdjflxncvzxnvasdjfaopsruosihjsdgghajsdflahgfsif \ukm\sifpoierjsodfgupoefasdfgjsdfgjsdfgjsdjflxncvzxnvasdjfaopsruosihjsdgghajsdflahgfsif_{3}^2   \sifpoierjsodfgupoefasdfgjsdfgjsdfgjsdjflxncvzxnvasdjfaopsruosihjsdgghajsdflahgfsif \vk\sifpoierjsodfgupoefasdfgjsdfgjsdfgjsdjflxncvzxnvasdjfaopsruosihjsdgghajsdflahgfsif_{3p}^{p}  \leq   C \epsilon    \sifpoierjsodfgupoefasdfgjsdfgjsdfgjsdjflxncvzxnvasdjfaopsruosihjsdgghajsdflahgfsif \vk\sifpoierjsodfgupoefasdfgjsdfgjsdfgjsdjflxncvzxnvasdjfaopsruosihjsdgghajsdflahgfsif_{3p}^{p}  \commaone j=1,2,3,   \end{split}   \llabel{qhX3apbPWsf1YOzHivDia1eODMILTC2mPojefmEVB9hWwMaTdIGjm9PdpHVWGV4hXkfK5Rtci05ekzj0L8Tme2JPXpDI8EbcqV4FdxvrHIeP8CdORJpTiMVEbAunSGsUMWPts4uBv2QSiXIb7B8zo7bp9voEwNRuXJ4ZxuRZYhc1h339THRXVFw5XVW8gaB39mFSv6MzeznkbLHrtZ73hUuaqLvPhgTlNnVpo1ZggmnRAqM3X31ORYSj8RktS8VGOjrz1iblt3uOuEs8Q3xJ1cA2NKoF8o6U3mW2Hq5y6jposxJgwWZ4Exd79JvlcwauoRDCYZzmpabV09jgumebzcbugpatf9yU9iBEyv3UhS79XdImPNEhN64Rs9iHQ847jXUCAufFmsnUudD4Sg3FMLMWbcBYs4JFyYzlrSfnkxPjOHhsqlbV5eBld5H6AsVtrHgCNYn5aC028FEqoWaKSs9uu8xHrbn1eRIp7sL8JrFQJatogZc54yHZvPxPknqRqGw7hlG6oBkzlEdJSEigf0Q1BoCManS1uLzlQ3HnAuqHGPlcIadFLRkdjaLg0VAPAn7c8DqoV8bRCvOzqk5e0Zh3tzJBWBORSwZs9CgFbGo1EFAK7EesLXYWaOPF4nXFoGQlh3pG7oNtG4mpTMwEqV4pO8fMFjfgktnkwIB8NP60fwfEhjADF3bMqEPV9U0o7fcGqUUL10f65lThLWyoXN4vuSYes96Sc2HbJ0hugJMeB5hVaEdLTXrNo2L78fJmehEQ58}\end{align} in which the finiteness of the far right side is ensured by \eqref{EQ14} and~\eqref{EQ37}. The terms with $f^{(1)}$ and $f^{(3)}$ are estimated in the same way, except that for $f^{(1)}$ we utilize the cutoff functions and obtain   \begin{align}   \begin{split}   \sifpoierjsodfgupoefasdfgjsdfgjsdfgjsdjflxncvzxnvasdjfaopsruosihjsdfghajsdflahgfsif_{\TT^3} |f^{(1)}|^2 |\vk|^{p-2} \,dx    \leq    C \bar{\epsilon}   \sifpoierjsodfgupoefasdfgjsdfgjsdfgjsdjflxncvzxnvasdjfaopsruosihjsdgghajsdflahgfsif \vk\sifpoierjsodfgupoefasdfgjsdfgjsdfgjsdjflxncvzxnvasdjfaopsruosihjsdgghajsdflahgfsif_{3p}^{p}   .   \end{split}   \llabel{CMd6LSWqktpMgskNJq6tvZOkgp1GBBqG4mA7tMVp8Fn60ElQGMxjoGWCrvQUYV1KYKLpPzVhhuXVnWaUVqLxeS9efsAi7LmHXCARg4YJnvBe46DUuQYkdjdz5MfPLHoWITMjUYM7Qryu7W8Er0Ogj2fKqXSclGmIgqXTam7J8UHFqzvbVvxNiuj6Ih7lxbJgMQYj5qtgaxbMHwbJT2tlBsib8i7zj6FMTLbwJqHVIiQ3O0LNnLypZCTVUM1bcuVYTejG3bfhcX0BVQl6Dc1xiWVK4S4RW5PyZEVW8AYt9dNVSXaOkkGKiLHhzFYYK1qNGGEEU4FxdjaS2NRREnhHmB8Vy446a3VCeCkwjCMe3DGfMiFopvlzLp5r0zdXrrBDZQv9HQ7XJMJogkJnsDxWzIN7FUfveeL0ljk83TxrJFDTvEXLZYpEq5emBawZ8VAzvvzOvCKmK2QngMMBAWcUH8FjSJthocw4l9qJTVGsq8yRw5zqVSpd9ArUfVDcDl8B1o5iyUR4KNqb84iOkIQGIczg2ncttxdWfLQlNnsg3BBjX2ETiPrpqigMOSw4CgdGPfiG2HNZhLeaQwywsiiAWrDjo4LDbjBZFDrLMuYdt6k6Hn9wp4Vk7tddFrzCKidQPfCRKUedzV8zISvntBqpu3cp5q7J4FgBq59pSMdEonG7PQCzMcWlVR0iNJhWHVugWPYdIMgtXB2ZSaxazHeWp7rfhk4qrAbJFFG0lii9MWIl44js9gNlu46CfP3HvS8vQxEQ73}   \end{align}  Next, using the smallness assumption~\eqref{EQ30} on $\sigma$, we have   \begin{align}   \begin{split}   &\sadklfjsdfgsdfgsdfgsdfgdsfgsdfgadfasdf\biggl[   \sifpoierjsodfgupoefasdfgjsdfgjsdfgjsdjflxncvzxnvasdjfaopsruosihjsdfghajsdflahgfsif_0^{T}\sifpoierjsodfgupoefasdfgjsdfgjsdfgjsdjflxncvzxnvasdjfaopsruosihjsdlghajsdflahgfsif_k^2 \sifpoierjsodfgupoefasdfgjsdfgjsdfgjsdjflxncvzxnvasdjfaopsruosihjsdkghajsdflahgfsif_k^2 \zeta_{k-1}   \sifpoierjsodfgupoefasdfgjsdfgjsdfgjsdjflxncvzxnvasdjfaopsruosihjsdfghajsdflahgfsif_{\TT^3} |\vk_j(t)|^{p-2}\sifpoierjsodfgupoefasdfgjsdfgjsdfgjsdjflxncvzxnvasdjfaopsruosihjsdgghajsdflahgfsif      \sigma_j(t, \uk( t,x))   -   \sigma_j(t, \ukm( t,x))\sifpoierjsodfgupoefasdfgjsdfgjsdfgjsdjflxncvzxnvasdjfaopsruosihjsdgghajsdflahgfsif_{l^2}^2\,dxdt   \biggr]   \\&\indeq\indeq   \leq C \epsilon_{\sigma}^2   \sadklfjsdfgsdfgsdfgsdfgdsfgsdfgadfasdf\biggl[    \sifpoierjsodfgupoefasdfgjsdfgjsdfgjsdjflxncvzxnvasdjfaopsruosihjsdfghajsdflahgfsif_0^{T}\sifpoierjsodfgupoefasdfgjsdfgjsdfgjsdjflxncvzxnvasdjfaopsruosihjsdlghajsdflahgfsif_k^2 \sifpoierjsodfgupoefasdfgjsdfgjsdfgjsdjflxncvzxnvasdjfaopsruosihjsdkghajsdflahgfsif_k^2    \sifpoierjsodfgupoefasdfgjsdfgjsdfgjsdjflxncvzxnvasdjfaopsruosihjsdgghajsdflahgfsif \vk\sifpoierjsodfgupoefasdfgjsdfgjsdfgjsdjflxncvzxnvasdjfaopsruosihjsdgghajsdflahgfsif_{p}^{p}   \,dt   \biggr]   \leq C \epsilon_{\sigma}   \sadklfjsdfgsdfgsdfgsdfgdsfgsdfgadfasdf\biggl[   \sifpoierjsodfgupoefasdfgjsdfgjsdfgjsdjflxncvzxnvasdjfaopsruosihjsdfghajsdflahgfsif_0^{T}\sifpoierjsodfgupoefasdfgjsdfgjsdfgjsdjflxncvzxnvasdjfaopsruosihjsdlghajsdflahgfsif_k^2 \sifpoierjsodfgupoefasdfgjsdfgjsdfgjsdjflxncvzxnvasdjfaopsruosihjsdkghajsdflahgfsif_k^2   \sifpoierjsodfgupoefasdfgjsdfgjsdfgjsdjflxncvzxnvasdjfaopsruosihjsdgghajsdflahgfsif \vk\sifpoierjsodfgupoefasdfgjsdfgjsdfgjsdjflxncvzxnvasdjfaopsruosihjsdgghajsdflahgfsif_{3p}^{p}   \,dt   \biggr]   .   \end{split}   \llabel{Yw9cEyGYXi3wi41aIuUeQXEjG3XZIUl8VSPJVgCJ3ZOliZQLORzOFVKqlyz8D4NB6M5TQonmBvikY88TJONaDfE2uzbcvfL67bnJUz8Sd7yx5jWroXdJp0lSymIK8bkKzqljNn4KxluFhYLg0FrO6yRztwFTK7QRN01O21ZcHNKgRM7GZ9nB1Etq8sqlAsfxotsl927c6Y8IY8T4x0DRhoh0718MZJoo1oehVLr8AEaLKhyw6SnDthg2HMt9D1jUF5b4wcjllAvvOShtK806ujYa0TYO4pcVXhkOOJVtHN98Qqq0J1HkNcmLS3MApQ75AlAkdnMyJMqACerDl5yPys44a7cY7sEp6LqmG3V53pBs2uPNUM7pX6sy95vSv7iIS8VGJ08QKhAS3jIDNTJsfbhIiUNfeH9Xf8WeCxmBLgzJTIN5NLhvdBOzPmopxYqM4VhkybtYga3XVTTqLyAHyqYqofKP58n8qR9AYrRRetBFxHGg7pduM8gm1TdplRKIW9gi5ZxEEAHDeAsfP5hbxAxbWCvpWk9caqNibi5A5NY5IlVAS3ahAaB8zzUTuyK55glDL5XO9CpORXwrEV1IJG7wEgpOag9zbJiGeT6HEmcMaQpDfyDxheTNjwfwMx2CipkQeUjRUVhCfNMo5DZ4h2adEjZTkOx946EeUIZv7rFL6dj2dwgRxgbObqJsYmsDqQAssn9g2kCb1MsgKfx0YjK0GlrXO7xI5WmQHozMPfCXTmDk2Tl0oRrnZvAsFr7wYEEQ59}   \end{align} Therefore,   \begin{align}   \begin{split}   &\sadklfjsdfgsdfgsdfgsdfgdsfgsdfgadfasdf\biggl[\sup_{0\leq t\leq T}\sifpoierjsodfgupoefasdfgjsdfgjsdfgjsdjflxncvzxnvasdjfaopsruosihjsdgghajsdflahgfsif\vk_j(t,\cdot)\sifpoierjsodfgupoefasdfgjsdfgjsdfgjsdjflxncvzxnvasdjfaopsruosihjsdgghajsdflahgfsif_p^p   +\sifpoierjsodfgupoefasdfgjsdfgjsdfgjsdjflxncvzxnvasdjfaopsruosihjsdfghajsdflahgfsif_0^{T} \sifpoierjsodfgupoefasdfgjsdfgjsdfgjsdjflxncvzxnvasdjfaopsruosihjsdfghajsdflahgfsif_{\TT^3} | \nabla (|\vk_j(t,x)|^{p/2})|^2 \,dx dt\biggr]   \\&\indeq   \leq C   \sadklfjsdfgsdfgsdfgsdfgdsfgsdfgadfasdf\biggl[   \sifpoierjsodfgupoefasdfgjsdfgjsdfgjsdjflxncvzxnvasdjfaopsruosihjsdgghajsdflahgfsif\vk_{0,j}\sifpoierjsodfgupoefasdfgjsdfgjsdfgjsdjflxncvzxnvasdjfaopsruosihjsdgghajsdflahgfsif_p^p   +   (\epsilon+\bar{\epsilon}+\epsilon_{\sigma})\colb\sifpoierjsodfgupoefasdfgjsdfgjsdfgjsdjflxncvzxnvasdjfaopsruosihjsdfghajsdflahgfsif_0^{T}    \sifpoierjsodfgupoefasdfgjsdfgjsdfgjsdjflxncvzxnvasdjfaopsruosihjsdgghajsdflahgfsif \vk\sifpoierjsodfgupoefasdfgjsdfgjsdfgjsdjflxncvzxnvasdjfaopsruosihjsdgghajsdflahgfsif_{3p}^{p}   \,dt   \biggr]   \commaone j=1,2,3.   \end{split}   \llabel{JHCd1xzCvMmjeR4ctk7cS2fncvfaN6AO2nIh6nkVkN8tT8aJdb708jZZqvL1ZuT5lSWGo08cLJ1q3TmAZF8qhxaoYJC6FWRuXHMx3Dcw8uJ87Q4kXVac6OOPDZ4vRtsP01hKUkdaCLBiPSAtLu9WLoyxMaBvixHyadnqQSJWgSCkF7lHaO2yGRIlK3aFZenCWqO9EyRofYbkidHQh1G2vohcMPoEUzp6f14NioarvW8OUc426ArsSo7HiBUKdVs7cOjaV9KEUtKne4VIPuZc4bPRFB9ABfqclU2ct6PDQudt4VOzMMUNrnzJXpxkE2NB8pfJiM4UNg4Oi1gchfOU62avNrpcc8IJm2WnVXLD672ltZTf8RDwqTvBXEWuH2cJtO1INQUlOmEPvj3OOvQSHxiKc8RvNnJNNCC3KXp3J8w50WsOTXHHhvL5kBpKr5urqvVFv8upqgPRPQbjCxme33uJUFhYHBhYMOd01Jt7ySfVpF0z6nCK8grRahMJ6XHoLGu4v2o9QxONVY88aum7cZHRNXHpG1a8KYXMayTxXIkO5vV5PSkCp8PBoBv9dBmepms7DDUaicXY8Lx8IBjFBtke2yShNGE7a0oEMFyAUUFkRWWheDbHhAM6Uh373LzTTTxxm6ybDBsIIIAPHhi837ra970Fam4O7afXUGrf0vWe52e8EPyBFZ0wxBzptJf8LiZkdTZSSPpSzrbGEpxb4KXLHLg1VPaf7ysvYsFJb8rDpAMKnzqDg7g2HwCruQNDBzEQ60}   \end{align} If $\epsilon$, $\bar{\epsilon}$, and $\epsilon_{\sigma}$ are sufficiently small, then by \eqref{EQ14} this estimate implies  \begin{align}  \begin{split}  &\sadklfjsdfgsdfgsdfgsdfgdsfgsdfgadfasdf\biggl[\sup_{0\leq t\leq T}\sifpoierjsodfgupoefasdfgjsdfgjsdfgjsdjflxncvzxnvasdjfaopsruosihjsdgghajsdflahgfsif\vk(t,\cdot)\sifpoierjsodfgupoefasdfgjsdfgjsdfgjsdjflxncvzxnvasdjfaopsruosihjsdgghajsdflahgfsif_p^p  +\sifpoierjsodfgupoefasdfgjsdfgjsdfgjsdjflxncvzxnvasdjfaopsruosihjsdfghajsdflahgfsif_0^{T}\sum_{j=1}^{3} \sifpoierjsodfgupoefasdfgjsdfgjsdfgjsdjflxncvzxnvasdjfaopsruosihjsdfghajsdflahgfsif_{\TT^3} | \nabla (|\vk_j(t,x)|^{p/2})|^2 \,dx dt\biggr]  \leq C  \sadklfjsdfgsdfgsdfgsdfgdsfgsdfgadfasdf\bigl[  \sifpoierjsodfgupoefasdfgjsdfgjsdfgjsdjflxncvzxnvasdjfaopsruosihjsdgghajsdflahgfsif\vk_{0}\sifpoierjsodfgupoefasdfgjsdfgjsdfgjsdjflxncvzxnvasdjfaopsruosihjsdgghajsdflahgfsif_p^p    \bigr]   ,  \end{split}  \label{EQ61}  \end{align} which leads to \eqref{EQ55}, completing the proof. \end{proof}   \par \begin{Remark}\label{R01}   We point out that, without the smallness assumption on $\sigma$, i.e., assuming only $\epsilon$ and $\bar{\epsilon}$ are small, we would have instead    \begin{align*} \begin{split} &\sadklfjsdfgsdfgsdfgsdfgdsfgsdfgadfasdf\biggl[\sup_{0\leq t\leq T}\sifpoierjsodfgupoefasdfgjsdfgjsdfgjsdjflxncvzxnvasdjfaopsruosihjsdgghajsdflahgfsif\vk(t,\cdot)\sifpoierjsodfgupoefasdfgjsdfgjsdfgjsdjflxncvzxnvasdjfaopsruosihjsdgghajsdflahgfsif_p^p +\sifpoierjsodfgupoefasdfgjsdfgjsdfgjsdjflxncvzxnvasdjfaopsruosihjsdfghajsdflahgfsif_0^{T}\sum_{j=1}^{3} \sifpoierjsodfgupoefasdfgjsdfgjsdfgjsdjflxncvzxnvasdjfaopsruosihjsdfghajsdflahgfsif_{\TT^3} | \nabla (|\vk_j(t,x)|^{p/2})|^2 \,dx dt\biggr] \leq C \sadklfjsdfgsdfgsdfgsdfgdsfgsdfgadfasdf\biggl[ \sifpoierjsodfgupoefasdfgjsdfgjsdfgjsdjflxncvzxnvasdjfaopsruosihjsdgghajsdflahgfsif\vk_{0}\sifpoierjsodfgupoefasdfgjsdfgjsdfgjsdjflxncvzxnvasdjfaopsruosihjsdgghajsdflahgfsif_p^p + \sifpoierjsodfgupoefasdfgjsdfgjsdfgjsdjflxncvzxnvasdjfaopsruosihjsdfghajsdflahgfsif_0^{T}  \sifpoierjsodfgupoefasdfgjsdfgjsdfgjsdjflxncvzxnvasdjfaopsruosihjsdgghajsdflahgfsif \vk\sifpoierjsodfgupoefasdfgjsdfgjsdfgjsdjflxncvzxnvasdjfaopsruosihjsdgghajsdflahgfsif_{p}^{p} \,dxdt \biggr], \end{split} \end{align*} which by Gr\"onwall's lemma implies     \begin{align} \begin{split} &\sadklfjsdfgsdfgsdfgsdfgdsfgsdfgadfasdf\biggl[\sup_{0\leq t\leq T}\sifpoierjsodfgupoefasdfgjsdfgjsdfgjsdjflxncvzxnvasdjfaopsruosihjsdgghajsdflahgfsif\vk(t,\cdot)\sifpoierjsodfgupoefasdfgjsdfgjsdfgjsdjflxncvzxnvasdjfaopsruosihjsdgghajsdflahgfsif_p^p +\sifpoierjsodfgupoefasdfgjsdfgjsdfgjsdjflxncvzxnvasdjfaopsruosihjsdfghajsdflahgfsif_0^{T}\sum_{j=1}^{3} \sifpoierjsodfgupoefasdfgjsdfgjsdfgjsdjflxncvzxnvasdjfaopsruosihjsdfghajsdflahgfsif_{\TT^3} | \nabla (|\vk_j(t,x)|^{p/2})|^2 \,dx dt\biggr] \leq C_T \sadklfjsdfgsdfgsdfgsdfgdsfgsdfgadfasdf\bigl[ \sifpoierjsodfgupoefasdfgjsdfgjsdfgjsdjflxncvzxnvasdjfaopsruosihjsdgghajsdflahgfsif\vk_{0}\sifpoierjsodfgupoefasdfgjsdfgjsdfgjsdjflxncvzxnvasdjfaopsruosihjsdgghajsdflahgfsif_p^p \bigr] , \end{split} \label{EQ62} \end{align} for any fixed~$T$. \end{Remark}   \par \cole \begin{Lemma}[A pointwise $L^{3}$ control] \label{L09}   Let $k\in\mathbb{N}_0$. Suppose that $\{\vk_0\}_{k\in\NNp}$ is a sequence of initial data of~\eqref{EQ34} satisfying the assertions in Lemma~\ref{L05}. Then the solution to \eqref{EQ34} satisfies  \begin{equation}    \sup_{t\in [0,\infty)}\sifpoierjsodfgupoefasdfgjsdfgjsdfgjsdjflxncvzxnvasdjfaopsruosihjsdgghajsdflahgfsif \vk(t)\sifpoierjsodfgupoefasdfgjsdfgjsdfgjsdjflxncvzxnvasdjfaopsruosihjsdgghajsdflahgfsif_{L^{3}}    \leq    \frac{\bar{\epsilon}}{2^{k-1}}    \label{EQ71}   \end{equation} $\PP$-almost surely. \end{Lemma} \colb \par \begin{proof}[Proof of Lemma~\ref{L09}] Fix $k\in\mathbb{N}_0$. Denote $M=\overline \epsilon/2^{k-1}$ and   \begin{equation}    \alpha( t)=\sifpoierjsodfgupoefasdfgjsdfgjsdfgjsdjflxncvzxnvasdjfaopsruosihjsdgghajsdflahgfsif \vk( t)\sifpoierjsodfgupoefasdfgjsdfgjsdfgjsdjflxncvzxnvasdjfaopsruosihjsdgghajsdflahgfsif_{L^{3}}     \comma t\geq 0       .    \llabel{Z5SNMayKB6RIePFIHFQawrRHAx38CHhoBGVIRvxSMYf0g8hacibKG3CuSl5jTKl42o6gAOYYHUB2SVO3Rc4whR8pwkrrANA4j7MfcEMal4HwKPTgZaZ9G8sevuwIAhkhR8WgafzJA0FVNmSCwUB0QJDgRjCSVSrsGMbWABxvzOMMycNSOylZzwFiNPcmfQhwZPanLp01EUVHMA2dE0nLuNVxKxco7opbQzRaAlowoVtorqU5eUXtEl1qh1IPCPEuVPxcnTwZvKJTpHZpqxXOFaGsrQNPnuqc9CMD8mJZoMOCy6XHjWAfEqI95ZjgcPdV8maWwkHlM30VwDjXlx1QfgyLFxe5iwJDnJIrUKL6eCVth3gX8RLACCC2q5kMYXo8NsDfAn3OsapZT8U3FdMIEDKqMo0YvbwCGMR66XTyIOtLUuC6cmcOFsvpWTniQmu0PeHEF9ImolIuThHWHwh8Jz4hC0rK2GdNzLXiEY7VuQfRbXpiQnPps9gMA8mWkyXsYFLoiRtl2Kl2pI9bSnyi07mUZqhEsBOCgI4F5AFFdjX3wf0Wu2Yqddp2ZUkjeFMAxnDlsut9qzbyRgDWrHldNZewzEK1cSwWJZywloSof6zVDAB6er0o2HZY1trZhBuL5zYzrAUdMKXVKGWKIHOqqx1zj8tlpxuUD83eLUerjxfHNMZlaqZVhT6Jku15FdLvdeo087AsGC8WdoMnf4dToHw47hglTqjKtAlwR9ufOhLKTDgWZhxHFX5gU5uN2S6esPEQ72}   \end{equation} Assume to the contrary that the assertion does not hold. Then \begin{equation} \mathbb{P}\Bigl( \sup_{t\in [0,\infty) }\alpha( t)  >M \Bigr)=\bigcup_n \mathbb{P}\Bigl( \sup_{t\in [0,\infty) }\alpha( t) >M+\frac{1}{n} \Bigr)>0   , \llabel{lxKpXzBmgyWUy5D01WD88a4YmWRfdmev1dBvHOmhTBqurAgTC6yrrRBPn9QfZ9T4mwIh3xjAtkiMlAlTd6fSQ5iQBBY6OErT3gf0DKeECnjgcTXAL8grKBpfcJvq4fpIhWGFSdh6LOqg0ao9AjakqEZKgv95BAqvCSJJgo1Lzsv5yhPQkMpPWnsXvHNZAUQt1DpO47V7AR6JCTR9fnH6QVYwjdZTR3TZSCdOcidDYXYiztVEgNK6hZWtLooE11MiqywCo9kUjdnt1CCc80eJHxMWn7GlDI9yCp9xs9zknCbFsjlKydvYatKpvKJpySVeTP1zRB9N91vA0XXSacVlNzZX3jRgbDjcscBBveadZerkNNT2ni9PDo6HyDNuAvUoYNaIuQ6oUCEi5k5kBw1fwktQDSG4Ky7U2nXSKlOez0PJ0v1SNnMkUdmxxN5t7vhLQxWUulVutc6EMFPa0mIkRDVwluKitmTncyT58CDjRPunXB74DYSJKWEPYi0YxlyAd7HGWyyCeJzd8ht2NnHGfOmsDLbWqhYk2v3J7j2nbB4taYDMN0OkJTdkNPO7JvkTRFYwud2MZ91SZPVQcLlvrOcIN92COu4QpaM7ShSsg1qs8uijWXMMnX9760otPu5BJwtxkMVH4wuj37tRdB7Za2FeTvXLlkC0kZ1ZRCZvbcVw9SRuUimZYbIyqO0qKkkirgpvLzBS44Rwj1NZRJHOafvDTKV3ePSJJ0wuXjKzgeXa11GuCRiRVPRSUNxSqionXMEQ75} \end{equation} and thus there exists $\overline M>M$ such that $\mathbb{P}(\tau<\infty)>0$, where $\tau$ is the first hitting time of $\alpha( t)$ at~$\overline M$.  Choose any $M_1,M_2$ satisfying   \begin{equation}    M < M_1 < M_2 < \overline M.    \llabel{k3f8cKO4inK7IfRUJ0MW6ZdcMT3LaSjZ4IqtQIFDukYcYz700unb8AtSmg2kMKAAL7DB4cgDSBFHX5GPcHAqCytP3oJRryBaylYND2HGJZrTgUURyUwzCli2F9vvyflNg1nR5Y1nxJC235JxySQAmanUoZh1VvDoMy3RL9pUIemC79uRdoMV0hz3sKl3uSB9WEO7EFbVYWJMeDYZe6UuJQ2rbloFIc6ca1hFEdw2d4sSTrHnhAIlQ1o9RphHtZ9C4DInJMbOmYJatZfwEAKGDpQb2Mx5o8ndJnyvrUaP1lkNOGdleO90C3QpE1fEXgQ5Y2APzGVPRJn6rPVXg9U8d9updt0YZpJ4i1h1WTmixB5H0ndNf7UYbKCXm1RHvRl6IgFtPgkhxdX95jIOZ0qtx9xEmzdF1L5sNPY0CKoPBSK6S7faLCutDrBVCtB3NykryNUqACBCto7OVcAjjsNMjfpDA9w2r6CzmQhf2xmwR0HvIhiIjhHoZ4obATsWkJgUBX5JbTtFr3fDq6DjRCxJo6kZKTUfJNw8uCuyAAxqjBCsvLawMMS37OhxIQcaW0SnnMPdZQLjtIIpnIneWsWiUQpO69palKXNRlv9bYG06ptJE7WLpWiXneaUhEszoGmQquLLn5bXdyFia5iLrg7PABBN4v8QapYAv8hGhMd7E10KVuQl78KRe38xdUEFe15K72PTLwYDkutECGADCEBMbcFvLgnnrbdwldq6CS8wVyBzPGZaC3fGkrvmhMj9uvRdnSEQ88}   \end{equation} We claim that there exist deterministic $t_0,r,s\in \mathbb{Q}_{+}=\mathbb{Q} \cap (0,\infty)$ such that   \begin{align}   \begin{split}    &    M_1 \leq \alpha(t) \leq M_2     \comma t\in [t_0,t_0+r],    \\&   M_1 \leq \alpha(t)     \comma t\in [t_0,t_0+r+s], \text{~and}   \\&   \alpha(t)\geq \overline M   \text{~for some $t\in[t_0+r,t_0+r+s]$}   \end{split}    \llabel{YXXg12IkXVNSN3601pWdykiU6kaUwDUZ2G8r5XiZXMQ7AGrplYAPlwn11dm00Jodc1h7zFn5rLbVOHMDh0QggiSOKll3vzZ0A6hDO56OyuNBgfzTkNNyR28PsJUDsofPaXgqBUKo9tXhTwgIFAxgS43mPTRh5QLfBBrLyi8we8dXmJnkR7DcCII4Df1yrovtwK6Zq8Fay5DrbFolZgiNNUKQko199y43VW46dUhwt4dEWECfqszQjuZcFNqRA8EBKTjxXvfBSuPr9IelzgCVxKdowNHtbT6jKdPL1MYPfjfmZiLHDZUfPBhjt64XadgQnI2ylWjZqHTp3H9LIGIdsX0liHs3aE1qHcNYwAfL2aJAMnqOiSdxF6cG4bEx1aZjpeJsDcYCD98Z7vMYv77Vfx3eVQakI2aPi3UhSEKfsDNOley7xrpP4S9FoFg8deOZeMJ5PQWSMloZjHqNXrtCh7pqHNFuq1MF5tPbBVgGwFtRYhAi2q52Rw4dFk76zGcROdFIXft1LK3fxk0xMVmqt2h7rqf9OlF4gjjR2B8FxKi9pwg5yRY8XMXIWISOjbcsRKniRqLwJzkAU7oq6tBKpEv4ESlNOy9u1tciXJCAYUOps6OAh6W4ZxV8lo2dueBiLmZrIw6dFUG4w8PcdiJSFmSwE9KhkgSZ75dycS17zskwXxP9UVlmzJwi5E0CZMvlQLeFZblh53SABDvgFzmlZH5lJ42UOcRojWmp7FtO4atvrDjVQbvSNkhDbNl8R21vnhEQ90}   \end{align} on an $\mathcal{F}$-measurable set $\tilde \Omega_0$ of positive probability. To prove this, first note that by the continuity of the solution trajectories,   \begin{align}   \begin{split}    \{\tau<\infty\}    =    \bigcup_{t_0,r,s\in\mathbb{Q}_{+}}      (A_{t_0,r} \cap  B_{t_0,r,s} \cap C_{t_0,r,s})      ,   \end{split}    \label{EQ91}   \end{align} where   \begin{align}   \begin{split}    A_{t_0,r}      &=    \bigl\{     \omega:     M_1 \leq \alpha(\omega,t)\leq M_2      \text{~for~all~$t\in[t_0,t_0+r]\cap \mathbb{Q}_+$}    \bigr\}    =    \bigcap_{t\in[t_0,t_0+r]\cap \mathbb{Q}_{+}}     \alpha_t^{-1}([M_1,M_2])   \end{split}    \llabel{X3ILYS65pr9OAdnA3j6KEL5tNaNVmsKvlGmy2kMyVvsb73RcLUNGzi6wlc1uFWkYrpXEmfkPv4tQozW2CHUXS2kCG3CfbZ7cLMohJXIukasc5F0vF8GoNZkeI6DEv2OlWLJpuuOiSXadleN1cgyW4u0bp8TDtYFQI8kp29nLjdvVTPfrKxCVr7pdxnVwdHHQubf5O4ixYrddMbrhd60rN8GL1TGfyeCQmNaJN3fJgn0we17GkQAg1WAj6l87vzmOzdKQ1HZ8qATPMo1KAULCHIjKsRXwKT5xYB9iwCmzcM5nlfbkMBhe7ONI5U2NrE5WXsFl5mkw3TzbJicLnBrAjcGa9wJZGTj7YmbbEr0cgxOxs75shL8m7RUxHRQBPihQ1Zg0p7UIyPsopiOaehpYI4z5HbVQxsth4RUeV3BsgK0xxzeyEFCVmbOLA52noFxBk8r1BQcFeK5PIKE9uvUNX3B77LAvk5Gdu4dVuPux0h7zEtOxlvInnvDyGe1qJiATORZ29Nb6q8RI4V3Dv4fBsJDvp6ago5FalZhU1Yx2evycybq7Jw4eJ9oewgCma6lFCjsOyzeoXOyIagDo4rJPDvdVdAIPfvaxOIsle7l60zfITnPR5IE34RJ1dqTXj0SVuTpTrmkFSn2gIWtUvMdtZIWIZTo2aJpeFRGmLIa5YG6yn0LboerwMPeyvJZHroCE38u158qCnrhQWMU6v8vncMbFGS2vM3vW5qwbO6UKlUBS9Y2oL2juyOlk3XJ7mfyQMGEQ92-1}   \end{align} with   \begin{align}   \begin{split}    B_{t_0,r,s}      &=    \bigl\{     \omega:     M_1 \leq \alpha(\omega,t)      \text{~for~all~$t\in[t_0,t_0+r+s]\cap \mathbb{Q}_+$}    \bigr\}    =    \bigcap_{t\in[t_0,t_0+r+s]\cap \mathbb{Q}_{+}}    \alpha_t^{-1}([M_1,\infty))   \end{split}    \llabel{cRyAVSc0Yk8bizEBIGDYdOGoW6emVaisrFm5Ehk6nwh98Pn1pr7Od6qGjlJObLuDe0UqZTHY6w0iZ67Kfw5cPZ1ZFpvG095DYKQTHJ7t5HPbc8WevwqlBtMEFsBgDSaccQONR6sBXLp8nu9ygZ47NBUCknTcXgssyAgKe8nTCWZjy1Ygl83LTLRtp6iSFIDk1BtU6O9pxcNWt3FHEPLmD7TXtkNZr0rYhB8frpuqccVbXrPCHjJbJoqMK7O6CvwuATercnN2SpjTAaMyFNpa9GeKl351nDs3KVr6WMCTYS0zsABO0SwHhCqG0qk62kpIM5YjYCeM76VcZ0cFJZTEHZy5LjzlsDRtfvNE1e4QcGKXy7y7Hp2oXfzX6DZbs8nmLkLREmO32NF7SezN3kn7UZI29gSEw2zk5YBa0aVY6umS1ABaKt5Ahkx6Qz8B2SpqdGFQV8bqASLCGwcdkXBaIL44gxbitC1txjgfQ6AKORWGJKF2ym4xY9EDRne7ODExbjdiFdVElumlJwtcNTiKoLEffm6QI1hmfEXqg4Sd3b8GBBQG7Fsyq6g8ISR17dioqXqfE2jRYpBJkWP2HfthwNrhDsCJ6Xj2O6ONi9jWzM1HHOlFm2TfMmujPhiVujhpDSS5vLdDjayX8BpBeeKnwz6Po6KUz01etrb2z8jwiJy9GB3F7RNKgEUWkFVuxBJpntmcqZF7NIVIiSWXEy2B7tRn6afnGtzp09oO1UsLvofZRiQcxV7tFgjBZ9mEssTsKSEQ92-2}   \end{align} and   \begin{align}   \begin{split}    C_{t_0,r,s}      &=    \bigcap_{\epsilon\in\mathbb{Q}_+}    \bigcup_{t\in(t_0+r,t_0+r+s)\cap \mathbb{Q}_+}    \{\omega:\alpha(\omega,t)\geq \overline M-\epsilon\}    =    \bigcap_{\epsilon\in\mathbb{Q}_+}    \bigcup_{t\in(t_0+r,t_0+r+s)\cap \mathbb{Q}_+}    \alpha_t^{-1}([\overline M-\epsilon,\infty))    ,   \end{split}    \llabel{ENx0N4gQ0ubFQGfRGM1dMTXJ53WWic7AlLkC76gLZm4Sn9zosCBe9fDlmE6Lq9lJY0heWK3oKvFizSvWqOwh8NvAIcoWO4QfR4CxSOGYXBYo2zigpP0nYii2uivvjhp9zpnfsROtkKm3afLxSh6D4eOoQ0TcRScNt9ZGUvxH0aRPrHhuZrJtGLHt533m8j3jJeSbCYI84loykM7iaY1pC09AJJ8puhnm6yTQO1WjgC5NvJJcAuK0Ib3YxCEleEHjdLl8L677tACt5w4MTfsZXC730oY7Lm9tGfKZQ3UFhPK7LqBEB1tnde3A05ZPCaVVMO4Wp1vl5NWYVAS3A6ppl8D702kndT3RzBToCoeQzgwwkUcHv6n8yGW6ADNujYLGNTMO4OrlG24pCbTFbaEWwLfvUlF7kpNGq0kmJd9voUPmdlK4yovDYNYPQhxO4g1sAZRKOsjpI9NvaBCW4EssGGiWLGR9RPwoE9NROT90cKFKyx5ZwzWBEnqOyITIPyL1ILo7eJfQdKiR1VJwiLx70OjSKHIShGIXvmm10FDxlINtLdHGzZ4OBmclD6VUmJvbO5I3EI1imd1hjxMLf9WNnq0cfLTeFzxKEQHL5u9r94VqDSQs1OhLadtGVbTmqRCt4RQKTJXDb5X7vc4gRO9owaiUdwAc4uIKr2eA0fjhHu7HukD1pbwm7YzSrpRxdBpLYAGi7aIh0aZot6yMJxsfBpw3JbbH4lIXf9s3jFDyyYSC8wNFPYZiO2AJBZKdaAD3huEQ92-3}   \end{align} where we denoted   \begin{equation*}    \alpha_t(\omega) = \alpha(\omega,t)     \comma \omega\in\Omega     ,   \end{equation*} for any $t\geq0$. Now, consider the equality~\eqref{EQ91}. Since the right-hand side has a positive probability, by $\mathbb{P}(\tau<\infty)>0$, there exist $t_0,r,s\in\mathbb{Q}_+$ such that $\PP(A_{t_0,r} \cap  B_{t_0,r,s} \cap C_{t_0,r,s})>0$. Set $\tilde \Omega_0=A_{t_0,r} \cap  B_{t_0,r,s} \cap C_{t_0,r,s}$, and the claim is established. \par Now, fix $t_0$, $r$, $s$, and $\tilde \Omega_0$ as in the claim. Consider the event   \begin{equation}    E=\bigl\{       \omega:       M_1\leq\alpha(t_0)\leq M_2      \bigr\}\in\mathcal{F}_{t_0}      .    \llabel{4WI0h1zMut1AyXXJ1IWtVOUvgjMEdr3H6ZASxZeGCNEG9U8XLLC8otnqMrr5llwno3yV0tSlPW1LjWvkj8W91VA8T6CXqWMs7WjvwZre71Buv1RkHBPRkBwhsucRyoHn8BL0fGPm3ANnQX1MqDgLrFJmPZ1sTLo46ZhffTFKGFDSIMoV0UtBu6d0TIkEwRX13S45chg6JyD2aHOEAbD98fDOYbXljuO3gIzi2BaEYczmnuOx94YNAzGeE0iHiqoXEtl8Ljgxh3GO5duEBzl0uoVOZLI0xiECVeNFhcQXgqqN2eSobAT4IXuflDOankxzPUd37eYwmvrud8eluLCBkwtTomGSSpWw7m10SlOpCJqqm6hQLau4qM8oyxlORGCT1csPUcGqV6zE9cPoOKWCTPvz9PMUTRd9e0f5g8B7nPURfC1t0u4HPVzeEgMRZ9ZVv1rWPIsPJf1FLiIa2TyZq8haK9V3qN8QsZt7HROzBAheG35fAkpgPKmIuZsYWbSQw6Dg25gM5HG2RWpuipXgEHDco7Ue6aFNVdPf5d27rYWJCOoz1ystngc9xlbcXZBlzAM9czXn1Eun1GxINAXz44Qh8hpfeXqdkdzuIRYMNMcWhCWiqZ5sIVyj9rh1A47uLkTwfyfryoiDx1enHSeip1vsO4KM6I5nql8i6XMK0rks0J59i2gzeHrjL1GbnITIBUV33u1He2frwh8SlMFfxURXLEeSH3mRUZPo4qNIEiCDpZPg9jUAiCbJ3yZnYviypDEQ76}   \end{equation} Then $u \mathds{1}_{E}$ solves our equation on the deterministic time interval $[t_0,t_0+r+s]$. But due to properties of the cutoff, it actually solves the deterministic heat equation on $[t_0,t_0+r+s]$, and thus $\alpha$ is nonincreasing on $(t_0,t_0+r+s)$ on the event~$E$. Hence,   \begin{equation}    \overline \Omega    =    \left\{     \omega\in E:     \alpha \geq (M_2 + \overline M)/2     \text{~at some point~}     t\in (t_0,t_0+r+s)    \right\}    \in\mathcal{F}_{t_0+r+s}\subseteq\mathcal{F}    \llabel{Y6UdI0D0Uvh8B5pzubybLzK6Okxiv1n6JFvXnAnLZkB2NBL7XrJXMMKs8CZsG212rnKHzisWyrbRdhAevTuCXC238cSHDCDRuSfLZ6OcD9ILJinlaki39ZA5i8M5PHQYRu9c4V1mhLVzeTTXVpgYHoa1Xj0tDpTKKrzAmUdJKYBVzYf2fuZHOiTzct88pOcWnFTvpWY7rmUX9mStlDy20kd9AusOJf9flOWhZwlWiHSoUyCdlxpPLdWIUhiT3LHokLLeFWnzMXpp1rcNzMdnmqvOOqqu8r9aHjZugynu1WR2Kt1bLAm4xFIJ56unzROpYf9QL0MlwkoJUIPqs7J4FuANaH5PLCga54bD9TmxtzHi82iDykpSWvkcvxuDk7rWAaUjbhIKIz4yML1wz7iFCxXO8nis8uKjb3uglgImLMaYSJSdAKU9AVHZ8G1n2uexpWzyZdj2gvCTeT5702uSeZ9BCuRQVhFJQNMZ3S2esCCUqV3CHgWPGx9gzqE69aQz5VOSuRPnXm57MllKGVzazF7XDlGQJ2Gawooq71nMVyltPDG7L0Ld0OErw0kLAYXZzR6XaeileVHuk7xcwyBOlV2JrQxWjPbeLaiIYqHqpbcvoLZVz1ytufBpcJ4MVecfeep7bry2UjHFLwzWFuMMdhVHxDX8Gsv0RQLK8YqbATHFEX0UkxjLiSKH3Rr5umoxSEQd0Myjy6DdKVFmCU7TfBI0xlsOWk20HafdxvncOzFEKskwmIHNHlUNhKYaDnErMsEQ77}   \end{equation} satisfies $\mathbb{P}(\overline \Omega)=0$. However, $\tilde \Omega_0\subseteq\overline \Omega$, which is a contradiction with $P(\tilde \Omega_0)>0$.  \end{proof} \par Let $\tau_k$ be the first time when $\sifpoierjsodfgupoefasdfgjsdfgjsdfgjsdjflxncvzxnvasdjfaopsruosihjsdgghajsdflahgfsif \vk\sifpoierjsodfgupoefasdfgjsdfgjsdfgjsdjflxncvzxnvasdjfaopsruosihjsdgghajsdflahgfsif_{L^{3}}$ hits $  \fractext{\bar{\epsilon}}{2^{k}}$, and $\sifpoierjsodfgupoefasdfgjsdfgjsdfgjsdjflxncvzxnvasdjfaopsruosihjsdjghajsdflahgfsif_k$ be the first hitting time of $\sifpoierjsodfgupoefasdfgjsdfgjsdfgjsdjflxncvzxnvasdjfaopsruosihjsdgghajsdflahgfsif \vk\sifpoierjsodfgupoefasdfgjsdfgjsdfgjsdjflxncvzxnvasdjfaopsruosihjsdgghajsdflahgfsif_{L^{6}}$ at~$M_k$. Also, consider the stopping times   \begin{equation}    \tau^{k}    = \tau_{1} \wedge \sifpoierjsodfgupoefasdfgjsdfgjsdfgjsdjflxncvzxnvasdjfaopsruosihjsdjghajsdflahgfsif_{1} \wedge \cdots \wedge \tau_{k} \wedge \sifpoierjsodfgupoefasdfgjsdfgjsdfgjsdjflxncvzxnvasdjfaopsruosihjsdjghajsdflahgfsif_{k}     \comma k\in\mathbb{N}_0    \label{EQ85}   \end{equation} and their limit   \begin{equation}    \tau=\inf_k \tau^k=\lim_k \tau^k    .    \label{EQ70}   \end{equation} Note that, for every $k\in\mathbb{N}_0$, the solution $\vk$ to \eqref{EQ34} satisfies   \begin{equation}   \sup_{t\in [0,\tau)}\sifpoierjsodfgupoefasdfgjsdfgjsdfgjsdjflxncvzxnvasdjfaopsruosihjsdgghajsdflahgfsif \vk(t)\sifpoierjsodfgupoefasdfgjsdfgjsdfgjsdjflxncvzxnvasdjfaopsruosihjsdgghajsdflahgfsif_{L^{3}}   \leq   \frac{\bar{\epsilon}}{2^{k}}   \comma k\in\NNz   \llabel{9hNUpc2L8ipcs8XwvjOeX35ehokgG6nyzyevqmZYDfdhQXEFk8qgF3J8SySTL1emZ1hqVGLs1rRD8Fs6u4NCoeyUCERC0xh4x67erCg8lf8J1Go29fRN6HMuUIwacxm2Af9sfbNov4CqQKcFloP7BRW85rUNd2BgxWkOpDr11KRKZlaHXQoN7OaohaotKxJGOUty96bCDK33hIEKUogH2RwEYISuKzS4pVz7BzdHrIjRgCvSJUs1x1D3RRqk9icqLUQ9TeTiCnifS9NpGBJFlPNUFbB5IU956UG9PrSw9oNNnvK4g3GFnUxtWPNXVvue5rAIN0eoRq2XpZNDRpQUwIjhrAXb13L7G4OSy8fznjOJBbaVLPazvFW8uCY5if8VkV0vydXHDDIU8ga2EpiAyYYWhIKYvQ5Xjd850u0AyrPIbkOEaxPZDHmR249GiAnrkNORVg45eQC6GzWVplxUGMUR37sseSzNrlavQzMRpUHdQ1jeJFhRqYlOTMQ1sTicKRHoxsD02Ase6ebYrNxUezPV1IhRwGxDgjslQdbWcoICqaGYMERvugV6kyVQl1l68rTm1okgKupWNIfnHZ4RmbdWEJGBaQ4CIXFDzckkQ1BqWuyqxN3N6kj6LoxcvV46tplSOTvyAQ4vQ0ZZVeCxCBvu6O4OkIGj6IXwLBiaouAtrSlFz795kew70VnFBnjEfMfiSep8xiXv83I36SZgoH8sv2BrtlOu8jTJ0ugcT08dORqnA2UzjWI4L5JjM3LQdgEQ80}   \end{equation} and   \begin{equation}   \sup_{t\in [0,\tau)}\sifpoierjsodfgupoefasdfgjsdfgjsdfgjsdjflxncvzxnvasdjfaopsruosihjsdgghajsdflahgfsif \vk(t)\sifpoierjsodfgupoefasdfgjsdfgjsdfgjsdjflxncvzxnvasdjfaopsruosihjsdgghajsdflahgfsif_{L^{6}}   \leq   M_k   \comma k\in\mathbb{N}_0   .   \llabel{sTyV3hNC72Q89Qb5bxur3w6rjyNl9kDte492PJf3TxTDd17VbnAPKptpusRQuHiBwvO39DUmCifhow4RUvFoejBw1JOq8mHPwCfIfh3uUxOVXLoz1E1d0V2KbeCLL9M3pMAkLuAcb6KUHnOXqdPGATF1Lrsh5tpx3OZIAD0H5nDGHlXBqTqgyAEA6pbVZcNjRE1BkH3omJFFjm9TJA7NUBmtg5ppAwIh190lJCmYeih6JWiCfyDWAByAbg61BS14QNzvZSQgqLjGvIW7Vz37vT4dZ92l3DdxftojupoeKITksc2uF2YnBJtMrrNsDJ0hPEK2h7KFiQmbEzr1TClt50d3LR9HDyUIetgwmyKv6NMDAzFDvMoxIYoMiKt9ZPAdaDYug53lgYerHdqgXX70KpLETmzD24CryGGRgzrNZVs78R7S0k0ji59HqHYiB0KvtZrqjtqBxRqbcLz69t0O7ceIlLvMkDc725QVgUMo6uO72ftjKNmef1dr5yiZKljQgHKgliZZQtkheh8wZ7ZRoSGVax9RKHQwgBEvzRwdGNcEPCGYhatb2jZ6XmUdNpYdg7l7sBfrhGmlueYkHA4vIesnrWIqxLSPzSnDFx7RPCfsGeziz44BxQVkm9edBAEhMdruUNIa03TUvnhkRmygOHTcTfiLkfUgYhruxfmyqLiFRozIcHkCC0kXQJo5C3jsfLMJe0vbW57SYQbbWW1ovoL2RRm7yIdmwpXqneEKrwBvjvor3pvUVDD8BBsuq9OobZEQ81}   \end{equation} \par Regarding the finiteness of $\tau$, we have the following lemma. \par \cole \begin{Lemma} \label{L10} Suppose that $\{\vk_0\}_{k\in\NNp}$ is a sequence of initial data of~\eqref{EQ34} satisfying the assertions of Lemma~\ref{L05}. For every $p_0\in(0,1)$,  if $\epsilon_0$ is sufficiently small relative to $\bar{\epsilon}$ and $M_k$ sufficiently large relative to $\MM_k$ (see their definitions in~\eqref{EQ17}, \eqref{EQ31}, \eqref{EQ35}, and \eqref{EQ36}), then $\PP(\tau<\infty)\leq p_0$. \end{Lemma} \colb \par \begin{proof}[Proof of Lemma~\ref{L10}] First, assume that $\bar\epsilon> 2\epsilon_0$ and $M_k > \MM_k$, so that $\sifpoierjsodfgupoefasdfgjsdfgjsdfgjsdjflxncvzxnvasdjfaopsruosihjsdgghajsdflahgfsif  v_{0}^{(k)}\sifpoierjsodfgupoefasdfgjsdfgjsdfgjsdjflxncvzxnvasdjfaopsruosihjsdgghajsdflahgfsif_{\LLLit}< \bar\epsilon/2^{k}$ and $\sifpoierjsodfgupoefasdfgjsdfgjsdfgjsdjflxncvzxnvasdjfaopsruosihjsdgghajsdflahgfsif   v_{0}^{(k)}\sifpoierjsodfgupoefasdfgjsdfgjsdfgjsdjflxncvzxnvasdjfaopsruosihjsdgghajsdflahgfsif_{\LLLis} < M_k$ for all $k\in\mathbb{N}_0$. Using Lemma~\ref{L07}, we infer that \begin{align} \begin{split} \sadklfjsdfgsdfgsdfgsdfgdsfgsdfgadfasdf\biggl[\sup_{0\leq t< \infty}\sifpoierjsodfgupoefasdfgjsdfgjsdfgjsdjflxncvzxnvasdjfaopsruosihjsdgghajsdflahgfsif\vk(t,\cdot)\sifpoierjsodfgupoefasdfgjsdfgjsdfgjsdjflxncvzxnvasdjfaopsruosihjsdgghajsdflahgfsif_p^p\biggr] \leq  C\sadklfjsdfgsdfgsdfgsdfgdsfgsdfgadfasdf[\sifpoierjsodfgupoefasdfgjsdfgjsdfgjsdjflxncvzxnvasdjfaopsruosihjsdgghajsdflahgfsif \vk_0\sifpoierjsodfgupoefasdfgjsdfgjsdfgjsdjflxncvzxnvasdjfaopsruosihjsdgghajsdflahgfsif_p^p] \comma p=3,6, \end{split} \llabel{Wmwy1ql338ivREeCOsJK7Tjax80TdcrxrU0WkrjidNQHLObsVEMR5ih9SmKjrgZCMDUD28M5TL5qTCSbIe9MYP3ARqOLxQH60dvw1XCvSnpTBcaBAMObspru0BDxjVUWWfMp10GNLtfRogR2TlC4KKhCbYsRCGh7FNpjgdSCHUzEKRsAcK2iZG9wFpgwCPUi7jh7fw5OfIJ6iw8EU7ECMBzzhKBEAOXFSz3qsSMv5pZldBJ6pQZcQ8ISDMzGh2GQDm2Rco69Z0zEEPSymYZK9vztjRhCxx980eqRB3C9DimZOyq936zi2X9pjxaf4T8T1yNdIFsdPMiU7gXdeHctkQGkJfjXs4uoBxZA2vY4sWxhqgNJoUNvlSsDq3xWsOjrc1Voi6eienlnGHf8UYCUexIgLa3wPm7A37TKTNSTsv5s2pLuyK6jXKqclcpnPEmofpXpXzVtuwpFO8Usppuucs3ZyVMGgG4zKXz1STngNjSRReL07cWjObWU2d1wq5k2UTAHguogqzVj9zR5cltRIRgGaJgQ2RRlMvdikH6UQiHfOJoTsCM5W2y6iNn1xpobqmpgwjyxoRsdmPboozEAyUoTmw1tQmqP9EwyL8HaGxAc0UhK2znvVWcv4xF5xAYXXjpJmE5fV2hUn6jt2Lt2npfeIdgLKL0xSUDMR51Rzg6r6qUCMcV24DAShdpmrmvene8eSZaP5BbbPPSB6eOsVJOV4dmUiZ2mEwiwTNt1JDDwWjT4mfF4jZC9hko7giefw6EQ82} \end{align} for some $C$ that is independent of~$k$, if $\bar \epsilon$ and $\epsilon_{\sigma}$ (see~\eqref{EQ30}) are sufficiently small. By Markov's Inequality,   \begin{align}   \begin{split}   \PP\biggl(\sup_{0\leq t< \infty}\sifpoierjsodfgupoefasdfgjsdfgjsdfgjsdjflxncvzxnvasdjfaopsruosihjsdgghajsdflahgfsif\vk(t,\cdot)\sifpoierjsodfgupoefasdfgjsdfgjsdfgjsdjflxncvzxnvasdjfaopsruosihjsdgghajsdflahgfsif_3\geq \frac{\bar{\epsilon}}{2^{k}}\biggr)   \leq    \frac{C\sadklfjsdfgsdfgsdfgsdfgdsfgsdfgadfasdf[\sifpoierjsodfgupoefasdfgjsdfgjsdfgjsdjflxncvzxnvasdjfaopsruosihjsdgghajsdflahgfsif \vk_0\sifpoierjsodfgupoefasdfgjsdfgjsdfgjsdjflxncvzxnvasdjfaopsruosihjsdgghajsdflahgfsif_3^3]}{\bar{\epsilon}^3/2^{3k}}   \leq    \frac{C\epsilon_0^3/4^{3k}}{\bar{\epsilon}^3/2^{3k}}   \leq    \frac{C}{2^{k}}   \left(    \frac{\epsilon_0}{\bar\epsilon}   \right)^{3}   \leq   \frac{p_0}{2^{k+2}}   \comma 
  k\in\mathbb{N}_0   ,   \end{split}    \llabel{w1vmlCzxs6ijEv9M3OtrOxW2pkMWi7zFKYpA0DcFZpnNS1Lk1rPOTNRZNpe0OWgJK5agU63LxqcHw0iSBcYqNRaD38nPvGHselXU3x8bUZuje0xiRm3iRAbwtJBlaV4M2PJMq7jb3KYDNMbl3CVKrZZDRnV4A65HtZiRodJP0T4NFm1G3ukOBMuK52ljg5v9tBsFjgCRa68HjYa694ZjFJ6CVArUC2OVqxQcnOtix4DXkXaU86p3k1dQqfJUcTYg9MU1RGZFJPXYBCZpgFVfYWG2vVeUqL8QD5IkcRtJuuS2jV2j9re9nJq1j1LPnGXB77qxq0Y2azEuZ4u441ZDhP2c6ltZfe7sMqQEKxXmdTmc3NG1cbV4nmp3bGwOzqXPVe7sEDvoeN9DmlBeZGcrMyDBgyvyBb0YaeHbb2PlECWXTTua3s6XI7ft2Ah6eEUiUur3aiIjiIQx2c3fDbvdyB7wo4oR3i1lrvQS8HjwHsbZzQ47M5Upnd0q5kKW6ZWnGPJtIEdnmGALnep3huZwDANUo8G1vcBzxXVdpxFBpUJlItAEBDsjvloldGbWbqs2ZYgwIZ5URD3kPbYDmbFN9fu026aB6pz5DgpDwrEGQ0Fqd0SE0sQdQAr8OTgO4fpavWhGZygRhreK5l6IFBWjy60VkeHEnXbH5mSL45RJe4oS9aMVTcaRLS768nDfgaCeH9JEvoHAo0guEz8XScEWRDNN53wJbB2uFhfL6zuTGzMbitUjWdxZ6a2TIfhtIyuR5IEQ83}   \end{align} provided $\epsilon_0$ is small enough relative to~$\bar{\epsilon}$. In addition, we can choose $M_k$ large relatively to $\MM_k$ so that  \begin{align} \begin{split} \PP\biggl(\sup_{0\leq t< \infty}\sifpoierjsodfgupoefasdfgjsdfgjsdfgjsdjflxncvzxnvasdjfaopsruosihjsdgghajsdflahgfsif\vk(t,\cdot)\sifpoierjsodfgupoefasdfgjsdfgjsdfgjsdjflxncvzxnvasdjfaopsruosihjsdgghajsdflahgfsif_6\geq M_k\biggr) \leq  \frac{C\sadklfjsdfgsdfgsdfgsdfgdsfgsdfgadfasdf[\sifpoierjsodfgupoefasdfgjsdfgjsdfgjsdjflxncvzxnvasdjfaopsruosihjsdgghajsdflahgfsif \vk_0\sifpoierjsodfgupoefasdfgjsdfgjsdfgjsdjflxncvzxnvasdjfaopsruosihjsdgghajsdflahgfsif_6^6]}{M_k^6} \leq  \frac{p_0}{2^{k+2}} \comma  k\in\mathbb{N}_0 . \end{split}\llabel{gpZCA5HZMZT9SvLIqSwDe8gBxlEdFkyj2QyuYzwlwpxYas2Xz8vCmy2027Jm0LHDMi8X5I255a2UfDR4mcoTUjWxZ3ZJ03alP15KPWffdCPQvOfakMNcpuVDt4cjR1oY7dqRGJvGT0ikreblMGbaNXGb8OEx6aIsztj7eeTt9OKLBFuMQbPLyqTdkjPdF7DNGeSMBOEYZ0aHFSUGw6GNIl3rCv2gKZvonEoA4IixK1D6RmzZgMEtWfXdAB9HMHCjO0woSnFR5ouCA2ug9Qr9w68xDKat9r3pnFjuBuaLmuZkyxCOZpfAbl4tXoggaAaws7Te4IOHrgzIvJq03hLQ585CRSLUZAxXz6RjOpeSx7BRT7txkA3dYoipdOpjMYASRDUXR02w2BAyxHnxBn5huR3MgQzB8mVlx3oBYOTTUMygvzKk6xrgiJZBGwQeoe3ToBsNhgzl64FVnOdf74cinPCeqU7hVu03CSpkpK0SZ9faBxRdRWaIz7Qlt63qkh0bbrEZ2FzUfvSD29l3cwm9uA1zYlTAONVdcOOfYcDS6Hgl4QNmCuPdgFEpuc4nz4vO8N6HC52kSn7ZMrzuz2zK9tk6DAjaUW6vghspP3MDTE9TD8WbIreHTM8VaqwE1kMM3vDcFVic5wDMXThn1NdqgIJZEwPXjazE6sa9UMIlsAZS4EF2vehcxfoAuLVxunTvRTKawGcHoviaXbtCuyMVT7DQdbuGB6WemA8AO9XEFkO8f6ArkAbi397pw9UoR7FCtPEQ84} \end{align} Then, \begin{equation} \mathbb{P}(\tau_k<\infty) \leq \frac{p_0}{2^{k+2}} ~\text{ and }~ \mathbb{P}(\sifpoierjsodfgupoefasdfgjsdfgjsdfgjsdjflxncvzxnvasdjfaopsruosihjsdjghajsdflahgfsif_k<\infty) \leq \frac{p_0}{2^{k+2}} \comma k\in \mathbb{N}_0.    \llabel{GytKq01TEmnbbNVQOuWDUJdgzsaNJLPMAg1EYKs7vNI45FZxQzabloG5bT4UMeFFsgh4PzviDrNUmfuDEumNFpNWbgBdstqQkxVgebWnkXxwEx5mnT5ClbHNY3dv7bvSXqVAKwZq9OZCVhcnFplqS46GIDIS6gByGMKMuec2QsaXEt6nqtuzWxEMcmWcY7Yo1TKsEhXIlBg4dU6fdJT1HYhGyNpqvl2LQazFBoUQpdn8OiN1cGW8lFI3sKQR4EpX4Sfs7d6xxt8hoY8FtaiiwWROI3r1shgHgY6CFK2738bWZCkOb3oNdMQ3Ex3PGeYrE1574bD7yGiOpfQeHH5HO7RQtOJ9ZZpbtYpaHlL9C0pbmVtRSVOzKACwQMj4Q4o4eJArRdciatphUeJADtQdkWk6m70KApmTV15qYnboDBktnk4ukg8SPIjOm3loj895BhImn8bXAUAmLWRxp2cOFiWNjXPoqOHNnZN9WqkEUNHh13JHCpYI5qYrOwvGfgUxhi4UTdN3VMSiRTvvH7xsEtlI5jkmHa35re1gHeq4LViMpiUNoWhTp9zvYhGulekd0rsCBZfjeHQGyhWiEQX1Jgrj0L9iprpDqFNiZsFSBY0Jz81ojF3V4dZnehUR0JXtNXswcRq7BspAp8od99NsP3ssZex8aSlCgDdurUaYHI1SdqCnaMKPrTWQCROzhuZqHxPEGLM9BBfabHUiSq4Qy1LPhQWaIWYo9IakJbRNCxHhp6eiWv7p7JUJgSBNgUSVZsEQ63} \end{equation}  Thus the stopping times $\tau^{k}$, defined in \eqref{EQ85}, satisfy   \begin{align}   \begin{split}    \PP(\tau^k<\infty)    &\leq    \sum_{j=0}^{k}    \mathbb{P}    \left(    \sup_{0\leq t< \infty}\sifpoierjsodfgupoefasdfgjsdfgjsdfgjsdjflxncvzxnvasdjfaopsruosihjsdgghajsdflahgfsif\vj(t,\cdot)\sifpoierjsodfgupoefasdfgjsdfgjsdfgjsdjflxncvzxnvasdjfaopsruosihjsdgghajsdflahgfsif_6\geq M_j    \right)    +    \sum_{j=0}^{k}    \mathbb{P}    \left(    \sup_{0\leq t< \infty}\sifpoierjsodfgupoefasdfgjsdfgjsdfgjsdjflxncvzxnvasdjfaopsruosihjsdgghajsdflahgfsif\vj(t,\cdot)\sifpoierjsodfgupoefasdfgjsdfgjsdfgjsdjflxncvzxnvasdjfaopsruosihjsdgghajsdflahgfsif_3\geq \frac{\bar{\epsilon}}{2^{j}}    \right)    \leq p_0     \comma k\in\mathbb{N}_0    ,   \end{split}    \llabel{HUJ3oeV1uWw9Aeeixbp8Axt5ZgQPREAQsA0iwxlNK0td3JrPygKrvUR0YXdIlWwJdEao9WMuReYfQUqUoebEgHTLIcL62xVJZhnGxEHyBRff4wT9LuePpwjrpRogWJ0Iv6sVWMIDiLzotM3ZrY4UKtDEJo2JotOUhTMblm3oNAr4FUYddaug68TX9BgoI0N0C5ySekTx6JxUvudb3KmMocy1WtAyC1uaC9rD3l579WJJ59LQaO7hxpDHsE1LtilFTP6jhASQePrSiosf5TsUVjOPX5qZZo0m0dzwu6ctAw6a3boSxFqDWJauvpSgcgUo3aPoBwWE2BJ9GAlK3c0ujcsQFNb4yCzpRUiQeZEDLIkzqkkEGLKC45BngCXim0yQL5UY0jWzkX2gvw2ui3S30GXpRRuRIOYERuRMD0bBNOb6yvIpeNpLKyV1DqnTur1ckV4yM5Z0dh7u9KmsrPGrxRkRPb09I1YXI9IMf6kPUDYOz9P535CU39tLxBPHIq5gL4wd3S87hiGY0diFePnJBIRNwR60scpaOS5GkEpCyWMEXyMH0v86FRcr8aQYquO5ZZda1IfuGzD9TN1PQLNYFtxB6NX3yOIAULFEswa9VuVUzvoN7i8Au67gdRnyN3ntXN75CDXmhDVWSDzclZWQ3D7FniEs4KRwvAMdkgXKS8feURSISDF3KoWBCeW8fkYNgH8ESnonm8fw8AhCulfSYR4LvuF5xYNuBFM0s8Q9fVrVjS6qkEdI9B7AaABnFEHyQ4EQ86}   \end{align} completing the proof.  \colb \end{proof} \par \begin{Remark}\label{R02} Without imposing the smallness of $\epsilon_\sigma$ in the assumption~\eqref{EQ30}, we would obtain \eqref{EQ62} instead of \eqref{EQ55} in Lemma~\ref{L07}. Correspondingly, we would arrive at $\PP(\tau<T)\leq p_0$ in Lemma \ref{L10}. \end{Remark} \par \cole \begin{Lemma}   \label{L08} Let $k\in\mathbb{N}_0$. Assume that $\vk_0$ is an initial datum of~\eqref{EQ34} satisfying the assertions in Lemma~\ref{L05}. With $\tau$ as in \eqref{EQ70}, we have   \begin{equation}    \PP(\tau>0)    = 1    ,    \label{EQ79}   \end{equation} provided that $0<\bar\epsilon\ll 1$, $\epsilon_0$ is sufficiently small relative to $\bar{\epsilon}$, and $M_k$ sufficiently large relative to $\MM_k$; see their definitions in~\eqref{EQ17}, \eqref{EQ31}, \eqref{EQ35}, and \eqref{EQ36}. \end{Lemma} \colb \begin{proof}[Proof of Lemma~\ref{L08}] \par We apply Lemma~\ref{L02} componentwise to the equations~\eqref{EQ34} on $[0, \delta]$, where $\delta\in(0,1]$ is small, obtaining 	\begin{align} 	\begin{split} 	&\sadklfjsdfgsdfgsdfgsdfgdsfgsdfgadfasdf\biggl[\sup_{0\leq t\leq \delta}\sifpoierjsodfgupoefasdfgjsdfgjsdfgjsdjflxncvzxnvasdjfaopsruosihjsdgghajsdflahgfsif\vk_j(t,\cdot)\sifpoierjsodfgupoefasdfgjsdfgjsdfgjsdjflxncvzxnvasdjfaopsruosihjsdgghajsdflahgfsif_p^p-\sifpoierjsodfgupoefasdfgjsdfgjsdfgjsdjflxncvzxnvasdjfaopsruosihjsdgghajsdflahgfsif\vk_{0,j}\sifpoierjsodfgupoefasdfgjsdfgjsdfgjsdjflxncvzxnvasdjfaopsruosihjsdgghajsdflahgfsif_p^p 	+\sifpoierjsodfgupoefasdfgjsdfgjsdfgjsdjflxncvzxnvasdjfaopsruosihjsdfghajsdflahgfsif_0^{\delta}  \sifpoierjsodfgupoefasdfgjsdfgjsdfgjsdjflxncvzxnvasdjfaopsruosihjsdfghajsdflahgfsif_{\TT^3} | \nabla (|\vk_j(t,x)|^{p/2})|^2 \,dx dt\biggr] 	\\&\indeq 	\leq C 	\sadklfjsdfgsdfgsdfgsdfgdsfgsdfgadfasdf\biggl[ 	\bar\epsilon\sifpoierjsodfgupoefasdfgjsdfgjsdfgjsdjflxncvzxnvasdjfaopsruosihjsdfghajsdflahgfsif_0^{\delta}  	\sifpoierjsodfgupoefasdfgjsdfgjsdfgjsdjflxncvzxnvasdjfaopsruosihjsdgghajsdflahgfsif \vk\sifpoierjsodfgupoefasdfgjsdfgjsdfgjsdjflxncvzxnvasdjfaopsruosihjsdgghajsdflahgfsif_{3p}^{p} 	\,dt 	+ \epsilon_{\sigma}^2 	\sifpoierjsodfgupoefasdfgjsdfgjsdfgjsdjflxncvzxnvasdjfaopsruosihjsdfghajsdflahgfsif_0^{\delta}\sifpoierjsodfgupoefasdfgjsdfgjsdfgjsdjflxncvzxnvasdjfaopsruosihjsdlghajsdflahgfsif_k^2 \sifpoierjsodfgupoefasdfgjsdfgjsdfgjsdjflxncvzxnvasdjfaopsruosihjsdkghajsdflahgfsif_k^2  	\sifpoierjsodfgupoefasdfgjsdfgjsdfgjsdjflxncvzxnvasdjfaopsruosihjsdgghajsdflahgfsif \vk\sifpoierjsodfgupoefasdfgjsdfgjsdfgjsdjflxncvzxnvasdjfaopsruosihjsdgghajsdflahgfsif_{p}^{p}\,dt 	\biggr] 	\commaone j=1,2,3 	, 	\end{split} 	\llabel{FE3p2zNYbeCWWsjICZdMiQxwExryPWHPEdyQb0vtY4GOLLAXSYFDdvhbwojup5CPrRfoTWAfFdeQy4QwgGS8Ss4K6UrdkhKQBOIMpMAPsjTHPwz4WqITmhrrrrh2ZJ2AmbTsaJHj3QxIxTN0HyXIgrXl7sQdnBWrzolZxCCwSlP6b52rqIT8yOeZKUW8iASz8WKiu3bjhg1zCxce2fpagwSyoiD3YxMRrbrprB8oFpJMtItcgQCJD8qQPX3tFyOyK9kqYsp1He43FFjfTYAoH9TrsZulJLfz5PN9MwOwlhKXLn4BWKHpc5r7rV3Ivf2qAk0ckyWiPt2Oo9BTfM5cXjaltIM7PPtTpRX1OysUMBp8XCuzgAd0NKAvdil19GwZVPrGLoGONxeU03Sm0BA1cxyfKGITFLNfgWPMpfRmO63LFMDUmz7qGhj190n34Oc9gieJVzMNOgAhfGTqBpw53q4eyb5J8ZeUNT8R38gKs708pFfj1ELgKiz0EACdAlPNN4okghcRHpmUuoeX4JJ2jtEXNXPRfb9FbN29ZtYUeX5EK9Lv22fvW5Q0VI0sq1IR1a5aojpZWyEKcOLaRUd2wxsHXws4XSrIVgWKywytyA00FoMN8est52u9fFKESCiP7WBvbjOhWuZIU5qDYQZOYVc1m2uyyWrATV9JIVCcILzr0XfXvLRjHkQ7dvQ5qQkKbUddVj2NgD9SbOe2NDztnIjaLQEEnLY4K0Do66oKZnE5tg9p8WgFJrymnIR4JuqanKEQ64} 	\end{align} for $p\in\{3,6\}$. Assuming that $\bar{\epsilon}$ is sufficiently small, this estimate implies  	\begin{align} 	\begin{split} 	&\sum_j \sadklfjsdfgsdfgsdfgsdfgdsfgsdfgadfasdf\biggl[ 	\sup_{0\leq t\leq \delta} 	\sifpoierjsodfgupoefasdfgjsdfgjsdfgjsdjflxncvzxnvasdjfaopsruosihjsdgghajsdflahgfsif\vk_j(t)\sifpoierjsodfgupoefasdfgjsdfgjsdfgjsdjflxncvzxnvasdjfaopsruosihjsdgghajsdflahgfsif_p^p-\sifpoierjsodfgupoefasdfgjsdfgjsdfgjsdjflxncvzxnvasdjfaopsruosihjsdgghajsdflahgfsif\vk_{0,j}\sifpoierjsodfgupoefasdfgjsdfgjsdfgjsdjflxncvzxnvasdjfaopsruosihjsdgghajsdflahgfsif_p^p 	\biggr] 	\leq C 	\epsilon_{\sigma}^2 	\sadklfjsdfgsdfgsdfgsdfgdsfgsdfgadfasdf\biggl[ 	\sifpoierjsodfgupoefasdfgjsdfgjsdfgjsdjflxncvzxnvasdjfaopsruosihjsdfghajsdflahgfsif_0^{\delta}\sifpoierjsodfgupoefasdfgjsdfgjsdfgjsdjflxncvzxnvasdjfaopsruosihjsdlghajsdflahgfsif_k^2 \sifpoierjsodfgupoefasdfgjsdfgjsdfgjsdjflxncvzxnvasdjfaopsruosihjsdkghajsdflahgfsif_k^2  	\sifpoierjsodfgupoefasdfgjsdfgjsdfgjsdjflxncvzxnvasdjfaopsruosihjsdgghajsdflahgfsif \vk\sifpoierjsodfgupoefasdfgjsdfgjsdfgjsdjflxncvzxnvasdjfaopsruosihjsdgghajsdflahgfsif_{p}^{p}\,dt 	\biggr] 	\comma p\in\{3,6\}. 	\end{split} 	\llabel{B0Tu9gTTdhmFWJr7SkcmRgSVZBHXMKPdve2TKbYwe1w6LjLt3fqNqPNrkL9xZ312OlYPJzYHIR9FOwtOda3azKYIZmfOpiDP2UY1xdXMxnx0FHt0Z7KK5meIiZqgEIAVqkXuCIau0ge7syt7wPL1Somn8trVNAP5vOeApza4OmcKk2KSIDKbjSH9Rt0UKyyFShxSWRzxSsuMN2vFJgdTwYNCC7fYPYG52l9DuZ30jM3nxAAjYnsbz3ZREXuoEuVh6vUiKQXl9xplK4pv4r1SOF3gu9JSquV3aSxxP4r2sqmh8Vnc5lAnLI692OJ3I9hfl4tXzaJe6GPjnSEJW9IKlRk7tDlXhKdgKqL5vJLR7DGF4USAIttBcO5DLELRkZR6VJCCuPGSsCMuMUqIge2P0YmcB39mmPK7m3iNo3LPlKT6DxnGuwVqMZrP9zdW0MIXlRsCvy2x5hWn7UxgtJMGkQSm5W7JEQfvmB2Gi6cXhyQt1WjTdz4Samd0QLh4lfRnGntTOT6YaQsAyNoew52Km6kidUykKV0tntPTpBygKcOhcx77EkXSLocdpouOwaCQ3fNM4vtwIJvNz7dwFNviTOv3Medg8iLhaM3kdcjDT30paBswxFVIaXyjw0jOSbROdY3e3Phl4qm1wtaM08QrmGBISKqyoygoFISt0tQp0pZ1QQJknImQg436gpzdYuwl5msooOh9l4Z25dXGcSBaavBCelg09qWH49vDRFPE6SRgMSrOEiFDqSmv7MqC8n0B4OEQ65} 	\end{align} 	Without loss of generality, we may assume that $0<\epsilon_0\ll \bar\epsilon$, so that $\sifpoierjsodfgupoefasdfgjsdfgjsdfgjsdjflxncvzxnvasdjfaopsruosihjsdgghajsdflahgfsif v_{0}^{(k)}\sifpoierjsodfgupoefasdfgjsdfgjsdfgjsdjflxncvzxnvasdjfaopsruosihjsdgghajsdflahgfsif_{\LLLit}\leq \bar{\epsilon}/4^{k+1}$. Then, using Markov's inequality, 	\begin{align} 	\begin{split} 	& 	\PP\biggl( 	\sup_{t\in [0, \delta]} \sifpoierjsodfgupoefasdfgjsdfgjsdfgjsdjflxncvzxnvasdjfaopsruosihjsdgghajsdflahgfsif\vk (t)\sifpoierjsodfgupoefasdfgjsdfgjsdfgjsdjflxncvzxnvasdjfaopsruosihjsdgghajsdflahgfsif_3 	\geq   	\frac{\overline\epsilon}{2^{k}} 	\biggr) 	\leq \sum_j 	\PP\biggl( 	\sup_{t\in [0, \delta]} \sifpoierjsodfgupoefasdfgjsdfgjsdfgjsdjflxncvzxnvasdjfaopsruosihjsdgghajsdflahgfsif\vk_j \sifpoierjsodfgupoefasdfgjsdfgjsdfgjsdjflxncvzxnvasdjfaopsruosihjsdgghajsdflahgfsif_3 	\geq  	\frac{\overline\epsilon}{3\cdot 2^{k}} 	\biggr) 		\\&\indeq 	\leq  	\sum_j \PP\biggl( 	\sup_{t\in [0, \delta]} \sifpoierjsodfgupoefasdfgjsdfgjsdfgjsdjflxncvzxnvasdjfaopsruosihjsdgghajsdflahgfsif\vk_j \sifpoierjsodfgupoefasdfgjsdfgjsdfgjsdjflxncvzxnvasdjfaopsruosihjsdgghajsdflahgfsif_3^3-\sifpoierjsodfgupoefasdfgjsdfgjsdfgjsdjflxncvzxnvasdjfaopsruosihjsdgghajsdflahgfsif\vk_{0,j} \sifpoierjsodfgupoefasdfgjsdfgjsdfgjsdjflxncvzxnvasdjfaopsruosihjsdgghajsdflahgfsif_3^3 	\geq \left( 	\frac{\overline\epsilon}{3\cdot 2^{k}} 	\right)^3 	- 	\left( 	\frac{\overline\epsilon}{4^{k+1}} 	\right)^3
	\biggr) 		\\&\indeq 	\leq  \sum_j \PP\biggl( \sup_{t\in [0, \delta]} \sifpoierjsodfgupoefasdfgjsdfgjsdfgjsdjflxncvzxnvasdjfaopsruosihjsdgghajsdflahgfsif\vk_j \sifpoierjsodfgupoefasdfgjsdfgjsdfgjsdjflxncvzxnvasdjfaopsruosihjsdgghajsdflahgfsif_3^3-\sifpoierjsodfgupoefasdfgjsdfgjsdfgjsdjflxncvzxnvasdjfaopsruosihjsdgghajsdflahgfsif\vk_{0,j} \sifpoierjsodfgupoefasdfgjsdfgjsdfgjsdjflxncvzxnvasdjfaopsruosihjsdgghajsdflahgfsif_3^3 \geq \frac{1}{2^5}\left( \frac{\overline\epsilon}{2^{k}} \right)^3 \biggr) 	\\&\indeq 	\leq  	C\left( 	\frac{\overline\epsilon}{2^{k}} 	\right)^{-3} 	\epsilon_{\sigma}^2 	\sadklfjsdfgsdfgsdfgsdfgdsfgsdfgadfasdf\biggl[ 	\sifpoierjsodfgupoefasdfgjsdfgjsdfgjsdjflxncvzxnvasdjfaopsruosihjsdfghajsdflahgfsif_0^{\delta}\sifpoierjsodfgupoefasdfgjsdfgjsdfgjsdjflxncvzxnvasdjfaopsruosihjsdlghajsdflahgfsif_k^2 \sifpoierjsodfgupoefasdfgjsdfgjsdfgjsdjflxncvzxnvasdjfaopsruosihjsdkghajsdflahgfsif_k^2  	\sifpoierjsodfgupoefasdfgjsdfgjsdfgjsdjflxncvzxnvasdjfaopsruosihjsdgghajsdflahgfsif \vk\sifpoierjsodfgupoefasdfgjsdfgjsdfgjsdjflxncvzxnvasdjfaopsruosihjsdgghajsdflahgfsif_{3}^{3}\,dt 	\biggr] 	\leq 	C\delta \left( 	\frac{\overline\epsilon}{2^{k}} 	\right)^{-3} 	\left( 	\frac{\overline\epsilon}{4^{k+1}} 	\right)^{3} 	\leq \frac{C \delta}{2^{3k}} 	. 	\end{split} 	\llabel{ujjipuVia053hgXhQqEIzH7cbOABdHwOdTEXvoykC4wrajhzqnquDojXl0Pz7UtybiOOsEINnWEXehc70MvZ7n6V9qYzryOYuiqBEmXaktOJI8ZAdNdvleWiAh0gwDT6fWH6x8uSAjGkKuV6bwiUfcfOWjyzWdRn7KpFuCAHIogHdU3WkbDq1yV8oMfJskTSJAFV4ehtkSEpCWQaSgnZkk3aeesYU7TC7DjdycEy6YhXSt7PPC75VyXVX0033ywEoy37OVWeFB6f5A4zI5U5FYFTvRFmZbF16anRnJcMwFJuXCgDVU7FkanwUZqfhivmvYiUAZaiJLeDmZbMKcvxRNJpwxkgIOqH6pTR821OCO97sijRlLagXkCrxnH5nGq2b8TD1QLjwMUG7wLB8hV0tcfNtxrtAKTbTV70Cus2Z9zWbi3fBhUBueLctDXuk88jYpxJzuARu5GQSMoTAOBuuByUhESR0IVrVGFcH0rtSLRCVouoU5QbOUKP8WaFjZ0RLHqyX8xXBLOwUtkiEwWG9fmn7Ifh4MwTpuLWjOVCmpvGnoyHPV3Q2xzY5iDGuMIyY2zHhLTMRUHz9PrIp7eqfjbQMhsqm7f21zapbnCNGWxjGJAHfCykEOXMfRkQMSOnRmHDslDBlTsX1OYIh6VnhtORVgXFclHw0MsU8taIoG1w40uBD1sOoHWK3IngjGMdIjSUJarrbi4PVOmzDgPpUKxjKv1BOITyQEs2wltfw1QmqU5NgZs044dVr2FikdG6eqEQ68} 	\end{align} Observe the geometric convergence in $k$ of the right-hand side along 	with an arbitrarily small expression~$\delta$. 	Using a similar derivation for the $L^{6}$ norm rather than $L^{3}$ norm, 	we get 	\begin{align} 	\begin{split} 	&\PP\biggl( 	\sup_{t\in [0, \delta]} \sifpoierjsodfgupoefasdfgjsdfgjsdfgjsdjflxncvzxnvasdjfaopsruosihjsdgghajsdflahgfsif\vk \sifpoierjsodfgupoefasdfgjsdfgjsdfgjsdjflxncvzxnvasdjfaopsruosihjsdgghajsdflahgfsif_6^6 	\geq  M_{k}^6 	\biggr) 	\leq 	\frac{C \delta}{2^{3k}} 	, 	\end{split} 	\llabel{DEgvajQFWQPmQeNscdxSAyhqL2LfTUWHJoEsYAre3LvakeYHXXZEbdJodq2jjufZFwkFwMLCiBAzKslRU0wObDZ08O1IxMMACdT6Cuj3WOgHbMfWXYKctRoIofd0EECQVbFHq3GVsjjJoPiGcrg9Oa2u5s6j2W5AphHhdSrLt2Ctx8DHzuZXSEjHgy397e86ALC5AbRC7FQ4wELlpDvHtoumXBqq92Lgd2nSRcRG0a4DbmyLRpMSH3BXeD4ty54XXqEMbCkHuuytF38almrJcZeHlP8cgwNXGTZXeLavohyQaSEh1B7TEz54dRZGzhPJZtYAFK4PoB5vc5JDphrxfvSFcW2hSRzpnyvlqRWwhzhZ7fcd766Gh9DEWg5OcwF3FcGMUe42nZF1yw047XIG22Jd1INGkQLrXcvIW54aU33tnhD9NqKOLtMKkt5UCl2CmLLGkDZs8pORclJ3BqG5hKneVXEUFbnERg2k1yzfokQoCsnOZVSZk0X7sUjP5RruAQiLb6j2y3XRi9uzWFZZQnfC3ZisnU24uqcESJ9ATwmRCmbjfWBbjvQbZ1GmaYOiFK2s7Z20eOVLynbcwRtelbgInPM0f6Bd17EoyU9WOzSvNJh6ruX82bA5FvGC0sjhJsR110KLyNnrwYeVxgGE2h6LZS6AuVvmZ22sBOAVLsDyC1Hx9AqVK0xjKEnBwZAEnjXaHBePhlGArtTYDqq5AI3n843Vc1ess4dN0ehvkiAGcytH6lpoRT3A2eYLTK0WLJEQ21} 	\end{align} 	provided $M_k$ is chosen sufficiently large. 	Then we have 	\begin{align} 	\begin{split} 	\PP(\tau^k<\delta) 	&\leq 	\sum_{j=0}^{k} 	\mathbb{P} 	\left( 	\sup_{0\leq t< \delta}\sifpoierjsodfgupoefasdfgjsdfgjsdfgjsdjflxncvzxnvasdjfaopsruosihjsdgghajsdflahgfsif\vj(t,\cdot)\sifpoierjsodfgupoefasdfgjsdfgjsdfgjsdjflxncvzxnvasdjfaopsruosihjsdgghajsdflahgfsif_6\geq M_j 	\right) 	+ 	\sum_{j=0}^{k} 	\mathbb{P} 	\left( 	\sup_{0\leq t< \delta}\sifpoierjsodfgupoefasdfgjsdfgjsdfgjsdjflxncvzxnvasdjfaopsruosihjsdgghajsdflahgfsif\vj(t,\cdot)\sifpoierjsodfgupoefasdfgjsdfgjsdfgjsdjflxncvzxnvasdjfaopsruosihjsdgghajsdflahgfsif_3\geq \frac{\bar{\epsilon}}{2^{j}} 	\right) 	\leq 	C 	\delta 	, 	\end{split} 	\llabel{Opp5mbSzoaMNxjCJvzeTs9fXWiIocB0eJ8ZLK7UuRshiZWs4gQ6t6tZSg2PblDsYoLBhVb2CfR4XpTOROqtoZGj6i5W89NZGbFuP2gH5BmGsIbWwrDvr9oomzXTFB1vDSZrPg95mCQf9OyHYQf47uQYJ5xoMCpH9HdH5W5LsA0SHBcWAWGo3BiXzdcJxBISVMLxHKySoLJcJshYQyybWSSGdwcdEZ6bu0qfmmj7WMsVRYqJD7WciZqArimTf6csurioA8Ow5xADChPG7slxY1gyegeB9OC39KFnGn7aB1yr4vo4hHwIC93Lq9KQvIgFnBoS04NZd5ay1rlErPpbvxpseYDufgnOiibGhYOCBwwv4pxm37IgUW9XY4zxXMVQDNmYwkHaG258v6MBSerLf0Fya6yCLCzGm3P8COVDg7Rpp3FaWdG0kIQYRIs0ieIu2tgMmiYoJOmYt3fBSHj29Sq8ab4N94WEhm8BC79jHL9KmhK9Zfr6AlpyfIhpG93BnOMt16adBKKu3oNcfIhWjMfkeGdnzssVHKcphPZlBgcewSd0u1PePYEZowysGuGBvc6ea5iWB3XgSzyHnOaDp8VRDVqmVy3099pj5gFkDpqzPQ96m5MyJSQ7bj2LyLRn1PsvRqAaN5QOIcJKVPgGovU3Q69LN4W4v8wbS0jsi6krY7uOj9j1aBSZ6uonf5hsIj6BQ8Nnb2fcKcWDzL7nMsD7NxvuY95USbCwaHEBadGIqHgoriu5esF7AbYAYtJodoqEQ109} 	\end{align} 	for all $k\in\mathbb{N}_0$, 	and \eqref{EQ79} follows. \end{proof} \par With the help of the previous two lemmas, we may now proceed to the proof of existence and uniqueness of solutions. \par \begin{proof}[Proof of Theorem~\ref{T01}] Without loss of generality, we may assume that $u_0\in L^\infty(\Omega, L^{3}(\TT^3))$, with the assumption~\eqref{EQ17}. Let $p_0$ be an arbitrarily small number in $(0,1)$. Applying the decomposition procedure in Lemma~\ref{L05}, we have the $L_x^3$-convergence of $\{\sum_{j=0}^{k}v_{0}^{(j)}\}_{k\in\NNz}$ \eqref{EQ27} to~$u_0$ almost surely. We assume that $0<\epsilon_0\ll \bar\epsilon\ll 1$ and $0<\MM_k\ll M_k$ are such that Lemma~\ref{L10} applies with $\PP(\tau<\infty)\leq p_0$, where $\tau$ is the limit of the stopping time sequence~\eqref{EQ85}; see~\eqref{EQ70}. By the definitions of the cutoff functions, for every $k\in \mathbb{N}_0$, the pair $(\vk, \tau)$ is a solution to~\eqref{EQ32}. Using    \begin{align}   \begin{split}    &    \sum_{j=0}^{k}     (     \vj\nabla\vj     + \ujm \nabla \vj     + \vj \nabla \ujm     )    =    \uk \nabla \uk   \end{split}    \llabel{q8VAyc8JApV0lJlYwyn6Sc6qTU7O1ek3r6y8zE0NvqFxOe0etCQfFQYU6wydAK5xsgtWyoHf4wkLJE7vsDtaSeyQYnecGAEE9Yiehg1bbi18Qic5rJyA6E1GNznKFeO8F23Ren7C5mlsRUt5moBQE0P0vMb5sNRJiRENqaZl1S0XQCLAWKOfnhsEcrp3BDnd8xRXBqgedxRV9ds7hxDYsXfuUOYUpPd6nkpnqU1BZylozBiHlOZqUwQHw6uyx1CevAKJpX8j3FDHKdkFGe1PLPwa3zuUlrnPllZsvR42FozOOJ2xv7SOMypK6I9dxqWRHMdk90PZn1CJiHjFiQOpyPnFWcnPNLzFCUTZbYKKM8QyLgFXtFqKzGFo8PnWkYxYg2TsGQZfmqqDqizHd93oVTk6QGDUljuHwS88DFPcV0PfY1PqwpsQSbMrVVNeZDe1F6mt7PgelfBt5ita8qXMizwKxCNoigAqPysitAbI7g4R8XE7tAKNexkg5a2urRig9qMqHSkERbIp9hrW5jzTcJlTgpMeLWFzzSVNgDtecg2YfSfNjVFP155DAWMbntnVm17177lYTObMWmVtwfPqJ0nF6baQxuSP13ZJubfmgRHsWKVnCRDcMkLRVXd3PEvZme7xKUmMURkbif3lAOuTemZqkO5fId0htO7xUpLztA9RO7gysAfqj751DjTsfImVIPQne9Fo9aRex4T4oypenSzEDID7KWchKsIrTolagXLh4f0vi4yl25MHPpbk9zYIfpEQ94}   \end{align} for all $k\in\mathbb{N}_0$, which can easily be proven by induction, we obtain that the pair $(\uk,\tau)$ solves the SNSE   \begin{align}    \begin{split}     &\partial_t u - \Delta u  + \mathcal{P}((u\cdot\nabla) u)     = \sigma(t,u) \dot{W}(t),     \\     &\nabla\cdot u = 0     \\     & u(0) = \uk_0 = v_{0}^{(0)} + \cdots + v_{0}^{(k)}     ,   \end{split}   \label{EQ95}   \end{align} (note that $u^{(-1)}=0$), which means that for any $T>0$, we have    \begin{align}    \begin{split}      (\uk_j( t \wedge \tau),\sifpoierjsodfgupoefasdfgjsdfgjsdfgjsdjflxncvzxnvasdjfaopsruosihjsdkghajsdflahgfsif)      &=      (u_{0,j}^{(k)},\sifpoierjsodfgupoefasdfgjsdfgjsdfgjsdjflxncvzxnvasdjfaopsruosihjsdkghajsdflahgfsif)    + \sifpoierjsodfgupoefasdfgjsdfgjsdfgjsdjflxncvzxnvasdjfaopsruosihjsdfghajsdflahgfsif_{0}^{t\wedge\tau}          (\uk_j (r), \Delta\sifpoierjsodfgupoefasdfgjsdfgjsdfgjsdjflxncvzxnvasdjfaopsruosihjsdkghajsdflahgfsif )      \,dr    \\&\indeq    - \sifpoierjsodfgupoefasdfgjsdfgjsdfgjsdjflxncvzxnvasdjfaopsruosihjsdfghajsdflahgfsif_{0}^{t\wedge\tau}       \bigl(\bigl(\mathcal{P}  (\uk_m(r) \uk(r))\bigr)_j ,\partial_{m} \sifpoierjsodfgupoefasdfgjsdfgjsdfgjsdjflxncvzxnvasdjfaopsruosihjsdkghajsdflahgfsif\bigr)      \,dr    \\&\indeq      +\sifpoierjsodfgupoefasdfgjsdfgjsdfgjsdjflxncvzxnvasdjfaopsruosihjsdfghajsdflahgfsif_0^{t\wedge\tau} \bigl(\sigma_j(r, \uk(r)),\sifpoierjsodfgupoefasdfgjsdfgjsdfgjsdjflxncvzxnvasdjfaopsruosihjsdkghajsdflahgfsif\bigr)\,dW(r)\quad      \PP\mbox{-a.s.}     \comma j=1,2,3     ,    \end{split}    \llabel{FVV19CRh1tkqoQ1bQGtzGL0f4xN6GCHP7RXOnRbpX2WlP56rNQlEyARF84JeprN20CZxiyjplnNJiE3CfUxUjGxVfClOs8VwvDylP654xADcvWYtPPlqyDzXN83tJgydFQhHueuqU6ZyVijYQjtdJf4J93VVpNwzagiEmsHTInsFJaiogOY7GnLhHqUyLxrAq8itXPsMtyXO32mVYU2GMV729mGz8vXx4rNFR84VgcS6Bmn15nKdREfkhXGuf2c9CrXBqsEVOhPCASyf9Js77D9kqMznbcTe2Uu2Jnq8DBs8rfU7rd0BqQ4zzXOxTg7CrZuthjYW4aLtrFcCjwGZs7MWMmSpHFHfLiB7kX34MWAu5LIj3gvVGAkt1HUEE8UBYWVMJzMsv5P5JZvQnVHI1ZMg6AfsQjxl2VB9XKgEZrlgD5GqiFivOJgx0uYgE71IUqgjKnLSelC6w6tFI8fKMiHsqHDSeNXK9vtSmVcR8qXEEOqdh6WiObbuBi1dL1TGyUlerX0GfVwJXZcO4PCie5n0gNVUs6zGEiK4Qu0eVOBbGIv31f7MmZWGUui0XsJeIKrc6e1Hdz1A1uAiEJ1cBUWYPsoTyxIE5BV9Eyujob6iZVLOiNN6OdNQmB1lIjdwiBRH8d2G2I62dVz4IwScfAfPtOMSH1H1XALjFIXixvFrN9eesLSkD92y7poOY5zfXT21phWyN6C3iGT3jHRxnLx4ShHFPSeag0pIBSVsPSRRgqoHFPB9xiM4Lzsaaon5qdEQ104}   \end{align} for all $\sifpoierjsodfgupoefasdfgjsdfgjsdfgjsdjflxncvzxnvasdjfaopsruosihjsdkghajsdflahgfsif\in C^{\infty}(\TT^3)$ and $t\in [0, T]$. Due to the definition~\eqref{EQ33} and the inequality \eqref{EQ71}, the sequence $\uk$ converges uniformly to some~$u$ in the space~$\LLLL{\infty}{\infty}{3}$.  Using the uniform convergence of $\uk$ to $u$, we obtain the trajectories of $u$ and   \begin{align} \begin{split} (\uu_j( t \wedge \tau),\sifpoierjsodfgupoefasdfgjsdfgjsdfgjsdjflxncvzxnvasdjfaopsruosihjsdkghajsdflahgfsif) &= (\uu_{j}(0),\sifpoierjsodfgupoefasdfgjsdfgjsdfgjsdjflxncvzxnvasdjfaopsruosihjsdkghajsdflahgfsif) + \sifpoierjsodfgupoefasdfgjsdfgjsdfgjsdjflxncvzxnvasdjfaopsruosihjsdfghajsdflahgfsif_{0}^{t\wedge\tau} (\uu_j (r), \Delta\sifpoierjsodfgupoefasdfgjsdfgjsdfgjsdjflxncvzxnvasdjfaopsruosihjsdkghajsdflahgfsif ) \,dr \\&\indeq - \sifpoierjsodfgupoefasdfgjsdfgjsdfgjsdjflxncvzxnvasdjfaopsruosihjsdfghajsdflahgfsif_{0}^{t\wedge\tau} \bigl(\bigl(\mathcal{P}  (\uu_m(r) \uu(r))\bigr)_j ,\partial_{m} \sifpoierjsodfgupoefasdfgjsdfgjsdfgjsdjflxncvzxnvasdjfaopsruosihjsdkghajsdflahgfsif\bigr) \,dr +\sifpoierjsodfgupoefasdfgjsdfgjsdfgjsdjflxncvzxnvasdjfaopsruosihjsdfghajsdflahgfsif_0^{t\wedge\tau} \bigl(\sigma_j(r, \uu(r)),\sifpoierjsodfgupoefasdfgjsdfgjsdfgjsdjflxncvzxnvasdjfaopsruosihjsdkghajsdflahgfsif\bigr)\,dW(r)\quad \PP\mbox{-a.s.} \comma j=1,2,3 , \end{split}    \llabel{O9eSIckw1TWpA0dPRxm0mJS4CO20oacZNRhyXd9vF8myOtVO1yy8IPXRi69P7nEiVT9tqwzd7Si8yQLJx6HizMwkKqPKVouNGwhrsqKiGALuz34Qk0yTDFsJvXIYjMYHS9XWlgEwN6gCwyVdnIxK4B05Xchlwsf7kVsFFU6TUiBVUraV7WrTLbySahdqdcZgGtKtv0Uqp86jB630pHaqGwDnGpg0ZprB4fFDmtx2J23Ed37CNW5LkvQFWEouXhTIqdUmrBx7BGPMED17hd24zjCEavON1DJoxfmZadkTVq0VaQaDqE0vOGEazPBEKbULHqxEUgdPZyHW7xH2QWZqjxAaP0GszJTMSWwxJ2sZwiYA2Hu5wNNpJNQr1x3edvP8QVEszZCf2893tzSk3sLWLvXG2Ro4lH9AvQb5Ty4O1FTOAYScPcPIhJBPFPuHDAtdSw9x8vu9hrWf4XOypGsNQvxf9kolHSNEkMEGaZMPg1AwXYaRSUot7Fmp3SE4ZMRFbqq9KCjfC9LHi1X5nEhUw8q2Y6HLATRc4TNlNWkstGQ2RiLjSdaizOb6P848n3HLWHyAptpfQyXEcd7x0zqVNmDKzl5NtEVG2EvbbivLHBtnas3T4ddSM0QJyx3fU3CuYD2A5yhtcshT2ktVFKXXUDHoEs9pq3V1LmrePCwx4R8R51mNsW3ViX6mzJoctVZW1os7k61bwJoz5nOEsiih2VHu4D0bfy5RNqGQbIH8lXV9pKsHGGjm4xCypKPPTAUwxsEQ74} \end{align}  for all $\sifpoierjsodfgupoefasdfgjsdfgjsdfgjsdjflxncvzxnvasdjfaopsruosihjsdkghajsdflahgfsif\in C^{\infty}(\TT^3)$ and all $t\in [0, T]$ for every~$T>0$. Namely, the pair $(u,\tau)$ solves the original SNSE~\eqref{EQ01}.  \par Now, we proceed to show \eqref{EQ15} for~$\uk$. Similarly to the proof of Lemma~\ref{L07}, we apply Lemma~\ref{L02} to the equations \eqref{EQ95}, obtaining   \begin{align}   \begin{split}   &\sadklfjsdfgsdfgsdfgsdfgdsfgsdfgadfasdf\biggl[\sup_{0\leq t\leq T\wedge\tau}\sifpoierjsodfgupoefasdfgjsdfgjsdfgjsdjflxncvzxnvasdjfaopsruosihjsdgghajsdflahgfsif\uk_j(t,\cdot)\sifpoierjsodfgupoefasdfgjsdfgjsdfgjsdjflxncvzxnvasdjfaopsruosihjsdgghajsdflahgfsif_3^3   +\sifpoierjsodfgupoefasdfgjsdfgjsdfgjsdjflxncvzxnvasdjfaopsruosihjsdfghajsdflahgfsif_0^{T\wedge\tau}  \sifpoierjsodfgupoefasdfgjsdfgjsdfgjsdjflxncvzxnvasdjfaopsruosihjsdfghajsdflahgfsif_{\TT^3} | \nabla (|\uk_j(t,x)|^{3/2})|^2 \,dx dt\biggr]   \\&\indeq   \leq C   \sadklfjsdfgsdfgsdfgsdfgdsfgsdfgadfasdf\biggl[   \sifpoierjsodfgupoefasdfgjsdfgjsdfgjsdjflxncvzxnvasdjfaopsruosihjsdgghajsdflahgfsif\uk_{0,j}\sifpoierjsodfgupoefasdfgjsdfgjsdfgjsdjflxncvzxnvasdjfaopsruosihjsdgghajsdflahgfsif_3^3   +   (\bar\epsilon+\epsilon_{\sigma})\sifpoierjsodfgupoefasdfgjsdfgjsdfgjsdjflxncvzxnvasdjfaopsruosihjsdfghajsdflahgfsif_0^{T\wedge\tau}    \sifpoierjsodfgupoefasdfgjsdfgjsdfgjsdjflxncvzxnvasdjfaopsruosihjsdgghajsdflahgfsif \uk\sifpoierjsodfgupoefasdfgjsdfgjsdfgjsdjflxncvzxnvasdjfaopsruosihjsdgghajsdflahgfsif_{9}^{3}   \,dxdt   \biggr]   \commaone j=1,2,3,   \end{split}   \llabel{x7SLjPjQdzHCVuarB6XUve32Z5b1lfFSRyfI0SaOdaYTBXN9qmijDO3az1zo9yKuqnB2UShf0GnWhQEht1eHPAd4a9l5p3Yi81r25uZwwR8G3OIzhTOGYsITG5YwtjFM4r2gLKn87OPhoVn11UEIn32O8sGpRnbNnRIZe0JQ29vldLUrB00Fd3eabLWFPv66hXwjkEcRTo4b0deLWv7rsDc7jiveH75j8F32Vz9lkcNphuH21hOE9LV28bc6f8a6szNZXRpJxDwCkqBrr6ewILJZFhHe8QzMVDJtJewQZJUI551nJ9PeovNSZZalieWJsew3kjdkmKn3Q6GKkfsnKeFlMnycL7fD9IaZ2zgk6dlqd01Z1rwEcbdMRH5I49KqEFRtg0tvkExZCRtMCkHE5snrnPg6Ijvsob8KThiOslyOFYa4mbysgIga7pzCdiXnubntkexBbWpwpMo4k8yxesQnXcG8mEFNWaegqbClcZ5bzZyNjEVtxND8cqbDf7Y7CZmIa140Y9sQMljx4EzBM6U0PHNOBsii1AKPov8wBSPX2VPf8xV3D3TDIWXly4kOf59po0Skn07HCFxP9bjyZBmyHMTfaFdJGqzVFPGNy6poIBElRsWC485AFty6t60NFgbMhZajDei2xvrczUQK9iMPhtiC1JBiBOtS95oFWT7jV0Q25YSeG3jID3SXk9xJIJ1QgYMR2hjzlnmTNtfqC5QI5IwpeqxKG1t8wq4sthwH6nBDBUclkzZWz0B2CqsO4fEQ100}   \end{align} which implies   \begin{align} \begin{split} \sadklfjsdfgsdfgsdfgsdfgdsfgsdfgadfasdf\biggl[ \sup_{0\leq s\leq \tau}\sifpoierjsodfgupoefasdfgjsdfgjsdfgjsdjflxncvzxnvasdjfaopsruosihjsdgghajsdflahgfsif\uk(s,\cdot)\sifpoierjsodfgupoefasdfgjsdfgjsdfgjsdjflxncvzxnvasdjfaopsruosihjsdgghajsdflahgfsif_3^3 +\sifpoierjsodfgupoefasdfgjsdfgjsdfgjsdjflxncvzxnvasdjfaopsruosihjsdfghajsdflahgfsif_0^{\tau}  \sum_{j}    \sifpoierjsodfgupoefasdfgjsdfgjsdfgjsdjflxncvzxnvasdjfaopsruosihjsdfghajsdflahgfsif_{\TT^3} | \nabla (|\uk_j(s,x)|^{3/2})|^2 \,dx ds \biggr] \leq  C\epsilon_0^{3}, \end{split} \label{EQ101} \end{align} provided $\bar\epsilon$ and $\epsilon_{\sigma}$ are sufficiently small. Utilizing the convergence of $\uk$ and Lemma~\ref{L03}, we may pass the limit in \eqref{EQ101} and arrive at~\eqref{EQ15}. \par To prove the uniqueness, we assume that $(u,\eta)$ and $(\tilde u,\tilde\eta)$ are two solutions of the SNSE satisfying \eqref{EQ15}, as in the main statement. The difference $(v,\bar\eta)$, where $v=u-\tilde u$ and $\bar\eta=\eta\wedge\tilde\eta$, solves   \begin{align}    \begin{split}     &\partial_t v - \Delta v       = - \mathcal{P}((u\cdot\nabla) v)         - \mathcal{P}((v\cdot\nabla) \tilde u)         + (\sigma(t,u) -\sigma(t,\tilde u))\dot{W}(t),     \\     &\nabla\cdot v = 0      \\     & v(0)=0.   \end{split}    \llabel{DgXc3OuMGJonsZM2EKIUOKr73Yv7JL9CdyguQqtIUrrhkHnHJ5JigfRrnKyNp0T37DBJiFzVqyPnQSY5QBv3LYZ5LrCdTqazeczlmGVAcx7uExXYu8ik7BIX3DEpaIyldWvAtDxvl70fkOcQ7RreyXDJE1gj3cLSNQTmK6lqOE4SVGxgDJO9muKDvvMaHGeevgMM7lWB15Wq9lFoUxFcPSyjHSWX5kgPup5yvw28JWVZ9At0WjfJuLfhFQd6aE7VnPq7ow2C6hK2loMkgJn03Vz5EteIMoiHWS2Htw4f99HNj7318gaIvJxjakNjunHNNM8kpBy8bp9KEQclHhTNFgu2uEkb6lPKahMNekvnAo0mMoeLXY0spWIDgTgbamnkCTSUvKLj6dtN3xYQr7kycTga6P3EVDPR1kftVZdHMlaCiTwUNJAw1pfZrEtytozvY85OB0FtLYpfKqFkodCG2d5ezkg1rqSaY8EJdLoUYFDMLsaWIITSAphuqtx20vtbdYDk8ReLv274NHf3aKJU0coov2EcMcwomvuDxZkoHa4L8uhvoKCiMy7UzuiN29wfPfG8TFtjOdqlG4M0luRt0lxZM3YA5uYoGkNkBzhbHo4aOEE5rxCUp1ArJjBN1h5cWs1F6qQVapGNgBgx7M5AswfEMyMgPsoVexGPimA4yXPma5dTjJ2BN3F3SMYKxqSegaB769y5PVO8QyJTRBe2vVe0pkNSyOnhkXc9d0clSXlbsqp8tCvKLTo5alzigOORUlEQ102}   \end{align} Proceeding as in the proof of Lemma~\ref{L07}, we obtain an analog of \eqref{EQ61}, which reads   \begin{align}   \begin{split}   &\sadklfjsdfgsdfgsdfgsdfgdsfgsdfgadfasdf\biggl[\sup_{0\leq t\leq \bar\eta}\sifpoierjsodfgupoefasdfgjsdfgjsdfgjsdjflxncvzxnvasdjfaopsruosihjsdgghajsdflahgfsif v(t,\cdot)\sifpoierjsodfgupoefasdfgjsdfgjsdfgjsdjflxncvzxnvasdjfaopsruosihjsdgghajsdflahgfsif_3^3   +\sifpoierjsodfgupoefasdfgjsdfgjsdfgjsdjflxncvzxnvasdjfaopsruosihjsdfghajsdflahgfsif_0^{\bar\eta}\sum_{j=1}^{3} \sifpoierjsodfgupoefasdfgjsdfgjsdfgjsdjflxncvzxnvasdjfaopsruosihjsdfghajsdflahgfsif_{\TT^3} | \nabla (|v_j(t,x)|^{3/2})|^2 \,dx dt\biggr]   \leq C  \sadklfjsdfgsdfgsdfgsdfgdsfgsdfgadfasdf\bigl[  \sifpoierjsodfgupoefasdfgjsdfgjsdfgjsdjflxncvzxnvasdjfaopsruosihjsdgghajsdflahgfsif v(0)\sifpoierjsodfgupoefasdfgjsdfgjsdfgjsdjflxncvzxnvasdjfaopsruosihjsdgghajsdflahgfsif_3^3  \bigr]
  = 0   ,  \end{split}    \llabel{P3qMGvTJARYozYDJLvTUuqMOONYNP82Ps9OiDs17x75SajMoJGBdp885oixszl6oKVEaWAZWwF96XfzlOd9zCCgma4P67OY59PgbGtzVZ4KWPNxAjy0Fp94LWBmVqjqW3vXH4rCz4IuAL6FdNeNtlHCh4NRn6UxnBYSpywkHC2I0Gictzw30lLmQmNTZ8HjuybH1W0zwXr9JEACx0pSDpGuO60FM2QmewNRFhP9fiIm1FeGcDRw2ovhfk2fSx8NUAIwPv4CmFCE8lJz5IeHrsPwiuDaT69r6Mj4keBfBhvgBXOKJdTbjUFNA59p2fDj1WpyNbiYAL9OywosNCiw4tSyNk8avAm4NgApm0XgUy3yAhP3TWptUe3EjdLgM1O9lt32XbpFzK4Avi5WwFdPyHAYdZrfFhrdkYrdfnJAepd29E66Cjsi1JRZ6lJUOAgxAp9MosewoBEV4n6aH75qqkvnXybkmeyYj4TaXPWupqUzKbkpqyo78o4zwXErs8jZMpW6ij9AP4qU6Bofbe06qWG1LY6kbllcrPmrXjpULG6mkqqGOqtbvMOO3tXZlHfb2xBrXkKXCfbTjxRvDbKlDrZ1aDF352Sg5Aav3lzOXMjar4a7qGoYXu3m60XPlQBE1PUxOiviBcglptUEyz5EDs7eDQpxJvDhG8DJzvScH9gTnXUuUftbxSmTtdb33oeQ6aiHUncNe4HecijPLdrnbGzeArrypC5pWBqjIsjCp4x25HWZwo8j2MNbdX2CfvqMUtlEQ93}  \end{align} and the proof of uniqueness is complete.  \end{proof} \par \begin{Remark}\label{R03} If we do not assume the smallness of $\epsilon_\sigma$ in~\eqref{EQ30}, we can still verify that the limit of $\uk$ solves the original SNSE~\eqref{EQ01} up to $\tau\wedge T$ for every~$T>0$. We would obtain  \begin{align} \begin{split} &\sadklfjsdfgsdfgsdfgsdfgdsfgsdfgadfasdf\biggl[\sup_{0\leq t\leq T\wedge\tau}\sifpoierjsodfgupoefasdfgjsdfgjsdfgjsdjflxncvzxnvasdjfaopsruosihjsdgghajsdflahgfsif\uk_j(t,\cdot)\sifpoierjsodfgupoefasdfgjsdfgjsdfgjsdjflxncvzxnvasdjfaopsruosihjsdgghajsdflahgfsif_3^3 +\sifpoierjsodfgupoefasdfgjsdfgjsdfgjsdjflxncvzxnvasdjfaopsruosihjsdfghajsdflahgfsif_0^{\tau\wedge T}  \sifpoierjsodfgupoefasdfgjsdfgjsdfgjsdjflxncvzxnvasdjfaopsruosihjsdfghajsdflahgfsif_{\TT^3} | \nabla (|\uk_j(t,x)|^{3/2})|^2 \,dx dt\biggr] \\&\indeq \leq C \sadklfjsdfgsdfgsdfgsdfgdsfgsdfgadfasdf\biggl[ \sifpoierjsodfgupoefasdfgjsdfgjsdfgjsdjflxncvzxnvasdjfaopsruosihjsdgghajsdflahgfsif\uk_{0,j}\sifpoierjsodfgupoefasdfgjsdfgjsdfgjsdjflxncvzxnvasdjfaopsruosihjsdgghajsdflahgfsif_3^3 + \bar\epsilon\sifpoierjsodfgupoefasdfgjsdfgjsdfgjsdjflxncvzxnvasdjfaopsruosihjsdfghajsdflahgfsif_0^{\tau\wedge T}  \sifpoierjsodfgupoefasdfgjsdfgjsdfgjsdjflxncvzxnvasdjfaopsruosihjsdgghajsdflahgfsif \uk\sifpoierjsodfgupoefasdfgjsdfgjsdfgjsdjflxncvzxnvasdjfaopsruosihjsdgghajsdflahgfsif_{9}^{3} \,dt +\epsilon_{\sigma}^2 \sifpoierjsodfgupoefasdfgjsdfgjsdfgjsdjflxncvzxnvasdjfaopsruosihjsdfghajsdflahgfsif_0^{T} \sifpoierjsodfgupoefasdfgjsdfgjsdfgjsdjflxncvzxnvasdjfaopsruosihjsdgghajsdflahgfsif \uk\sifpoierjsodfgupoefasdfgjsdfgjsdfgjsdjflxncvzxnvasdjfaopsruosihjsdgghajsdflahgfsif_{3}^{3} \,dt \biggr] \commaone j=1,2,3, \end{split}    \llabel{0eT39LzGoZGUXcC1cxRKQmrXFI2A5XLOpxHgNP1vXnZ35YTT3VGgG0Y6loTkgX0LjtDCaBSFgBKQvpoDC9HobyECQkmxBrJyBuZ3W39tu4Dea8xsYqE4OxT15FIkmYw0KTTYcNjy8VvmxWdyr6uNTV13OrOhZwAeZavnDhcEmBSweWlJps7N31mM2jXy92SdSas7K9c3MBcWyUSHdIXC9V5oGtzaEY8nMBJ5DNZHwBDOyWamRPVYnoV2qlhDcqpB8tHmM0e4ja5ptjDRAUaAgTEbYwBr4rapjLtCHQORU6i4GrYkFzx91EBrqRGhmlSu3RKzh489F6L6b4bxcdxYhlbk26D4DPDjmcj3Gm6MqaXGeNphBOSVYr0mOQs30RQc5wEge7aeVYf4gvJBA2eWoh292cHSKXVNgRN8T0oSND87zYU63epzrMIO2auIWoZe6SOIILTpn4dMOk7XtfvenMlUF0r5GvYIauSAak20YFPB9jOxRu8zz3Dq32eq6iUeKOXgJ0zCW18zX2n0RhKCtqtZJJx2v2Polu2Sxjglz9aeKoC1db0OnKaouVpCm4lvK41H5Gev2sruVbkFYDjm7MwfiT1blqc93lFjHNwUtfPPNvyGdZOkimHRHTNNNzkw71lIGsYTSWR3ulXvokSf1KE3N9QLhfH7OEAELTre0BXLhex4qypTkhHsZBVnrgnHatAiEJv6Mz3X5fP5iiBI7m9q5VzUODo9YgA9uFrhGKOYZH4NNngGYwPYwLPrtmj63dEQ66} \end{align}	 for every $\uk$, which leads to   \begin{align} \begin{split} \sadklfjsdfgsdfgsdfgsdfgdsfgsdfgadfasdf\biggl[ \sup_{0\leq s\leq \tau\wedge T}\sifpoierjsodfgupoefasdfgjsdfgjsdfgjsdjflxncvzxnvasdjfaopsruosihjsdgghajsdflahgfsif\uk(s,\cdot)\sifpoierjsodfgupoefasdfgjsdfgjsdfgjsdjflxncvzxnvasdjfaopsruosihjsdgghajsdflahgfsif_3^3 +\sifpoierjsodfgupoefasdfgjsdfgjsdfgjsdjflxncvzxnvasdjfaopsruosihjsdfghajsdflahgfsif_0^{\tau\wedge T}  \sum_{j}    \sifpoierjsodfgupoefasdfgjsdfgjsdfgjsdjflxncvzxnvasdjfaopsruosihjsdfghajsdflahgfsif_{\TT^3} | \nabla (|\uk_j(s,x)|^{3/2})|^2 \,dx ds \biggr] \leq  C\epsilon_0^{3}, \end{split}    \llabel{kYx01d8zwbroon3LtCqpSeeQHTdVyRdz2GIxyjxKPpLa43Qqkz9bJVgNuivcKNtiU4ujWunorFDIcKzhicQZNU6Pz2uuRdQuKOmeh3mgffnVsf6MJMjBMsQhPH3o32OSx2MDEZMH5HnpWJibL5OFkDOUwiRxT3v06s2BYw4IJjX1zNhb7vTWSRff0ksZRBBODc5bCv1kW9ImjwmTRJwtKdlnqPn05tygZWviIEgd8s1MtAoe0mEjEIz3EafdFHNRWD785FlERfwQnJKGnjmzfVU24NgtQrxrOBtTKwBpqF7DJJJAJCF81h2eRhZvuStrU9VrPvjDzscPPw3rRiPiJgkT7MfVPTQHq12JwvTwx4cNeEegDTNyQWJX10wYClRDLGm1xe7jcIASC2Jgx3cfRzO9Ig6MxcdNRwZL0nyw5ufePnsNP09sz5ZIWrMQ3uh0NiF92ae0vXnX05RreQzCA5jvDdHXnSHmkTwKQaSFKKKQC5JeHMTGvF4peEMRxWeVYXcqNHhqFBvBuEXPxCf80MNsv3RRJVhcMyUPRU89oE0z7TCgQFf4WPet8YMqfaL8BmyuoDrCSc5yZ5DJj9icvgYu0SwkO9exAq2B2SIi96GD0UY9M6YcTlOvAjeT1AVE9R01xsVSOKOLxXtRzLb4LO1iXmyBmOTsiN6yYd9By2JhhZs7aek4dsjUZgANw1V9Q3EQHil8oz1jQVr4AbJQe7CPnRsumj5y1JvKb3ZeVPOx39tp3IaPyxezp10fAv5r50EQ67} \end{align} with the constant $C$ depending on $T$. Utilizing the convergence of $\uk$, we may pass to the limit and arrive at~\eqref{EQ15-2}. The pathwise uniqueness is first established on $[0, \tau\wedge T]$ for a sufficiently small $T>0$, and then on $[0, \tau\wedge T]$ for any $T>0$ extending by a finitely many steps. \end{Remark} \par \section*{Acknowledgments} \rm IK was supported in part by the NSF grant DMS-2205493.  \ifnum\sketches=1 \fi \end{document}